\documentclass[a4paper,10pt,twoside]{amsart}
\usepackage[affil-it]{authblk}
\usepackage{titling}
\usepackage[latin1]{inputenc}
\usepackage{calligra}
\usepackage[OT1]{fontenc}
\usepackage[english]{babel}
\usepackage{amsfonts}
\usepackage{amsmath}
\usepackage{amssymb}
\usepackage{amsthm}
\usepackage{dsfont}
\usepackage{faktor}
\usepackage{float}
\usepackage{enumerate}
\usepackage{color}
\usepackage{pdfpages}
\usepackage{esint}
\usepackage{hyperref}
\usepackage{cite}
\usepackage{fancyhdr}
\usepackage{mathtools}
\usepackage{mathrsfs}
\usepackage{verbatim}

\usepackage{etoolbox}
\patchcmd{\section}{\scshape}{\bfseries}{}{}
\makeatletter
\renewcommand{\@secnumfont}{\bfseries}
\makeatother

\newcommand\testname{Abstract}
\makeatletter
\newenvironment{abs}{%
    \small
    \begin{center}%
        {\textbf \testname\vspace{-.2em}\vspace{\z@}}%
    \end{center}%
    \quote
    }
   {\endquote}
\makeatother

\DeclareMathOperator*{\Id}{Id}

\newcommand{\Div}{\mathrm{div}}

\newcommand{\D}{\mathcal{\wtilde D}}

\newcommand{\dd}{\mathrm{d}}

\newcommand{\Z}{\mathcal{Z}}

\newcommand{\J}{\mathcal{J}}
\newcommand{\I}{\mathcal{I}}

\newcommand{\EE}{\mathcal{E}}
\newcommand{\Ff}{\mathscr{F}}

\newcommand{\RR}{\mathbb{R}}

\theoremstyle{definition}
\newtheorem{defin}{Definition}[section]

\newtheorem{rem}[defin]{Remark}

\theoremstyle{plane}
\newtheorem{thm}[defin]{Theorem}
\newtheorem{prop}[defin]{Proposition}
\newtheorem{coroll}[defin]{Corollary}
\newtheorem{lemma}[defin]{Lemma}

\newcommand{\mbb}{\mathbb}

\newcommand{\mc}{\mathcal}

\newcommand{\veps}{\varepsilon}
\newcommand{\what}{\widehat}
\newcommand{\wtilde}{\widetilde}
\newcommand{\vphi}{\varphi}

\newcommand{\ra}{\rightarrow}

\newcommand{\g}{\gamma}
\newcommand{\s}{\sigma}
\newcommand{\z}{\zeta}
\newcommand{\de}{\delta}
\newcommand{\lan}{\langle}
\newcommand{\ran}{\rangle}

\newcommand{\R}{\mathbb{R}}

\newcommand{\N}{\mathbb{N}}
\newcommand{\T}{\mathbb{T}}

\renewcommand{\div}{{\rm div}\,}

\allowdisplaybreaks

\def\d{\partial}
\def\div{{\rm div}\,}
\def\rot{{\rm rot}\,}

\title{\textsc{\Large{\textbf{Global well-posedness and long-time dynamics  \\ \vspace{.1cm}
for a higher order Quasi-Geostrophic type equation 
}}}\vspace{.3cm} }
	
\author{\textsl{Francesco De Anna$\,^1\quad$}}
\author{\textsl{$\quad$ Francesco Fanelli$\,^2$} \vspace{.3cm}}
\affil{	\small\textsc{$\,^1\,$ Penn State University} \\ 
		\small\textit{Department of Mathematics}\\ \vspace{0.1cm}
		\small\texttt{fzd16@psu.edu}}
\affil{\small \textsc{$\,^2\,$ Universit\'e de Lyon, Universit\'e Claude Bernard -- Lyon 1} \\ 
		\small\textit{Institut Camille Jordan, UMR CNRS 5208}\\ \vspace{0.1cm}
		\small\texttt{fanelli@math.univ-lyon1.fr} \vspace{.5cm}}
\date{\today}

\markleft{ Francesco De Anna $\;$ \& $\;$ Francesco Fanelli}
\begin{document}
\maketitle

\vspace{.3cm}
\begin{abs}
In this paper we study a higher order viscous quasi-geostrophic type equation. This equation was derived in \cite{F_2016_MA} as the limit dynamics of
a singularly perturbed Navier-Stokes-Korteweg system with Coriolis force, when the Mach, Rossby and Weber numbers go to zero at the same rate.

The scope of the present paper is twofold. First of all, we investigate well-posedness of such a model on the whole space $\R^2$: we prove that it is well-posed in $H^s$ for any $s\geq3$,
globally in time. Interestingly enough, we show that this equation owns two levels of energy estimates, for which one gets existence and uniqueness of weak
solutions with different regularities (namely, $H^3$ and $H^4$ regularities); this fact can be viewed as a remainder of the so called BD-entropy structure of the original system. 

In the second part of the paper we investigate the long-time behaviour of these solutions. We show that they converge to the solution of the corresponding linear parabolic type equation,
with same initial datum and external force. Our proof is based on dispersive estimates both for the solutions to the linear and non-linear problems.
\end{abs}

\subsubsection*{\textbf{2010 Mathematics Subject Classification}:}{\small 35Q35 
(primary); 
35K25, 
35B65, 
35B40, 
35Q86 
(secondary).}

\subsubsection*{\textbf{Keywords}:}{\small Quasi-geostrophic equation; energy estimates; BD-entropy structure; global well-posedness; long-time behaviour; decay estimates.}

\section{Introduction}

In this paper we are concerned with well-posedness and log-time dynamics issues for the non-linear parabolic-type equation
\begin{equation}\label{intro_eq:main}
		\partial_t \left(\Id-\Delta+\Delta^2\right)r\,+\,\nabla^{\perp}(\Id-\Delta)r\cdot\nabla\Delta^2 r\,+\,\mu\,\Delta^2(\Id-\Delta)r\,=\,f\,,
\end{equation}
where $r$ and $f$ are functions of $(t,x)\in\R_+\times\R^2$. The parameter $\mu>0$ will be kept fixed throughout all the paper;
it can be interpreted as a sort of viscosity coefficient. We supplement equation \eqref{intro_eq:main} with the initial condition $r_{|t=0}\,=\,r_0$, where $r_0$ is a suitably smooth function defined
over $\R^2$.

\subsection{Derivation of the model} \label{ss:i-model}

Equation \eqref{intro_eq:main} was derived in \cite{F_2016_MA} as the equation describing the limit dynamics, for $\veps\ra0$, of the following singular perturbation problem:
\begin{equation} \label{intro_eq:NSK+rot}
\begin{cases}
\d_t\rho_\veps+\div\left(\rho_\veps u_\veps\right)\,=\,0 \\[2ex]
\d_t\left(\rho_\veps u_\veps\right)+\div\bigl(\rho_\veps u_\veps\otimes u_\veps\bigr)+\dfrac{1}{\veps^2}\,\nabla P(\rho_\veps)+
\dfrac{1}{\veps}\,e^3\times\rho_\veps u_\veps-\nu\div\bigl(\rho_\veps Du_\veps\bigr)-
\dfrac{1}{\veps^{2}}\,\rho_\veps \nabla\Delta\rho_\veps\,=\,0\,.
\end{cases}
\end{equation}
The previous equations are posed on $\R^2\times\,]0,1[\,$ and supplemented by conplete slip boundary conditions, which allow to avoid boundary layers effects.

System \eqref{intro_eq:NSK+rot} is the so-called Navier-Stokes-Korteweg system; it describes the dynamics of a compressible viscous fluid, whose motion is mainly influenced
by internal tension forces and Earth rotation. Here above, at each value of $\veps$ fixed, the scalar function $\rho_\veps=\rho_\veps(t,x)\geq0$ represents the density of the fluid,
$u_\veps=u_\veps(t,x)\in\R^3$ its velocity field and the function $P(\rho_\veps)$ its pressure. The number $\nu>0$ is the viscosity coefficient; the viscous stress tensor is supposed
to depend on (and possibly degenerate with) the density. Finally, the term $\rho_\veps\, \nabla\Delta\rho_\veps$ is the capillarity tensor, which takes into account the effects of
a strong surface tension, while the term $e^3\times\rho_\veps\,u_\veps\,=\,\rho_\veps\,\bigl(-u^2_\veps,u^1_\veps,0\bigr)$ is the Coriolis operator, which takes into account effects due
to the fast rotation of the ambient space.
We refer e.g. to \cite{B-D-L}, \cite{F_2016_sub} and references therein for more details on the previous model.

The scaling introduced in \eqref{intro_eq:NSK+rot} corresponds to taking the Mach number \textit{Ma}, the Rossby number \textit{Ro} and the Weber number \textit{We} to be all proportional
to the small parameter $\veps$. In turn, this means that we are studying the incompressible, fast rotation and strong capillarity limits at the same time, focusing our attention on
their mutual interaction.
See again \cite{F_2016_sub} and the references therein for additional comments about the adimensionalisation of the equations and for more insights on this scaling.

A similar asymptotic analysis was performed e.g. in \cite{F-G-N} for the classical barotropic Navier-Stokes equations (no capillarity forces were taken into account), and, for models
much closer to \eqref{intro_eq:NSK+rot}, in e.g. \cite{B-D_2003} and \cite{J-L-W}. Notice however that, in these last two references, the low Mach and low Rossby numbers limit
was coupled with a vanishing capillarity limit, which corresponds to take \textit{We} of a smaller order $\veps^{1-\alpha}$, for some $\alpha\in\,]0,1]$: then the authors
found that the limit dynamics was characterized by a quasi-geostrophic equation
\begin{equation} \label{intro_eq:q-geo}
\d_t\bigl(r\,-\,\Delta r\bigr)\,+\,\nabla^\perp r\,\cdot\,\nabla \Delta r\,+\,\frac{\nu}{2}\,\Delta^2r\,=\,0
\end{equation}
over $\R^2$. We will say something more about this system in Subsection \ref{ss:i-work} below. For the time being, let us just mention that this equation is well-known
in physical theories for geophysical flows (see e.g. \cite{Ped}, \cite{Z}) as an approximate model when considering the limit of fast Earth rotation.

Let us comment on the fact that, although the original problem \eqref{intro_eq:NSK+rot} is three-dimensional, the limit system \eqref{intro_eq:main}, or its analogue \eqref{intro_eq:q-geo}
in absence of capillarity effects, becomes $2$-D. This issue is not surprising at all, since it is an expression of the celebrated Taylor-Proudman theorem
for geophysical flows (see e.g. \cite{Ped}). More precisely, it is well-known that fast rotation has a stabilizing effect on the fluid motion, in the sense that it introduces some vertical rigidity:
then, the dynamics tends to be purely horizontal, meaning that it takes place on planes which are hortogonal to the rotation axis.

As a final comment, we want to highlight the strong analogy linking equations \eqref{intro_eq:main} and \eqref{intro_eq:q-geo}: they both share the same structure, although in \eqref{intro_eq:main}
we can notice the presence of higher order derivatives and, more importantly, of the operator $\Id-\Delta$ in some terms.
We remark that, in light of the study carried out in \cite{F_2016_MA}, the presence of higher order derivatives in \eqref{intro_eq:main} has to be interpreted as a remainder
of the action of strong capillarity forces, which persist also after passing to the limit for $\veps\ra0$ in \eqref{intro_eq:NSK+rot}.
We then expect this equation to be somehow more pertinent in physical approximations when describing non-homogeneous flows with strong internal tension forces, or whenever one wants to keep track of
potential energy due to steep changes of density in small regions (like in diffuse interface theories and propagation of interfaces, for instance).

We also point out that the $\d_tr$ term, appearing in both equations \eqref{intro_eq:main} and \eqref{intro_eq:q-geo}, is a remainder of the balance between pressure forces and
Coriolis effects. Technically, it arises from the singular limit problem, when using the mass equation to get rid of the singular behaviour of the rotation term.

\subsection{Related works, and content of the paper} \label{ss:i-work}

To the best of our knowledge, the derivation of \eqref{intro_eq:NSK+rot} given in \cite{F_2016_MA} is completely new, and we are not aware of previous studies carried out
on this equation. 
As already pointed out, it is fair to mention that it shares strong similarities with the quasi-geostrophic equation \eqref{intro_eq:q-geo}, which has been esxtensively studied so far, also from the mathematical point of view.
See e.g. books \cite{Maj} and \cite{Maj-Wang}, and paper \cite{No-Va} for interesting recent advances.
We remark here also the analogy of our equation with equations for second-grade fluids, and especially with the so-called Leray-$\alpha$ and Euler-$\alpha$ equations,
which are often adopted as sub-grid scale models for turbulence. We refer the interested reader to e.g. \cite{Far-L-T} and \cite{B-I-LF-NL} and the references quoted therein.

However, it is important to highlight some apparently small differences which exist between equations \eqref{intro_eq:main} and \eqref{intro_eq:q-geo}, and which make the analysis carried out
in the present paper not to be an obvious adaptation of what is known for the quasi-geostrophic equation. Apart from the higher order operators involved in the former model,
the probabily most relevant difference is represented by the appearing of the operator $\Id-\Delta$: in particular, in \eqref{intro_eq:main} we lose the
stream-function relation linking, in the convective term, the transport velocity field to the transported quantity. More precisely, if we set $v\,:=\,\nabla^\perp r$ in \eqref{intro_eq:q-geo},
we notice that $\rot v\,=\,\Delta r$, which is exactly the transported term; on the contrary, in \eqref{intro_eq:main}, if we set $u\,:=\,\nabla^\perp(\Id-\Delta)$, we get that 
$\rot u$ is different from both $\Delta^2r$ and $(\Id-\Delta+\Delta^2)r$.

This fact slightly complicates the structure of our equation. As a matter of fact, as a consequence of the presence of this operator $\Id-\Delta$, simple energy estimates
do not work anymore: after multiplying \eqref{intro_eq:main} by $r$ and integrating by parts, it seems impossible to get rid of a fourth-order derivative occurring in the
convective term, which cannot be absorbed since one disposes at most of three derivatives for $r$ (thanks to the smoothing effect provided by the viscosity term). We refer also
to the beginning of Subsection \ref{ss:energy} for further comments about this point.
This problem forces us to test the equation rather against $(\Id-\Delta)r$:
an easy inspection of the structure of the convective term shows that it thus identically vanishes, and this is the key to get first-order energy estimates. The reason why they are called
``first-order'' will be manifest in a while; for the time being, let us point out that, in this way, one gets bounds for $r$ in $L^\infty_T(H^3)\,\cap\,L^2_T(H^4)$. For the reason expressed above,
such a regularity seems to be the minimal one which is required for bluiding up a theory of weak solutions for equation \eqref{intro_eq:main}.
Indeed, thanks exactly to this energy estimate of first kind, we are able to prove existence of weak solutions to \eqref{intro_eq:main}, which
are weak solutions \textsl{\`a la Leray} (see the masterwork \cite{Leray} about the homogeneous incompressible Navier-Stokes equations), since they possess finite
energy. But this is not all: taking advantage of the fact that the space dimension is $d=2$, by the study of the associated parabolic equation (which can be viewed as the analogue of the time-dependent
Stokes problem in our context) we are able to show that weak solutions are in fact unique, and that they actually verify an energy equality.

Let us go further: after observing once again the particular form of the convective term, it is not hard to convince oneself that this term identically vanishes also when tested against the quantity
$\mbb Dr\,:=\,(\Id-\Delta+\Delta^2)r$. Therefore, if one multiply the equation by $\mbb Dr$, one finds a second energy conservation, which gives propagation of $L^\infty_T(H^4)\,\cap\,L^2_T(H^5)$
regularity (provided the initial datum and the external force are smooth enough). This is a remarkable property of our system, which can be viewed as a remainder of
the \emph{BD-entropy structure} owned by the primitive system \eqref{intro_eq:NSK+rot}, see e.g. papers \cite{B-D_2003} and \cite{B-D-L}.
Furthermore, the previous cancellations in the convective term and the second-order energy conservation prompt us to look also for propagation of intermediate and higher regularities,
and indeed we are able to prove existence and uniqueness of solutions at the $H^s$ level of regularity, for any $s\geq 3$.

The proof of higher regularity energy estimates  (namely, for $s>4$) relies on a paralinearization of the convective term and a special decomposition for treating it, which has already been used in \cite{DeA_2017}:
in particular, thanks to this approach we are able to reproduce the special cancellations in the transport operator, up to some remainders; now, a careful analysis of these remainder terms
allows us to control them by the $H^{s}$-energy of the solution $r$, so that one can conclude by an application of the Gronwall lemma.
On the contrary, propagation of intermediate regularities (namely for $3<s<4$) is surprisingly more involved, since now the special cancellations concern only the lower order item in the
convective term, and no more $\Delta^2r$, which hence needs to be controlled very carefully. This can be done by resorting once again to a paralinearization of the transport term
and to delicate estimates concerninig the commutators involved in the computations; eventually, we manage to bound all the terms, and Gronwall lemma allows us to close the estimates
as before.

\medbreak
The well-posedness having been established, globally in time, we pass to investigate the log-time dynamics. Not too surprisingly, we show the convergence of solutions to \eqref{intro_eq:main}
to solutions of the related linear parabolic equation with same bulk force and initial datum, namely
$$
\begin{cases}
\partial_t \left(\Id-\Delta+\Delta^2\right)w\,+\,\mu\,\Delta^2(\Id-\Delta)w\,=\,f \\[1ex]
w_{|t=0}\,=\,r_0\,.
\end{cases}
$$

For proving this result, we follow an approach initiated in \cite{Scho_1985} for the Navier-Stokes equations (see also \cite{Scho_1986} and \cite{A-B-S}), and adapted in \cite{C-Wu}
for the classical quasi-geostrophic equation (roughly, the term $\d_tr$ in \eqref{intro_eq:q-geo} is missing). In our case, some complications appear at the technical level:
once again, they come from the higher order differential operators involved in the computations, and by somehow the non-homogeneity of the operator $\Id-\Delta$.
The latter point entails the presence of remainder terms in the estimates, which require some little effort to be absorbed; the former point is more deep, and we are going to comment it
in a while.

The approach we adopt is the following: first of all we establish decay properties for the solution $w$ of the linear equation. Even though the operator is parabolic, we have to notice
that its symbol (or better, the elliptic part of its symbol) vanishes at order $4$ close to $0$; so we expect to recover a worst decay than for e.g. the classical heat equation.
However, since we are interested in first-type energy estimates, it is enough for us to bound the $L^\infty$ norm of higher order derivatives of $w$: this is a key point, since for them we can establish
a faster decay (namely, like $t^{-1/2}$ rather than $t^{-1/4}$). Nonetheless, this decay reveals to be not enough for our scopes, so that we need to find dispersion properties also for the solution
$r$ to the original equation \eqref{intro_eq:main}: this can be done with some more work, looking at both first- and second-type energy estimates.
Finally, as a last step, we take advantage of these properties to establish decay for the difference $z\,:=\,r-w$ in the $H^3$ norm.

\subsection{Organization of the paper} \label{ss:i-org}
Before going on, let us give an overview of the paper. In the next section we present our main assumptions and state our main results. In Section \ref{s:math}
we discuss some mathematical properties of equation \eqref{intro_eq:main}: namely, we study in detail the non-linear term (transport operator) and
we present the basic energy estimates, of first and second type, as mentioned above. Section \ref{s:weak} is devoted to the theory of weak solutions related to minimal (namely $H^3$)
regularity initial data; in Section \ref{s:strong} we show propagation of higher regularities. Finally, in Section \ref{s:long-time} we study the long-time dynamics.
We collect in Appendix \ref{app:LP} some tools from Littlewood-Paley theory which are needed in the course of our study.

\subsection*{Acknowledgements}

The work of the second author has been partially supported by the LABEX MILYON (ANR-10-LABX-0070) of Universit\'e de Lyon, within the program ``Investissement d'Avenir''
(ANR-11-IDEX-0007), and by the project BORDS, both operated by the French National Research Agency (ANR).

\section{Well-posedness results} \label{s:results}

In $\R_+\times\R^2$, let us consider the scalar equation
\begin{equation}\label{eq:main}
	\begin{cases}
		\partial_t \left(\Id-\Delta+\Delta^2\right)r\,+\,\nabla^{\perp}(\Id-\Delta)r\cdot\nabla\Delta^2 r\,+\,\mu\,\Delta^2(\Id-\Delta)r\,=\,f 	\\[1ex]
		r_{|t=0}\,=\,r_0\,,
	\end{cases}
\end{equation}
where $r$ and $f$ are functions of $(t,x)$ and $r_0$ is a function defined on the whole $\R^2$. The parameter $\mu>0$ is fixed, and it can be interpreted as a sort of viscosity coefficient.

We are interested in both weak and strong solutions theory for the previous equation.
Let us start by considering the former framework, and more precisely by giving
the definition of weak solution which is relevant for us. We point out that our analysis relies basically on energy estimates, so that we will look for weak solutions
\textsl{\`a la Leray}.

\begin{defin} \label{d:weak}
Let $r_0$ belong to $H^3(\R^2)$ and $f\in L^2_{\rm loc}\bigl(\R_+;H^{-2}(\R^2)\bigr)$. Then $r$ is a \emph{weak solution} to equation \eqref{eq:main}
on $[0,T[\,\times\R^2$, supplemented with initial datum $r_0$ and external force $f$, if
$$
r\,\in\,L^\infty\bigl([0,T[\,;H^3(\R^2)\bigr)\,\cap\,L^2\bigl([0,T[\,;H^4(\R^2)\bigr)
$$
and it solves the equation in the weak sense: for any $\phi\in\mc{C}_0^\infty\bigl([0,T[\,\times\R^2\bigr)$ one has
\begin{eqnarray*}
& & \hspace{-0.7cm}
-\int^T_0\!\!\!\int_{\R^2}\bigl(\Id-\Delta+\Delta^2\bigr)r\;\d_t\phi\,dx\,dt\,-\,
\int^T_0\!\!\!\int_{\R^2}\Delta^2r\;\nabla^\perp(\Id-\Delta)r\cdot\nabla\phi\,dx\,dt\,+ \\
& & \hspace{-0.3cm}
+\,\mu\int^T_0\!\!\!\int_{\R^2}\Delta(\Id-\Delta)r\,\Delta\phi\,dx\,dt\,=\,\int^T_0\!\!\!\langle f(t),\phi(t)\rangle_{H^{-2}\times H^2}\,dt\,+\,
\int_{\R^2}\Bigl((\Id-\Delta)r_0\,\phi\,-\,\nabla\Delta r_0\cdot\nabla\phi\Bigr)dx\,,
\end{eqnarray*}
where we have denoted by $\langle\cdot\,,\,\cdot\rangle_{H^{-s}\times H^s}$ the duality pair of $H^{-s}\times H^s$, for any $s>0$.

The solution is said to be \emph{global} if the previous properties are satisfied for all fixed $T>0$.
\end{defin}

\begin{rem} \label{r:initial}
The requirement $r_0\in H^3$ may look to be ``too much'' for a weak solutions theory. Nonetheless, this is somehow the natural regularity for the initial datum, because it is
imposed by the singular perturbation problem from which our model derives. In addition, it also seems to us the minimal smoothness which is needed to get
basic energy estimates (see also Subsection \ref{ss:energy} below with respect to this point).
\end{rem}

%

We now state our first main result. It asserts the \emph{existence and uniqueness} of weak solutions to our system, for any given initial datum and external force.
\begin{thm} \label{t:weak}
For all initial datum $r_0\in H^3(\R^2)$ and all external force $f\in L^2_{\rm loc}\bigl(\R_+;H^{-2}(\R^2)\bigr)$,
there exists a unique global in time weak solution $r$ to equation \eqref{eq:main}, such that
$$
r\,\in\,\mc C\bigl(\R_+;H^3(\R^2)\bigr)\,\cap\,L^\infty_{\rm loc}\bigl(\R_+;H^3(\R^2)\bigr)\,\cap\,L^2_{\rm loc}\bigl(\R_+;H^4(\R^2)\bigr)\,.
$$
Moreover, for any $T>0$ fixed, $r$ satisfies the following energy equality, for all $t\in[0,T]$:
\begin{eqnarray*}
& & \hspace{-1cm} 
\mc E[r(t)]\,+\,\mu\int^t_0\left(\|\Delta r(\tau)\|^2_{L^2}+2\|\nabla\Delta r(\tau)\|^2_{L^2}+\|\Delta^2r(\tau)\|^2_{L^2}\right)\,d\tau\,= \\
& & \qquad\qquad\qquad\qquad\qquad\qquad\qquad\qquad
=\,\mc E[r_0]\,+\,\int^t_0\lan f(\tau),(\Id-\Delta)r(\tau)\ran_{H^{-2}\times H^2}\,d\tau\,,
\end{eqnarray*}
where, for all functions $\vphi\in H^3$, we have defined $\mc E[\vphi]$ to be the quantity
$$
\mc E[\vphi]\,:=\,\bigl(\|\vphi\|^2_{L^2}\,+\,2\|\nabla \vphi\|^2_{L^2}\,+\,2\|\Delta \vphi\|^2_{L^2}\,+\,\|\nabla\Delta \vphi\|^2_{L^2}\bigr)/2\,.
$$
\end{thm}

We postpone to Subsection \ref{ss:h-4} (see Theorem \ref{t:weak_2} therein) a similar statement, for slightly more regular initial data and external forces. That statement corresponds to
energy estimates of second kind, as pointed out in the Introduction.

\medbreak
Let us now devote attention to more regular solutions, for which we are able to establish existence and uniqueness, globally in time. More precisely, we have the following general statement.
\begin{thm} \label{t:higher-reg}
Let $s>0$, and assume the initial datum $r_0$ to belong to $H^{3+s}(\RR^2)$ and the external force $f\in L^2_{\rm loc}\bigl(\R_+;H^{s-2}(\R^2)\bigr)$ .
Then system \eqref{eq:main} admits a unique global in time solution
\begin{equation*}
	r\,\in\,\mc C\bigl(\RR_+;H^{3+s}(\RR^2)\bigr)\,\cap\,L^\infty_{\rm loc}\bigl(\R_+;H^{3+s}(\R^2)\bigr)\,\cap\,L^2_{\rm loc}\bigl(\RR_+; H^{4+s}(\RR^2)\bigr)\,.
\end{equation*}
Moreover, there exist two positive constants $C_1$ and $C_2$ such that
\begin{equation} \label{intro_reg_ineq}  
	\| r(t)\|_{H^{3+s}}^2 + 
	\int_0^t\| \nabla \Delta^2 r(\tau) \|_{H^s}^2\dd \tau\,\leq\, C_1
	\Big(	
		\|r_0\|_{H^{3+s}}^2 + 
		\frac{1}{\mu}\int_0^t \| f(\tau)\|^2_{H^{s-2}}\dd \tau
	\Big)
	\exp
	\left\{
		\frac{C_2}{\mu}\,F(t)  
	\right\}\,, \nonumber
\end{equation} 
is fulfilled for any time $t\in\RR_+$, where we have defined the function
$$
F(t)\,:=\,\left\{\begin{array}{lcl}
                 e^{2\,\mu\,t}\,\Big(\| r_0\|_{H^3}^2\, +\, \frac{1}{\mu}\, \int_0^t \| f (\tau) \|_{H^{-2}}^2\,d\tau\Big)^2 & \qquad\mbox{ if } & 0<s<1 \\[2ex]
                 \mu^2\,t & \qquad\mbox{ if } & s=1 \\[2ex]
                 e^{\mu\,t}\,\Big(\| r_0\|_{H^3}^2\, +\, \frac{1}{\mu}\, \int_0^t \| f (\tau) \|_{H^{-2}}^2\,d\tau\Big) & \qquad\mbox{ if } & s>1\,.
                 \end{array}
\right.
$$
\end{thm}

We postpone to Section \ref{s:strong} the precise statements, with the previse energy estimates, concerning the various cases $s=1$ (see Subsection \ref{ss:h-4}), $s>1$
(treated in Subsection \ref{ss:higher}) and $0<s<1$ (see Subsection \ref{ss:intermediate}).



\section{Mathematical structure of the equation} \label{s:math}

In this section we show some mathematical properties of equation \eqref{eq:main}. The main goal is to establish energy estimates for smooth solutions: this will be done in
Subsection \ref{ss:energy}.
It is interesting to notice that we can identify two kinds of energy estimates: this fact seems to be a sort of remainder of the BD entropy structure of the original
Navier-Stokes-Korteweg system.

Energy estimates will be the key to the proof of our main theorems. However, first of all we need to study the structure of the bilinear term: the next subsection
is devoted to this.

\subsection{Analysis of the bilinear term} \label{ss:bilinear}

We start by studying the bilinear term in \eqref{eq:main}, namely the operator $\Lambda$ formally defined by the formula
\begin{equation} \label{eq:bilin}
\Lambda(\rho,\z)\,:=\,\div\left(\nabla^\perp(\Id-\Delta)\rho\;\;\Delta^2\z\right)\,.
\end{equation}
We collect in the next statement some continuity properties and useful formulas related to it.

\begin{lemma} \label{l:bilin}
The bilinear operator $\Lambda$, defined in \eqref{eq:bilin} above, acts continuously on the following spaces:
\begin{enumerate}[(a)]
 \item from $H^3\times H^4$ into $H^{-(2+s)}$, for all $s>0$ arbitrarily small;
 \item from $H^{3+s}\times H^4$ into $H^{-(2-s)}$, for all $0<s<1$;
 \item from $H^4\times H^4$ into $H^{-(1+s)}$, for all $s>0$ arbitrarily small; 
\item  for all $s>0$, from $H^{4+s}\times H^4$ into $H^{-1}$.
\end{enumerate}

Moreover, given $s\in\,]0,1[\,$, for any $(\rho,\z)\in H^{3+s}\times H^4$ and any $\phi\in H^{2-s}$, one has the identity
$$
\langle \Lambda(\rho,\z)\,,\,\phi\rangle_{H^{-(2-s)}\times H^{2-s}}\,=\,-\,\langle \Lambda\bigl(\wtilde{\phi},\z\bigr)\,,\,(\Id-\Delta)\rho\rangle_{H^{-(2-\sigma)}\times H^{2-\sigma}}\,,
$$
where we have defined $\sigma=1-s$ and  $\wtilde{\phi}:=(\Id-\Delta)^{-1}\phi$.
The same formula holds true also for $\rho\in H^4$ and $\phi\in H^{1+s}$, for some $s>0$ arbitrarily small,
up to take the $H^{-(1+s)}\times H^{1+s}$ duality on the left-hand side, and the $H^{-(2-s)}\times H^{2-s}$ duality on the right-hand side. \\
In particular, for all tempered distributions $\rho$ and $\z$ in $H^4$, one has $\langle \Lambda(\rho,\z)\,,\,(\Id-\Delta)\rho\rangle_{H^{-2}\times H^2}\,=\,0$.
\end{lemma}

\begin{proof}
By density of $\mc{C}^\infty_0$ into $H^\s$ for any $\s\geq0$, it is enough to prove the previous properties assuming that all the functions involved in the computations are smooth.

First of all, by integrating by parts once, we have
\begin{equation} \label{eq:bil-action}
\langle\Lambda(\rho,\z)\,,\,\phi\rangle_{H^{-(2+s)}\times H^{2+s}}\,=\,-\,\int_{\R^2}\nabla^\perp(\Id-\Delta)\rho\cdot\nabla\phi\,\Delta^2\z\,dx\,.
\end{equation}
Then, by H\"older inequality there exists a constant $C>0$ such that
$$
\left|\langle\Lambda(\rho,\z)\,,\,\phi\rangle_{H^{-(2+s)}\times H^{s+2}}\right|\,\leq\,C\,\left\|\nabla^\perp(\Id-\Delta)\rho\right\|_{L^2}\,\left\|\Delta^2\z\right\|_{L^2}\,
\left\|\nabla\phi\right\|_{L^\infty}\,.
$$
Now, since $s>0$ (and the space dimension $d$ is equal to $2$), the space $H^{1+s}$ embeds continuously into $L^\infty$, so that $\left\|\nabla\phi\right\|_{L^\infty}\,\leq\,C\,\|\phi\|_{H^{s+2}}$.

The other properties easily follow in an analogous way, by using, in estimating the integral in \eqref{eq:bil-action},
the Sobolev embeddings $H^1\hookrightarrow L^p$ for all $p\in[2,+\infty[\,$ and $H^s\hookrightarrow L^q$, where $q=2/(1-s)$, which hold true
in dimension $d=2$ (see also Theorem 1.66 of \cite{B-C-D}).

\medbreak
As for the identity given in the statement of the lemma, by density again and continuity properties of $\Lambda$ established above, it is enough to work with smooth compactly supported functions.

First of all we observe that, for all $g$ and $h$ in $\mc C^\infty_0$, one has $\nabla^\perp g\cdot\nabla h\,=\,-\nabla g\cdot\nabla^\perp h$ pointwise.
In addition, we remark that if $\phi\in H^{2-s}$, then $\wtilde{\phi}\in H^{4-s}=H^{3+\s}$, by definition of $\s=1-s$, and if $\rho\in H^{3+s}$, then $(\Id-\Delta)\rho\in H^{1+s}=H^{2-\s}$.

When $\rho\in H^4$ and $\phi\in H^{1+s}$, the proof is exactly identical, but one has to work in the right spaces where the formula is continuous.
\end{proof}

The requirement $\z\in H^4$ is somehow critical in the theory of weak solutions. However, if one disposes of more regularity for $\z$, some easy generalizations of the previous lemma
can be obtained.
Here we limit ourselves to state the following properties.

\begin{lemma} \label{l:bilin-high}
For all $s>0$, $\Lambda$ acts continuously from $H^4\times H^{4+s}$ into $H^{-1}$. In particular, for any $\rho\in H^4$ and $\z\in H^5$, one has the equality
$$
\langle \Lambda(\rho,\z)\,,\,\Delta^2\z\rangle_{H^{-1}\times H^1}\,=\,0\,.
$$
\end{lemma}

\begin{proof}
First of all, let us take $q=2/(1-s)$ and $p>2$ such that $1/q\,+\,1/p\,=\,1/2$. Hence, making use of Sobolev embeddings, we can estimate
\begin{eqnarray*}
\left|\langle\Lambda(\rho,\z)\,,\,\phi\rangle_{H^{-1}\times H^{1}}\right| & \leq & C\,\left\|\nabla^\perp(\Id-\Delta)\rho\right\|_{L^p}\,\left\|\Delta^2\z\right\|_{L^q}\,
\left\|\nabla\phi\right\|_{L^2} \\
& \leq & C\,\left\|\nabla^\perp(\Id-\Delta)\rho\right\|_{H^1}\,\left\|\Delta^2\z\right\|_{H^s}\,\left\|\nabla\phi\right\|_{L^2}\,.
\end{eqnarray*}

The equality $\langle \Lambda(\rho,\z)\,,\,\Delta^2\z\rangle_{H^{-1}\times H^1}\,=\,0$ can be obtained arguing once again by density. First of all, for smooth functions we have
the identity
$$
\langle \Lambda(\rho,\z)\,,\,\Delta^2\z\rangle_{H^{-1}\times H^1}\,=\,-\,\int_{\R^2}\nabla^\perp(\Id-\Delta)\rho\cdot\nabla\Delta^2\z\,\Delta^2\z\,dx\,;
$$
now, since $\nabla\Delta^2\z\,\Delta^2\z\,=\,\nabla\left(|\Delta^2z|^2/2\right)$, integrating by parts with respect to the $\nabla^\perp$ operator gives us $0$. One then concludes the proof
by noticing that both sides of the final expression are continuous on $H^4\times H^5$.
\end{proof}

\begin{rem} \label{r:bilin-high}
By Bony's paraproduct decomposition \eqref{eq:bony} and Proposition \ref{p:op}, it is possible to see that the statement of Lemma \ref{l:bilin} is not optimal.
For instance, using also the fact that the space dimension is $d=2$, one has that $\Lambda$ maps continuously
$H^{4}\times H^{4+s}$ into $H^{s-1-\delta}$ for any $\delta>0$ arbitrarily small when $0<s<1$, into $H^{-\delta}$ for all $\delta>0$ when $s=1$, and into $L^2$ when $s>1$.
In particular, for all $\z\in H^{4+s}$ with $s>1/2$, it makes sense to apply $\Lambda(\rho,\z)$ to $\Delta^2\z$, obtaining that
$\langle \Lambda(\rho,\z)\,,\,\Delta^2\z\rangle_{H^{-(s-1-\delta)}\times H^{1-s+\delta}}\,=\,0$ (where $\delta>0$ is fixed small enough).

We preferred to keep the previous statement for simplicity of exposition, since this version is enough for our scopes.
\end{rem}

\begin{rem} \label{r:bil-formular}
We also point out that, whenever we take $(\rho,\phi)\in H^4\times H^{1+s}$ or $(\rho,\phi)\in H^{4+s}\times H^1$, for some $s>0$, we have the identity
$$
\langle \Lambda(\rho,\rho)\,,\,\phi\rangle\,=\,-\int_{\R^2}\nabla^\perp(\Id-\Delta)\rho\,\cdot\nabla\phi\;\left(\Id-\Delta+\Delta^2\right)\!\rho\;dx\,,
$$
where the symbol $\langle\,\cdot\,,\,\cdot\,\rangle$ denotes the suitable duality pair, respectively in the spaces $H^{-(1+s)}\times H^{1+s}$ and $H^{-1}\times H^1$.

Indeed, working first with $\mc C^\infty_0$ functions, several integrations by parts lead us to
\begin{align*}
\langle \Lambda(\rho,\rho)\,,\,\phi\rangle\,
&=\,-\,\int_{\R^2}\nabla^\perp(\Id-\Delta)\rho\,\cdot\nabla\phi\;\Delta^2\rho\,dx\;=\;-\,\int_{\R^2}\Div\left(\nabla^\perp(\Id-\Delta)\rho\,\phi\right)\;\Delta^2\rho\,dx\\
&=\quad\int_{\R^2}\nabla^\perp(\Id-\Delta)\rho\,\cdot\nabla\Delta^2\rho\,\phi\;\,dx\;=\;\int_{\R^2}\nabla^\perp(\Id-\Delta)\rho\,\cdot\nabla\left(\Id-\Delta+\Delta^2\right)\!\rho\;\phi\;dx\\
&=\,-\,\int_{\R^2}\nabla^\perp(\Id-\Delta)\rho\,\cdot\nabla\phi\;\left(\Id-\Delta+\Delta^2\right)\!\rho\;dx\,,
\end{align*}
where, in the last equality of the second line, we have used that $\nabla g\times\nabla^\perp g\,=\,0$ pointwise, for all smooth $g$.
\end{rem}

The analysis of the bilinear operator will be of fundamental importance in the theory weak solutions with minimal regularity, as well as in the propagation of higher smoothness,
for which we refer respectively to Sections \ref{s:weak} and \ref{s:strong}.

\subsection{Basic energy estimates} \label{ss:energy}

In the present subsection, we want to keep studying the structure of our equation, by establishing (formal) energy estimates for smooth enough solutions. As announced above, it is interesting
to notice that we can identify two levels energy conservation laws, which correspond to different levels of regularity of solutions.

Energy estimates of the first kind involve the quantity $(\Id-\Delta)r$: this seems to us the minimal kind of regularity which can be derived.
Indeed we are not able to deduce anything from testing the equation just on $r$: this fact may seem surprising, but actually it is quite natural
if one looks at the transport term.
As a matter of fact, taking the $L^2$ scalar product of the bilinear term with $r$ and integrating by parts, it is impossible to get rid of a fourth-order term; but the viscosity term gives a
control on the derivatives of $r$ up to the third order, so that the bilinear term seems to be out of control at this stage. See also relation \eqref{est:en-0} below about this point.
On the other hand, testing the bilinear operator on $-\Delta r$ will allow us to erase this bad term (keep in mind also Lemma \ref{l:bilin}).

In the end, we obtain the following statement.
\begin{prop} \label{p:a-priori_I}
Let us fix an $r_0\in H^3(\R^2)$ and a forcing term $f\in L^2_{\rm loc}\bigl(\R_+;H^{-2}(\R^2)\bigr)$. Let $r$ be a smooth solution of equation \eqref{eq:main} with external force $f$
and initial datum $r_0$.

Then, there exists a constant $C>0$ such that, for any time $t\geq0$ fixed, one has the estimate
\begin{align*}
 & \hspace{-0.5cm} \|r(t)\|^2_{L^2}\,+\,\|\nabla r(t)\|^2_{L^2}\,+\,\|\Delta r(t)\|^2_{L^2}\,+\,\|\nabla\Delta r(t)\|^2_{L^2}\,+\,\mu\int^t_0\|\Delta r(\tau)\|^2_{H^2}\,d\tau\,\leq \\
 & \qquad\qquad\leq\,C\;e^{\mu\,t}\left(\|r_0\|^2_{L^2}\,+\,\|\nabla r_0\|^2_{L^2}\,+\,\|\Delta r_0\|^2_{L^2}\,+\,\|\nabla\Delta r_0\|^2_{L^2}\,+\,\frac{1}{\mu}\int^t_0\|f(\tau)\|^2_{H^{-2}}\,d\tau\right)\,.
\end{align*}
\end{prop}

\begin{proof}
The proof is carried out in two steps: the former consists testing the equation on $r$, while in the latter we test it on $-\Delta r$; finally, it is just a matter of summing up the two expressions.

Let us start by taking the $L^2$ scalar product of equation \eqref{eq:main} by $r$: easy computations lead us to
\begin{eqnarray}
& & \frac{1}{2}\,\frac{d}{dt}\int\left(|r|^2+|\nabla r|^2+|\Delta r|^2\right)dx\,- \label{est:en-0} \\
& & \qquad\qquad-\,\int\nabla^\perp(\Id-\Delta)r\cdot\nabla r\;\Delta^2r\,dx\,+\,
\mu\int\left(|\Delta r|^2+|\nabla\Delta r|^2\right)\,dx\,=\,\lan f,r\ran_{H^{-2}\times H^2}\,. \nonumber
\end{eqnarray}

Now, we take the $L^2$ scalar product of the equation by $-\Delta r$: this time we get
\begin{eqnarray}
& & \frac{1}{2}\,\frac{d}{dt}\int\left(|\nabla r|^2+|\Delta r|^2+|\nabla\Delta r|^2\right)dx\,+ \label{est:en_1} \\ 
& & \qquad\qquad+\,\int\nabla^\perp(\Id-\Delta)r\cdot\nabla\Delta r\;\Delta^2r\,dx\,+\,\mu\int\left(|\nabla\Delta r|^2+|\Delta^2r|^2\right)\,dx\,=\,\lan f,-\Delta r\ran_{H^{-2}\times H^2}\,. \nonumber
\end{eqnarray}

Summing up \eqref{est:en-0} and \eqref{est:en_1} and integrating in time, we find the equality
\begin{align*}
& \|r(t)\|^2_{L^2}\,+\,2\,\|\nabla r(t)\|^2_{L^2}\,+\,2\,\|\Delta r(t)\|^2_{L^2}\,+\,\|\nabla\Delta r(t)\|^2_{L^2}\,+ \\ 
& +\,2\,\mu\int^t_0\left(\|\Delta r(\tau)\|^2_{L^2}+2\|\nabla\Delta r(\tau)\|^2_{L^2}+\|\Delta^2r(\tau)\|^2_{L^2}\right)d\tau\,= \\ 
& =\,\|r_0\|^2_{L^2}\,+\,2\,\|\nabla r_0\|^2_{L^2}+\,2\,\|\Delta r_0\|^2_{L^2}\,+\,\|\nabla\Delta r_0\|^2_{L^2}\,+\,
2\int^t_0\lan f(\tau),(\Id-\Delta)r(\tau)\ran_{H^{-2}\times H^2}\,d\tau\,. 
\end{align*}
for all $t\geq0$, owing to the fact that $\int\nabla^\perp(\Id-\Delta)r\cdot\nabla(\Id-\Delta)r\;\Delta^2r=0$. For simplicity, let us introduce the notation
$$
X(t)\,:=\,\|r(t)\|^2_{L^2}\,+\,\|\nabla r(t)\|^2_{L^2}+\,\|\Delta r(t)\|^2_{L^2}\,+\,\|\nabla\Delta r(t)\|^2_{L^2}\,,
$$
and let us set $X_0$ to be the same quantity computed on the initial datum.

Now we estimate the forcing term in the following way:
\begin{eqnarray*}
\left|\lan f,(\Id-\Delta)r\ran_{H^{-2}\times H^2}\right| & \leq & \|f\|_{H^{-2}}\,\left\|(\Id-\Delta)r\right\|_{H^2} \\
& \leq & C\,\|f\|_{H^{-2}}\,\left(\|r\|_{L^2}+\|\nabla r\|_{L^2}+\|\Delta r\|_{L^2}+\|\nabla\Delta r\|_{L^2}+\|\Delta^2r\|_{L^2}\right)\,,
\end{eqnarray*}
and an application of Young inequality immediately gives us
\begin{align}
&X(t)\,+\,\mu\int^t_0\left(\|\Delta r(\tau)\|^2_{L^2}+\|\nabla\Delta r(\tau)\|^2_{L^2}+\|\Delta^2r(\tau)\|^2_{L^2}\right)d\tau\,\leq \label{est:X} \\
&\qquad\qquad\qquad\qquad\leq\,X_0\,+\,\frac{C}{\mu}\int^t_0\|f(\tau)\|^2_{H^{-2}}\,d\tau\,+\,\mu\int^t_0\left(\|r(\tau)\|^2_{L^2}\,+\,\|\nabla r(\tau)\|^2_{L^2}\right)d\tau\,. \nonumber
\end{align}
At this point, Gronwall lemma allows us to obtain the claimed estimate.
\end{proof}

\begin{rem} \label{r:en-1}
We point out here that having additional regularity in space on $f$, like e.g. $f\in L^2_{\rm loc}(\R_+;L^2)$, does not help to improve the estimates. Indeed, the presence of
the lower order term forces us to apply Gronwall's inequality, and this produces an exponential growth in time of the norm of the solution.

On the other hand, we have to notice that, whenever $f\equiv0$, the estimates simply reduce to $X(t)\,+\,\mu\,\left\|\Delta^2 r\right\|^2_{L^2_t(H^2)}\,\leq\,C\,X_0$,
where $X(t)$ and $X_0$ are the quantities introduced in the proof of the previous proposition. See also Lemma \ref{l:r-prelim} below for a similar estimate in the case when $f\neq0$.
\end{rem}

This having been established, let us derive higher regularity estimates, under suitable assumptions on the initial datum and the external force.
\begin{prop} \label{p:a-priori_II}
Suppose now that $r_0\in H^4(\R^2)$ and $f\in L^2_{\rm loc}\bigl(\R_+;H^{-1}(\R^2)\bigr)$. Let $r$ be a smooth solution of equation \eqref{eq:main} with external force $f$
and initial datum $r_0$.

Then, there exists a constant $C>0$ such that, for any time $t\geq0$ fixed, one has the estimate
\begin{eqnarray*}
& & \hspace{-0.5cm} \|r(t)\|^2_{L^2}\,+\,\|\nabla r(t)\|^2_{L^2}\,+\,\|\Delta r(t)\|^2_{L^2}\,+\,\|\nabla\Delta r(t)\|^2_{L^2}\,+\,\|\Delta^2r(t)\|^2_{L^2}\,+\,
\mu\int^t_0{\|\nabla \Delta^2r(\tau)\|^2_{L^2}}\,d\tau\,\leq \\
& & \leq\,C\;e^{\mu\,t}\left(\|r_0\|^2_{L^2}\,+\,\|\nabla r_0\|^2_{L^2}\,+\,\|\Delta r_0\|^2_{L^2}\,+\,\|\nabla\Delta r_0\|^2_{L^2}\,+\,\|\Delta^2r_0\|^2_{L^2}\,+\,
\frac{1}{\mu}\int^t_0\|f(\tau)\|^2_{H^{-1}}\,d\tau\right)\,.
\end{eqnarray*}
\end{prop}

\begin{proof}
We start by multiplying equation \eqref{eq:main} by $\Delta^2r$. We notice that, in light of Lemma \ref{l:bilin-high} above, the transport term identically vanishes. So, straightforward computations
give us the equality
\begin{equation} \label{est:en-2}
\frac{1}{2}\,\frac{d}{dt}\int\left(|\Delta r|^2+|\nabla\Delta r|^2+|\Delta^2 r|^2\right)dx\,+\,\mu\int\left(|\Delta^2 r|^2+|\nabla\Delta^2 r|^2\right)\,dx\,=\,\lan f,\Delta^2r\ran_{H^{-1}\times H^1}\,.
\end{equation}

The control of the forcing term can be performed similarly as before: we have
$$ 
\left|\lan f,\Delta^2r\ran_{H^{-1}\times H^1}\right|\,\leq\,\|f\|_{H^{-1}}\,\left\|\Delta^2r\right\|_{H^1}\,\leq\,C\,\|f\|_{H^{-1}}\,\left(\|\Delta^2r\|_{L^2}+\|\nabla\Delta^2 r\|_{L^2}\right)\,.
$$ 
Applying Young inequality to the previous relation, inserting the result in \eqref{est:en-2} and integrating in time, we finally find
\begin{equation} \label{est:en-3}
Y(t)\,+\,\mu\int^t_0\left(\|\Delta^2 r(\tau)\|^2_{L^2}+\|\nabla\Delta^2 r(\tau)\|^2_{L^2}\right)d\tau\,\leq\,C\left(Y_0\,+\,\frac{1}{\mu}\int^t_0\|f(\tau)\|^2_{H^{-1}}\,d\tau\right)\,,
\end{equation}
where we have defined
\begin{equation} \label{def:Y}
Y(t)\,:=\,\|\Delta r(t)\|^2_{L^2}\,+\,\|\nabla\Delta r(t)\|^2_{L^2}\,+\,\|\Delta^2 r(t)\|^2_{L^2}
\end{equation}
and $Y_0$ to be the same quantity, when computed on the initial datum $r_0$.

At this point, we sum inequality \eqref{est:en-3} to the one found in Proposition \ref{p:a-priori_I}. Noticing that $e^{\mu t}\geq1$ for all $t\geq0$, and that $H^{-1}\hookrightarrow H^{-2}$,
it is easy to get the desired estimate.
\end{proof}

\begin{rem} \label{r:en-2}
When $f$ is globally integrable in time, i.e. $f\in L^2\bigl(\R_+;H^{-1}(\R^2)\bigr)$ , one gets the better estimate
$Y(t)\,+\,\mu\,\left\|\Delta^2 r\right\|^2_{L^2_t(H^1)}\,\leq\,K\,\bigl(Y_0+\|f\|^2_{L^2_t(H^{-1})}\bigr)$.

In particular, whenever $f\equiv0$,  one has no more the exponential growth of the Sobolev norms in the estimates of both Propositions \ref{p:a-priori_I} and \ref{p:a-priori_II}, and
the regularity in time (see the statement of Theorem \ref{t:weak}) becomes global.
We refer again to Lemma \ref{l:r-prelim} below for further results in this spirit.
\end{rem}

We postpone to Section \ref{s:strong} the proof of both intermediate and higher order energy estimates, because they are based on a different technique (namely, a paralinearization
of the equation and on a broad use of paradifferential calculus).

\section{Weak solutions theory} \label{s:weak}

This section is devoted to the proof of Theorem \ref{t:weak}. Namely, in a first time we will establish existence of global in time weak solutions to equation \eqref{eq:main} having minimal
regularity. Then, we will derive their uniqueness by studying a higher order parabolic equation, which can be viewed as the analogue of the time-dependent Stokes problem in our context. Besides,
this analysis will allow us to recover the energy equality.

In fact, our system shares a lot of similarities with the homogeneous incompressible Navier-Stokes equations; hence, in this section we mainly resort to arguments which are typical
for dealing with that problem (we refer e.g. to Chapters 2 and 3 of \cite{C-D-G-G}).

\subsection{Global in time existence} \label{ss:weak-ex}
We start by treating the existence issue. 
First of all, let us define the Fourier multipliers $a(\xi)\,:=\,( 1+ |\xi|^2 + |\xi|^4)$ and $d(\xi)\,:=\,a^{-1}(\xi)$; let respectively $\mc A$ and $\D$ be the related pseudo-differential
operators, defined by the formula
\begin{equation} \label{def:A-D}
\mc A g\, :=\, \Ff^{-1} \big( ( 1+ |\xi|^2 + |\xi|^4) \hat g (\xi)\big)\qquad\quad\mbox{ and }\qquad\quad
\D g\, :=\, \Ff^{-1} \big( ( 1+ |\xi|^2 + |\xi|^4)^{-1} \hat g (\xi)\big)
\end{equation}
for any $g\in\mc{S}(\R^2)$, where $\mc S$ denotes the Schwartz class and $\Ff$ the Fourier transform in the space variables.
Straightforward computations show that $\mc A$ and $\D $ are continuous operator respectivly from $H^{s+4}$ to $H^{s}$ and from $H^s$ to $H^{s+4}$, for any real $s$
(see also Proposition 2.78 of \cite{B-C-D}), and each of them is the inverse operator of the other.

Applying operator $\D$ to  system \eqref{eq:main}, we can recast it in the following form:
\begin{equation}\label{eq:main_reform}
	\begin{cases}
		\d_t r\,+\,\D\Lambda(r,r)\,+\,\mu \D \Delta^2 (\Id - \Delta )r\,=\,\D f \\[1ex]
		r_{|t=0}\,=\,r_0\,.
	\end{cases}
\end{equation}

Our next goal is to contstruct smooth approximate solutions to system \eqref{eq:main_reform}, and then prove that they converge to a weak solution of the original equation \eqref{eq:main}.

\subsubsection{Construction of approximate solutions and uniform bounds} \label{sss:approx}
Denoting by $\mathds{1}_{A}$ the characteristic function of a set $A$, for any $n$ we introduce the regularizing operator $\J_n$ by the formula
\begin{equation*}
	\J_n g\, :=\, \Ff^{-1}\big( \mathds{1}_{\{|\xi|^2\leq n\}}(\xi)\; \what g (\xi) \big)\,,
\end{equation*}
which localizes the Fourier transform of $g$ in the ball $B_\xi(0, n^{1/2})$. Correspondingly, we define the space $L^2_n\,:=\,\J_nL^2$; remark that, for any $n\geq1$, functions in $L^2_n$
are actually smooth, since they have compact spectrum, and more precisely the embedding $L^2_n\,\subset\,H^s$ holds for all $s\in\R$. We also set
$$
r_0^n\,:=\,\J_nr_0\qquad\qquad\mbox{ and }\qquad\qquad f^n(t)\,:=\,\J_nf(t)\,.
$$
Notice that, for all time $T>0$ fixed, we have the inequalities $\left\|r^n_0\right\|_{H^3}\leq C\left\|r_0\right\|_{H^3}$ and
$\left\|f^n\right\|_{L^2_T(H^{-2})}\leq C\left\|f\right\|_{L^2_T(H^{-2})}$, together with the convergence properties
\begin{equation} \label{conv:initial}
\lim_{n\ra+\infty}\left\|r^n_0\,-\,r_0\right\|_{H^3}\,=\,0\qquad\qquad\mbox{ and }\qquad\qquad\lim_{n\ra+\infty}\left\|f^n\,-\,f\right\|_{L^2_T(H^{-2})}\,=\,0\,.
\end{equation}
The former limit can be easily proved by using dyadic characterization of Sobolev spaces (see e.g. Lemma \ref{l:Id-S} in the Appendix); for the latter, one has also to
apply Lebesgue dominated convergence theorem in order to deal with the time integral.

Now, we consider the approximate system
\begin{equation}\label{appx_system}
	\begin{cases}
		\d_tr\,+\,\J_n\D\Lambda(r,r)\,+\,\mu\,\J_n\D\Delta^2(\Id - \Delta )r\,=\,\D f^n \\[1ex]
		r_{|t=0}\,=\,r^n_0\,,
	\end{cases}
\end{equation}
for which we can show existence of global in time smooth solutions.

\begin{prop} \label{p:ex-smooth}
For all $n\geq1$ fixed, system \eqref{appx_system} admits a unique classical solution $r^n$, which is globally defined in time and which fulfills
	\begin{equation*}
		r^n\,\in\, 
L_{\rm loc}^\infty\bigl(\RR_+;H^3(\R^2)\bigr)\,\cap\, L^2_{\rm loc}\bigl(\RR_+; H^4(\R^2)\bigr)\,.
	\end{equation*}
Moreover for all $T>0$, there exists a positive constant $C_T$, possibly depending on $T$ but not on $n$, such that
	\begin{equation*}
		\| r^n \|_{L^\infty_T(H^3)}\,+\,\|r^n \|_{L^2_T(H^4)}\,\leq\,C_T\, \left(\| r_0 \|_{H^3}\,+\,\|f\|_{L^2_T(H^{-2})}\right)\,.
	\end{equation*}
\end{prop}
\begin{proof}
To begin with, let us define the functional $\mc L_n$ by the formula
	\begin{equation*}
		\mc L_n ( g )\,:=\, -\,\J_n \D \big( \nabla^{\perp}( \Id - \Delta )g \cdot \nabla \Delta^2 g \big)\,-\, \mu\,  \D \Delta^2 (\Id - \Delta )\J_n g\,.
	\end{equation*}
Observe that, using the embeddings $L^2_n\,\subset\,H^s$ for all $s\in\R$ to control the bilinear term, one easily deduces that $\mc L_n$ is an endomorphism of the space $L^2_n$;
more precisely, there exists a constant $C_n$ such that, for any function $g\in L^2_n$,
	\begin{align*}
		\| \mc L_n ( g ) \|_{L^2_n}\,\leq\,C_n\,\| g \|_{L^2_n}\,.
	\end{align*}
	
Next, we rewrite system \eqref{appx_system} as an ODE in the closed subspace $L^2_n$:
	\begin{equation}\label{CL_system}
\partial_t r\,=\, \mc L_n(r)\,+\,\D f^n\;,\qquad\qquad\qquad\mbox{ with }\qquad r_{|t=0}\,=\,r^n_0\,.
	\end{equation}
Then, Cauchy-Lipschitz theorem yields the existence of a unique maximal solution $r^n$ which belongs to $\mc C^1\bigl([0,T_n[\,;L^2_n\bigr)$, for some positive time $T_n$
(notice that, actually, $r^n$ is $\mc C^\infty$ in time).

Next, we observe that, from equation \eqref{CL_system}, Lemma \ref{l:bilin} and the continuity of $\D$ as an operator from $H^s$ to $H^{s+4}$, we infer the estimate
$$
\left\|\d_tr^n(t)\right\|_{L^2}\,\leq\,C\left(\left\|f^n(t)\right\|_{L^2}\,+\,n\,\left\|r^n(t)\right\|_{L^2}\,+\,n^{5/2}\,\left\|r^n(t)\right\|^2_{L^2}\right)\,.
$$
From this inequality we deduce that, if $\left\|r^n(t)\right\|_{L^2}$ is bounded on some interval $[0,T[\,$, then $\left\|\d_tr^n(t)\right\|_{L^2}$ remains bounded too. Therefore,
by Cauchy criterion, the solution can be extended beyond $T$. This immediately implies that $T_n=+\infty$ for all $n$, provided we show a uniform bound on
$\left\|r^n(t)\right\|_{L^2}$.

\medbreak
Hence, our next goal is to get estimates on $r^n$, uniformly in $n\in\N$. Recall that $r^n$ is a classical solution of \eqref{appx_system}: then, applying operator $\mc A$
to the equation, where $\mc A$ has been defined in \eqref{def:A-D} above, we see that $r^n$ satisfies
\begin{equation} \label{eq:approx}
	\partial_t ( \Id - \Delta  + \Delta^2   )r^n\,+\,\J_n\Lambda(r^n,r^n)\,+\, \mu\,\J_n\Delta^2 (\Id - \Delta ) r^n\,  =\, f^n\,.
	\end{equation}

Now, we multiply both the left and right-hand side by $(\Id-\Delta)r^n$ and we integrate over $\RR^2$. Notice that $\J_n$ is a self-adjoint operator on $L^2$,
$\J_n^2\,\equiv\,\J_n$ and $\J_nr^n\,=\,r^n$; then, the bilinear term identically vanishes, thanks to the last property of Lemma \ref{l:bilin}.
Therefore, the same computations as in Proposition \ref{p:a-priori_I} entail
$$
\|r^n(t)\|^2_{H^3}\,+\,\mu\int^t_0\|\Delta r^n(\tau)\|^2_{H^2}\,d\tau\,\leq\,C\;e^{\mu\,t}\left(\|r^n_0\|^2_{H^3}\,+\,
\frac{1}{\mu}\int^t_0\|f^n(\tau)\|^2_{H^{-2}}\,d\tau\right)\,,
$$
for some constant $C>0$ not depending on $n\geq1$. This estimate, together with the properties of the appoximate families $\bigl(r^n_0\bigr)_n$ and $\bigl(f^n\bigr)_n$, immediately implies
that, for all $T>0$ fixed, there exists a constant $C_T>0$ such that  
\begin{equation*}
\sup_{t\in[0,T]}\| r^n (t) \|_{H^3}^2\,+\,\mu\int_0^T \| r^n(t)\|_{H^4}^2\,dt\, \leq\,C_T\,\left(\| r_0\|_{H^3}\,+\,\left\|f\right\|^2_{L^2_T(H^{-2})}\right)\,.
	\end{equation*}

The proposition is hence completely proved.
\end{proof}

As a consequence of the previous proposition, up to the extraction of a subsequence, we gather the existence of a function
\begin{equation} \label{def:r}
r\,\in\,L_{\rm loc}^\infty\bigl(\RR_+;H^3(\R^2)\bigr)\,\cap\, L^2_{\rm loc}\bigl(\RR_+; H^4(\R^2)\bigr)\qquad\qquad\mbox{ such that }\qquad r^n\,\stackrel{*}{\rightharpoonup}\,r
\end{equation}
in the previous space, where the symbol $\stackrel{*}{\rightharpoonup}$ denotes the convergence in the weak-$*$ topology.

\subsubsection{Compactness and convergence} \label{sss:comp-conv}
The previous convergence property is not enough to pass to the limit in the weak formulation of equation \eqref{appx_system}, the problem relying of course in the convergence
of the non-linear term. Therefore, we need to derive compactness properties for the family of approximate solutions $\bigl(r^n\bigr)_n$: this is our next goal.

\begin{lemma} \label{l:cpt}
The family of smooth solutions $\bigl(r^n\bigr)_n$, constructed in the previous paragraph, is compact in the space $L^2_T\bigl(H^3_{\rm loc}(\R^2)\bigr)$ for all time $T>0$ fixed.
In particular, up to a further extraction, $r^n$ strongly converges to $r$ in the previous space.
\end{lemma}

\begin{proof}
We start by considering equation \eqref{CL_system}. Since $\bigl(r^n\bigr)_n\,\subset\,L^2_T(H^4)$ by Proposition \ref{p:ex-smooth}, we get, for some $C>0$ independent of $n$,
the bound
$$
\bigl\|\D \Delta^2 (\Id - \Delta )\J_nr^n\bigr\|_{L^2_T(H^2)}\,\leq\,C\,.
$$
On the other hand, combining the previous property with the fact that $\bigl(r^n\bigr)_n$ is uniformly bounded also in $L^\infty_T(H^3)$, by Lemma \ref{l:bilin} we deduce that
$$
\bigl\|\J_n \D\Lambda(r^n,r^n)\bigr\|_{L^2_T(H^{2-s})}\,\leq\,C\,,
$$
where $s>0$ is arbitrarily small and the estimate is still uniform with respect to $n\geq1$.

These inequalities together with equation \eqref{CL_system} imply that the sequence $\bigl(\d_tr^n\bigr)_n$ is uniformly bounded in $L^2_T(H^{2-s})$, and then it follows that
$\bigl(r^n\bigr)_n$ is uniformly bounded in $\mc{C}^{0,1/2}_T(H^{2-s})$.

Hence, by Ascoli-Arzel\`a and Rellich-Kondrachov theorems we gather that $\bigl(r^n\bigr)_n$ is compactly included in e.g. the space $L^\infty_T(H^\alpha_{\rm loc})$, for all
$2-s\leq\alpha<3$. Therefore, up to passing to a suitable subsequence, we can assume the strong convergence
$$
r^n\,\longrightarrow\,r\qquad\qquad\mbox{ in }\qquad L^\infty_T(H^{2}_{\rm loc})\,.
$$
Finally, combining this fact with the uniform boundedness $\bigl(r^n\bigr)_n\,\subset\,L^2_T(H^4)$ again gives us the strong convergence in $L^2_T(H^3_{\rm loc})$, as claimed.
\end{proof}

Thanks to the previous result, we can now pass to the limit in the weak formulation of equation \eqref{eq:approx}, and prove that $r$ is a weak solution to the original
system \eqref{eq:main}. This concludes the proof of the existence part of Theorem \ref{t:weak}.

\begin{prop} \label{p:limit}
The limit-point $r$ of the sequence $\bigl(r^n\bigr)_n$, identified in \eqref{def:r} above, solves equation \eqref{eq:main} in the weak sense. Moreover, it verifies the energy
inequality stated in Proposition \ref{p:a-priori_I}.
\end{prop}

\begin{proof}
Let us write down the weak formulation of \eqref{eq:approx}: for $\phi\in\mc C^{\infty}_0\bigl([0,T[\,\times\R^2\bigr)$, we have
\begin{eqnarray*}
& & \hspace{-0.7cm}
-\int^T_0\!\!\!\int_{\R^2}\bigl(\Id-\Delta+\Delta^2\bigr)r^n\;\d_t\phi\,-\,\int^T_0\!\!\!\int_{\R^2}\Delta^2r^n\;\nabla^\perp(\Id-\Delta)r^n\cdot\nabla\J_n\phi\,+ \\
& & \hspace{-0.3cm}
+\,\mu\int^T_0\!\!\!\int_{\R^2}\Delta(\Id-\Delta)r^n\,\J_n\Delta\phi\,=\,\int^T_0\!\!\!\langle f^n(t),\phi(t)\rangle_{H^{-2}\times H^2}\,+ \,
\int_{\R^2}\Bigl((\Id-\Delta)r^n_0\,\phi\,-\,\nabla\Delta r^n_0\cdot\nabla\phi\Bigr)\,.
\end{eqnarray*}
Thanks to the weak convergence properties of $r^n$ and the strong convergence properties of the approximate initial data, it is easy to pass to the limit in the $\d_t\phi$ terms
and on the right-hand side of the previous equality.

As for the ``viscosity'' and transport terms, first of all we notice that, by Lemma 2.4 of \cite{C-D-G-G}, the strong convergence $\J_n\phi\,\rightarrow\,\phi$ holds true in
$\mc C\bigl([0,T];H^s\bigr)$ for all $s\geq0$. So, taking e.g. $s=0$ and using the uniform boundedness of $\bigl(\Delta(\Id-\Delta)r^n\bigr)_n$ in $L^2_T(L^2)$ we immediately get
$$
\lim_{n\ra+\infty}\int^T_0\!\!\!\int_{\R^2}\Delta(\Id-\Delta)r^n\,\J_n\Delta\phi\,=\,\lim_{n\ra+\infty}\int^T_0\!\!\!\int_{\R^2}\Delta(\Id-\Delta)r^n\,\Delta\phi\,=\,
\int^T_0\!\!\!\int_{\R^2}\Delta(\Id-\Delta)r\,\Delta\phi\,.
$$
Moreover, picking some $s>1$ and taking advantage of the uniform boundedness of the family $\bigl(\Delta^2r^n\;\nabla^\perp(\Id-\Delta)r^n\bigr)_n$ in e.g. $L^2_T(L^1)$, we have
$$
\lim_{n\ra+\infty}\int^T_0\!\!\!\int_{\R^2}\Delta^2r^n\;\nabla^\perp(\Id-\Delta)r^n\cdot\nabla\J_n\phi\,=\,
\lim_{n\ra+\infty}\int^T_0\!\!\!\int_{\R^2}\Delta^2r^n\;\nabla^\perp(\Id-\Delta)r^n\cdot\nabla\phi\,,
$$
so that we reduce our convergence problem to study the limit of the term on the right-hand side of the previous equality. Now, for this term we can use the weak convergence
$\Delta^2r^n\,\rightharpoonup\,\Delta^2r$ in $L^2_T(L^2)$ and the strong convergence (recall Lemma \ref{l:cpt} above) of $\nabla^\perp(\Id-\Delta)r^n$ to $\nabla^\perp(\Id-\Delta)r$
in $L^2_T(L^2)$: we finally obtain
$$
\lim_{n\ra+\infty}\int^T_0\!\!\!\int_{\R^2}\Delta^2r^n\;\nabla^\perp(\Id-\Delta)r^n\cdot\nabla\phi\,=\,\int^T_0\!\!\!\int_{\R^2}\Delta^2r\;\nabla^\perp(\Id-\Delta)r\cdot\nabla\phi\,,
$$
and this completes the proof that $r$ is indeed a weak solution of equation \eqref{eq:main}.

The proof of the energy inequality directly follows from the weak convergence properties $\bigl(r^n\bigr)_n$ to $r$ combined with Fatou's lemma, and the strong convergence
properties \eqref{conv:initial}.
\end{proof}

To complete this part, we want to establish continuity properties of the solution in time. This analysis will justify also in which sense the initial datum is taken at $t=0$.
\begin{prop} \label{p:r-time}
Let $r$ be a weak solution to equation \eqref{eq:main} related to the initial datum $r_0\,\in\, H^3$ and to the external force $f\,\in\,L^2_{\rm loc}\bigl(\R_+;H^{-2}(\R^2)\bigr)$,
which moreover satisfies the energy inequality of Proposition \ref{p:a-priori_I}.

Then $(\Id-\Delta+\Delta^2)r\,\in\,\mc C\bigl(\R_+;H^{-2}(\R^2)\bigr)$.
\end{prop}

We start by proving a preliminary lemma, which will be fundamental also in the proof of uniqueness of weak solutions, carried out in Subsection \ref{ss:weak-u}.
\begin{lemma} \label{l:bilinear}
Under the hypotheses of Proposition \ref{p:r-time}, we have, for all time $T>0$,
$$
\Lambda(r,r)\,\in\,L^{4/3}\bigl([0,T];H^{-3/2}(\R^2)\bigr)\,. 
$$
\end{lemma}

\begin{proof}[Proof of Lemma \ref{l:bilinear}]
By definition, we have $\Lambda(r,r)\,=\,\div\left(\nabla^\perp(\Id-\Delta)r\;\Delta^2r\right)$, with
$\nabla^\perp(\Id-\Delta)r\,\in\,L^\infty_T(L^2)\cap L^2_T(H^1)$ and $\Delta^2r\,\in\,L^2_T(L^2)$ thanks to the energy inequality.

By interpolation and Gagliardo-Nirenberg inequality (see Corollary 1.2 of \cite{C-D-G-G}), we deduce that $\nabla^\perp(\Id-\Delta)r$ belongs to $L^4_T(L^4)$,
and hence $\nabla^\perp(\Id-\Delta)r\;\Delta^2r\,\in\,L^{4/3}_T(L^{4/3})$.
At this point, Sobolev embeddings imply that $L^{4/3}\,\hookrightarrow\,H^{-1/2}$, and this property concludes the proof of the Lemma.
\end{proof}

We are now in the position of proving Proposition \ref{p:r-time}.
\begin{proof}[Proof of Proposition \ref{p:r-time}]
Since $r$ is a weak solution to equation \eqref{eq:main}, we know that
$$
\d_t(\Id-\Delta+\Delta^2)r\,=\,f\,+\,\mu\,\Delta^2(\Id-\Delta)r\,-\,\Lambda(r,r)
$$
in the sense of distributions.
By hypothesis, $f\,\in\, L^2_T(H^{-2})$, while the viscosity term belongs to the same space in view of \eqref{def:r}. Moreover, thanks to Lemma \ref{l:bilinear}, we get
$\Lambda(r,r)\,\in\,L^{4/3}_T(H^{-2})$.

As a consequence, keeping in mind that $r_0\in H^3$, we deduce that $(\Id-\Delta+\Delta^2)r\,\in\,\mc C^{0,1/4}_T(H^{-2})$,
which in particular implies the claim of the proposition.
\end{proof}

\subsection{Uniqueness of weak solutions} \label{ss:weak-u}

The present subsection is devoted to the proof of uniqueness of weak solutions to our equation.

Rather than proving directly stability estimates, in a first time we focus on the study
of a higher order parabolic equation related to our problem, which can be viewed as a sort of generalization of the time-dependent Stokes equation in our context.
In passing, this analysis will allow us to improve time regularity of solutions, and to justify the energy equality in Theorem \ref{t:weak}.

After that, we will come back to the problem of uniqueness of weak solutions for \eqref{eq:main}. The previous part will be fundamental in the proof of uniqueness, since we will
exploit the underlying parabolic structure of our equation.

\subsubsection{A higher order parabolic equation} \label{sss:parabolic}

The scope of the present paragraph is to study existence and uniqueness of weak solutions to the following higher order parabolic-type equation:
\begin{equation}\label{eq:parab}
	\begin{cases}
		\partial_t\left(\Id-\Delta+\Delta^2\right)w\,+\,\mu\,\Delta^2(\Id-\Delta)w\,=\,f\,+\,g 	\\[1ex]
		w_{|t=0}\,=\,w_0\,.
	\end{cases}
\end{equation}
The weak formulation of \eqref{eq:parab} is the same as the one given in Definition \ref{d:weak} above, with the appropriate modifications in order to treat also the $g$ term and
to erase the bilinear term.

We notice that, the non-linear term having disappeared, this equation has ``basic'' energy estimates. Namely we can obtain \textsl{a priori} bounds by multiplying this equation just by $w$,
and then we can build up a theory of weak solutions for initial data just in $H^2$ (and suitable external forces).

Nonetheless, in view of the application of this study to our problem, we keep considering here initial data $w_0$ in $H^3$, and external forces $f$ in $L^2_T(H^{-2})$.
Moreover, in order to be able to deal with the bilinear term $\Lambda$, we add also a second forcing term $g\in L^{4/3}_T(H^{-3/2})$, keep in mind Lemma \ref{l:bilinear}.

We aim at proving the following result.
\begin{thm} \label{t:parabolic}
Let $w_0\in H^3(\R^2)$, $f\in L^2_{\rm loc}\bigl(\R_+;H^{-2}(\R^2)\bigr)$ and $g\in L^{4/3}_{\rm loc}\bigl(\R_+;H^{-3/2}(\R^2)\bigr)$ be given.

Then there exists a unique weak solution $w$ to equation \eqref{eq:parab}, which belongs to the energy space
$$
\mc{C}\bigl(\R_+;H^{3}(\R^2)\bigr)\,\cap\,L^\infty_{\rm loc}\bigl(\R_+;H^3(\R^2)\bigr)\,\cap\,L^2_{\rm loc}\bigl(\R_+;H^4(\R^2)\bigr)\,.
$$
In addition, $w$ 
satisfies the energy equality
\begin{eqnarray*}
& & \hspace{-0.7cm}
\mc E[w(t)]\,+\,\mu\int^t_0\left(\|\Delta w(\tau)\|^2_{L^2}+2\|\nabla\Delta w(\tau)\|^2_{L^2}+\|\Delta^2w(\tau)\|^2_{L^2}\right)\,d\tau\,= \\
& & \qquad =\,\mc E[w_0]\,+\,\int^t_0\lan f(\tau),(\Id-\Delta)w(\tau)\ran_{H^{-2}\times H^2}\,d\tau\,+\,\int^t_0\lan g(\tau),(\Id-\Delta)w(\tau)\ran_{H^{-3/2}\times H^{3/2}}\,d\tau\,,
\end{eqnarray*}
where the functional $\mc E$ is defined as in the statement of Theorem \ref{t:weak}.
\end{thm}

\begin{proof}
To begin with, let us prove existence. The argument being very similar to the one used in the previous subsection, we will just sketch it.

Resorting to the same notations as in Paragraph \ref{sss:approx}, we set $w^n_0=\J_nw_0$, $f^n=\J_nf$ and $g^n=\J_ng$. Then, we first solve the approximate problem
$$
	\begin{cases}
		\partial_tw^n\,+\,\mu\,\J_n\D\Delta^2(\Id-\Delta)w^n\,=\,\D f^n\,+\,\D g^n 	\\[1ex]
		w^n_{|t=0}\,=\,w^n_0\,.
	\end{cases}
$$
This is an ODE on the space $L^2_n$, so that Cauchy-Lipschitz theorem provides us with a global solution $w^n$ belonging to the space $\mc C^1\bigl(\R_+;L^2_n\bigr)$.
Moreover, since $w_n$ is a classical solution, we can recast the previous equation in the form
\begin{equation} \label{eq:w^n}
\partial_t\left(\Id-\Delta+\Delta^2\right)w^n\,+\,\mu\,\J_n\Delta^2(\Id-\Delta)w^n\,=\,f^n\,+\,g^n\,.
\end{equation}
Taking the $L^2$ scalar product of this equation with $(\Id-\Delta)w^n$ and then integrating in time (recall also the formal computations performed in Subsection \ref{ss:energy})
lead us to
\begin{eqnarray}
& & \hspace{-0.7cm}
\mc E[w^n(t)]\,+\,\mu\int^t_0\left(\|\Delta w^n\|^2_{L^2}+2\|\nabla\Delta w^n\|^2_{L^2}+\|\Delta^2w^n\|^2_{L^2}\right)\,d\tau\,= \label{est:en-eq_n} \\
& & \qquad =\,\mc E[w^n_0]\,+\,\int^t_0\lan f^n,(\Id-\Delta)w^n\ran_{H^{-2}\times H^2}\,d\tau\,+\,\int^t_0\lan g^n,(\Id-\Delta)w^n\ran_{H^{-3/2}\times H^{3/2}}\,d\tau\,. \nonumber
\end{eqnarray}
Notice that, from this relation, we immediately deduce uniform bounds for $w^n$ in the spaces of Theorem \ref{t:parabolic}. As a matter of fact, the control of 
the term in $f^n$ is straightforward, while for $g^n$ it is enough to observe that, by interpolation and Young inequality, one has
\begin{eqnarray*}
\bigl|\lan g^n,(\Id-\Delta)w^n\ran_{H^{-3/2}\times H^{3/2}}\bigr| & \leq & \|g^n\|_{H^{-3/2}}\;\|(\Id-\Delta)w^n\|^{1/2}_{H^1}\;\|(\Id-\Delta)w^n\|^{1/2}_{H^2} \\
& \leq & C_\mu\,\|g^n\|^{4/3}_{H^{-3/2}}\;\|(\Id-\Delta)w^n\|^{2/3}_{H^1}\;+\;\veps\,\mu\,\|(\Id-\Delta)w^n\|^{2}_{H^2} \\
& \leq & C_\mu\,\|g^n\|^{4/3}_{H^{-3/2}}\left(1\,+\,\|(\Id-\Delta)w^n\|^{2}_{H^1}\right)\,+\,\veps\,\mu\,\|(\Id-\Delta)w^n\|^{2}_{H^2}\,,
\end{eqnarray*}
for all $\veps>0$ small enough (so that the last term can be absorbed on the left-hand side).

Next, we claim that $\bigl(w^n\bigr)_n$ is a Cauchy sequence in the space $\mc C_{\rm loc}\bigl(\R_+;H^3\bigr)\,\cap\,L^2_{\rm loc}\bigl(\R_+;H^4\bigr)$. Indeed, let us set
$\de_n^k\vphi\,:=\,\vphi^{n+k}-\vphi^n$, where the function $\vphi$ can be either $w$ or $f$ or $g$. Then, by linearity of equation \eqref{eq:w^n}, it is easy to derive
an energy estimate for $\de_n^kw$: exactly as above, one finds
\begin{eqnarray}
& & \hspace{-0.7cm}
\mc E[\de_n^kw(t)]\,+\,\mu\int^t_0\left(\|\Delta \de_n^kw\|^2_{L^2}+2\|\nabla\Delta \de_n^kw\|^2_{L^2}+\|\Delta^2\de_n^kw\|^2_{L^2}\right)\,= \label{est:E_delta} \\
& & \qquad =\,\mc E[\de_n^kw_0]\,+\,\int^t_0\lan \de_n^kf,(\Id-\Delta)\de_n^kw\ran_{H^{-2}\times H^2}\,+\,\int^t_0\lan \de_n^kg,(\Id-\Delta)\de_n^kw\ran_{H^{-3/2}\times H^{3/2}}\,. \nonumber
\end{eqnarray}
The term in $f$ can be bounded as follows (recall the computations in the proof of Proposition \ref{p:a-priori_I}):
\begin{align*}
\left|\int^t_0\lan \de_n^kf,(\Id-\Delta)\de_n^kw\ran_{H^{-2}\times H^2}\right|\,&\leq\,C_\mu\int^t_0\left\|\de_n^kf\right\|^2_{H^{-2}}\,+\,
C_\mu\int^t_0\left(\|\de_n^kw\|^2_{L^2}+\|\nabla\de_n^kw\|^2_{L^2}\right)\,+ \\
&\qquad\qquad\qquad+\,\frac{\mu}{4}\int^t_0\left(\|\Delta\de_n^kw\|^2_{L^2}+\|\nabla\Delta\de_n^kw\|^2_{L^2}+\|\Delta^2\de_n^kw\|^2_{L^2}\right)\,.
\end{align*}
As for the term in $g$, we can argue as here above to get
\begin{align*}
&\bigl|\lan\de_n^kg,(\Id-\Delta)\de_n^kw\ran_{H^{-3/2}\times H^{3/2}}\bigr|\,\leq\,C_\mu\int^t_0\left\|\de_n^kg\right\|^{4/3}_{H^{-3/2}}\left(1+\mc E[\de_n^kw]\right)\,+ \\
&\qquad\qquad+\,C_\mu\int^t_0\left(\|\de_n^kw\|^2_{L^2}+\|\nabla\de_n^kw\|^2_{L^2}\right)\,+\,
\frac{\mu}{4}\int^t_0\left(\|\Delta\de_n^kw\|^2_{L^2}+\|\nabla\Delta\de_n^kw\|^2_{L^2}+\|\Delta^2\de_n^kw\|^2_{L^2}\right)\,.
\end{align*}
Inserting these last inequalities into \eqref{est:E_delta} and applying Gronwall lemma, we are finally led to
\begin{align*}
&\mc E[\de_n^kw(t)]\,+\,\mu\int^t_0\left(\|\Delta \de_n^kw\|^2_{L^2}+2\|\nabla\Delta \de_n^kw\|^2_{L^2}+\|\Delta^2\de_n^kw\|^2_{L^2}\right)\,\leq \\
&\qquad\qquad\leq\,C_1\left(\mc E[\de_n^kw_0]\,+\,\int^t_0\left\|\de_n^kf\right\|^2_{H^{-2}}\,+\,\int^t_0\left\|\de_n^kg\right\|^{4/3}_{H^{-3/2}}\right)\,
\exp\left(C_2\,t\,+\,2\,C_2\int^t_0\left\|g\right\|^{4/3}_{H^{-3/2}}\right)\,,
\end{align*}
for some constants $C_1$ and $C_2$ just depending on $\mu$. In the end, this estimate proves our claim. In particular, $\bigl(w^n\bigr)_n$ strongly converges
in $\mc C_{\rm loc}\bigl(\R_+;H^3\bigr)\,\cap\,L^2_{\rm loc}\bigl(\R_+;H^4\bigr)$.

Hence, let us call $w$ the limit of the sequence $\bigl(w^n\bigr)_n$ in $\mc C\bigl(\R_+;H^3\bigr)\,\cap\,L^2_{\rm loc}\bigl(\R_+;H^4\bigr)$.
Thanks to these properties, it is easy to pass to the limit in the weak formulation of the linear equation \eqref{eq:w^n}. As for the energy equality,
we notice that $\bigl((\Id-\Delta)w^n\bigr)_n$ is uniformly bounded in $L^\infty_T(H^1)\,\cap\,L^2_T(H^2)\,\hookrightarrow\,L^4_T(H^{3/2})$; then,
thanks to the strong convergence of $\bigl(f^n\bigr)_n$ in $L^2_T(H^{-2})$ and of $\bigl(g^n\bigr)_n$ in $L^{4/3}_T(H^{-3/2})$, we can pass to the limit also in \eqref{est:en-eq_n}.
Finally, uniqueness can be proved by a stability estimate, which is obtained arguing exactly as for showing that $\bigl(w^n\bigr)_n$ is a Cauchy sequence
in the space $\mc C_{\rm loc}\bigl(\R_+;H^3\bigr)\,\cap\,L^2_{\rm loc}\bigl(\R_+;H^4\bigr)$.
\end{proof}

From this result, we immediately infer the energy equality stated in Theorem \ref{t:weak}.
\begin{coroll} \label{c:en-eq}
Let $r$ be a weak solution to \eqref{eq:main}, related to the initial datum $r_0\,\in\, H^3$ and the external force $f\,\in\,L^2_{\rm loc}\bigl(\R_+;H^{-2}(\R^2)\bigr)$,
which moreover satisfies the energy inequality of Proposition \ref{p:a-priori_I}.

Then $r\,\in\,\mc C\bigl(\R_+;H^{3}(\R^2)\bigr)$ and it verifies the energy equality stated in Theorem \ref{t:weak}.
\end{coroll}

\begin{proof}
Thanks to Lemma \ref{l:bilinear}, $r$ solves the parabolic equation \eqref{eq:parab}, with external force $f+g$, where this time $g\,=\,\Lambda(r,r)$. Then, Theorem \ref{t:parabolic}
implies that $r$ belongs to $\mc C\bigl(\R_+;H^3(\R^2)\bigr)$ and it satisfies the equality
\begin{eqnarray*}
& & \hspace{-0.7cm}
\mc E[r(t)]\,+\,\mu\int^t_0\left(\|\Delta r(\tau)\|^2_{L^2}+2\|\nabla\Delta r(\tau)\|^2_{L^2}+\|\Delta^2r(\tau)\|^2_{L^2}\right)\,d\tau\,= \\
& & \quad =\,\mc E[r_0]\,+\,\int^t_0\lan f(\tau),(\Id-\Delta)r(\tau)\ran_{H^{-2}\times H^2}\,d\tau\,-\,\int^t_0\lan\Lambda(r,r)(\tau),(\Id-\Delta)r(\tau)\ran_{H^{-3/2}\times H^{3/2}}\,d\tau\,.
\end{eqnarray*}
At this point, we notice that the last term on the right-hand side actually vanishes, thanks to Lemma \ref{l:bilin},
and then the energy equality is fulfilled.
\end{proof}

\subsubsection{Stability estimates and uniqueness}

In this paragraph, we prove stability estimates for solutions to equation \eqref{eq:main}. Also in this case, the analysis performed for system \eqref{eq:parab}
will be fundamental.

\begin{thm} \label{t:stab}
For $j=1,2$, let us take $r^j_0\,\in\,H^3(\R^2)$ and $f^j\,\in\,L^2_{\rm loc}\bigl(\R_+;H^{-2}(\R^2)\bigr)$, and let us denote by $r^j$ the respective
weak solutions to equation \eqref{eq:main}.

Let us define $\de r_0\,:=\,r^1_0\,-\,r^2_0$, $\de f\,:=\,f^1\,-\,f^2$ and $\de r\,:=\,r^1\,-\,r^2$. Then, there exist two positive constants $C_\mu$ and $K_\mu$,
just depending on $\mu$, such that, for all time $t>0$, one has the estimate
\begin{eqnarray*}
& & \hspace{-1cm}
\mc E[\de r(t)]\,+\,\mu\int^t_0\left(\|\Delta\de r(\tau)\|^2_{L^2}+\|\nabla\Delta\de r(\tau)\|^2_{L^2}+\|\Delta^2\de r(\tau)\|^2_{L^2}\right)\,d\tau\,\leq \\
& & \qquad\qquad\qquad\qquad
\leq\,C_\mu\left(\mc E[\de r_0]\,+\,\int^t_0\|\de f(\tau)\|^2_{H^{-2}}\,d\tau\right)\;\exp\Bigl(K_\mu\,e^{2\,\mu\,t}\,m^2(t)\Bigr)\,,
\end{eqnarray*}
where we have defined the function
$$
m(t)\,:=\,\min\left\{\|r_0^1\|^2_{H^3}\,+\,\int^t_0\|f^1(\tau)\|^2_{H^{-2}}\,d\tau\;,\;\|r_0^2\|^2_{H^3}\,+\,\int^t_0\|f^2(\tau)\|^2_{H^{-2}}\,d\tau\right\}\,.
$$
\end{thm}

\begin{proof}
We start by remarking that $\de r$ is a solution of an equation of type \eqref{eq:parab}, with initial datum $\de r_0$ and external force $\de f\,-\,\de\Lambda$, where 
we have set $\de\Lambda\,:=\,\Lambda(r^1,r^1)\,-\,\Lambda(r^2,r^2)$. Then, by Theorem \ref{t:parabolic} we gather the energy equality
\begin{eqnarray}
& &
\mc E[\de r(t)]\,+\,\mu\int^t_0\left(\|\Delta\de r(\tau)\|^2_{L^2}+2\|\nabla\Delta\de r(\tau)\|^2_{L^2}+\|\Delta^2\de r(\tau)\|^2_{L^2}\right)\,d\tau\,= \label{est:en-eq_delta} \\
& & \quad =\,\mc E[\de r_0]\,+\,\int^t_0\lan\de f(\tau),(\Id-\Delta)\de r(\tau)\ran_{H^{-2}\times H^2}\,d\tau\,-\,
\int^t_0\lan\de\Lambda(\tau),(\Id-\Delta)\de r(\tau)\ran_{H^{-3/2}\times H^{3/2}}\,d\tau\,. \nonumber
\end{eqnarray}

As usual, we can control the $\de f$ term by the quantity
\begin{equation} \label{est:delta-f}
C_\mu\int^t_0\|\de f\|^2_{H^{-2}}\,d\tau\,+\,
\frac{\mu}{4}\int^t_0\left(\|\de r\|^2_{L^2}+\|\nabla\de r\|^2_{L^2}+\|\Delta\de r\|^2_{L^2}+\|\nabla\Delta\de r\|^2_{L^2}+\|\Delta^2\de r\|^2_{L^2}\right)\,d\tau\,.
\end{equation}

Let us now focus on the term in $\de\Lambda$. By definition, we have
\begin{align*}
&\int^t_0\lan\de\Lambda,(\Id-\Delta)\de r\ran_{H^{-3/2}\times H^{3/2}}\,d\tau\,= \\
&=\,\int^t_0\int_{\R^2}\nabla^\perp(\Id-\Delta)r^1\cdot\nabla(\Id-\Delta)\de r\;\Delta^2r^1\,dx\,d\tau\,-\,
\int^t_0\int_{\R^2}\nabla^\perp(\Id-\Delta)r^2\cdot\nabla(\Id-\Delta)\de r\;\Delta^2r^2\,dx\,d\tau\,,
\end{align*}
but, using the pointwise relations $\nabla^\perp g\cdot\nabla g=0$ and $\nabla^\perp g\cdot\nabla h=-\nabla g\cdot\nabla^\perp h$, straightforward computations allow us to write
\begin{align}
\int^t_0\lan\de\Lambda,(\Id-\Delta)\de r\ran_{H^{-3/2}\times H^{3/2}}\,d\tau\,&=\,
\int^t_0\int_{\R^2}\nabla^\perp(\Id-\Delta)r^2\cdot\nabla(\Id-\Delta)r^1\;\Delta^2\de r\,dx\,d\tau \label{eq:delta_non-lin} \\
&=\int^t_0\int_{\R^2}\nabla^\perp(\Id-\Delta)r^2\cdot\nabla(\Id-\Delta)\de r\;\Delta^2\de r\,dx\,d\tau \nonumber \\
&=\,\int^t_0\int_{\R^2}\nabla^\perp(\Id-\Delta)r^1\cdot\nabla(\Id-\Delta)\de r\;\Delta^2\de r\,dx\,d\tau\,. \nonumber
\end{align}
By symmetry, without loss of generality we choose to work with the second formulation. By H\"older and Gagliardo-Nirenberg inequality we immediately infer
\begin{align*}
&\hspace{-0.7cm}
\left|\int^t_0\lan\de\Lambda,(\Id-\Delta)\de r\ran_{H^{-3/2}\times H^{3/2}}\right|\,\leq\,C\,\left\|\Delta^2\de r\right\|_{L^2}\,\left\|\nabla(\Id-\Delta)\de r\right\|_{L^4}\,
\left\|\nabla(\Id-\Delta)r^2\right\|_{L^4} \\
&\qquad\qquad\qquad\qquad\leq\,C\,\left\|\Delta^2\de r\right\|_{L^2}\,\left\|\nabla(\Id-\Delta)\de r\right\|^{1/2}_{L^2}\,\left\|\nabla^2(\Id-\Delta)\de r\right\|^{1/2}_{L^2}\,
\left\|\nabla(\Id-\Delta)r^2\right\|_{L^4}\,.
\end{align*}
Now we remark that, by the continuity of the Calder\'on-Zygmund operator $\nabla^2(-\Delta)^{-1}$ on $L^2$, we can estimate
$\left\|\nabla^2(\Id-\Delta)\de r\right\|_{L^2}\,\leq\,C\,\left\|\Delta(\Id-\Delta)\de r\right\|_{L^2}\,\leq\,C\,\|\Delta\de r\|_{H^2}$; then we are led to
\begin{align*}
\left|\int^t_0\lan\de\Lambda,(\Id-\Delta)\de r\ran_{H^{-3/2}\times H^{3/2}}\right|\,&\leq\,C\,\int^t_0\left\|\Delta\de r\right\|^{3/2}_{H^2}\,\left\|\de r\right\|^{1/2}_{H^3}\,
\left\|\nabla(\Id-\Delta)r^2\right\|_{L^4}\,d\tau \\
&\leq\,\frac{\mu}{4}\int^t_0\left\|\Delta\de r\right\|^2_{H^2}\,d\tau\,+\,C_\mu\int^t_0\left\|\de r\right\|^2_{H^3}\,\left\|\nabla(\Id-\Delta)r^2\right\|^4_{L^4}d\tau\,.
\end{align*}

We remark that $\left\|\de r\right\|^2_{H^3}$ is controlled by $\mc E[\de r]$. Hence, combining this last inequality together with \eqref{est:E_delta} and \eqref{est:delta-f} gives us
\begin{eqnarray*}
& & \mc E[\de r(t)]\,+\,\mu\int^t_0\left(\|\Delta\de r\|^2_{L^2}+\|\nabla\Delta\de r\|^2_{L^2}+\|\Delta^2\de r\|^2_{L^2}\right)\,d\tau\,\leq \\
& & \qquad\qquad\qquad\qquad
\leq\,C\left(\mc E[\de r_0]\,+\,C_\mu\int^t_0\|\de f\|^2_{H^{-2}}\,d\tau\,+\,C_\mu\int^t_0\mc E[\de r]\,\left\|\nabla(\Id-\Delta)r^2\right\|^4_{L^4}d\tau\right)\,,
\end{eqnarray*}
and and the application of Gronwall lemma allows us to conclude that
\begin{eqnarray*}
& & \mc E[\de r(t)]\,+\,\mu\int^t_0\left(\|\Delta\de r\|^2_{L^2}+\|\nabla\Delta\de r\|^2_{L^2}+\|\Delta^2\de r\|^2_{L^2}\right)\,d\tau\,\leq \\
& & \qquad\qquad\qquad\qquad
\leq\,C_\mu\left(\mc E[\de r_0]\,+\,\int^t_0\|\de f\|^2_{H^{-2}}\,d\tau\right)\;\exp\left(\int^t_0\left\|\nabla(\Id-\Delta)r^2(\tau)\right\|^4_{L^4}d\tau\right)\,.
\end{eqnarray*}
Obviously, working on the third formulation in \eqref{eq:delta_non-lin}, one gets exactly the same estimate, up to replace $\left\|\nabla(\Id-\Delta)r^2\right\|^4_{L^4_tL^4}$ with
$\left\|\nabla(\Id-\Delta)r^1\right\|^4_{L^4_tL^4}$ inside the exponential term, so that one can take the minimum value of the two quantities.

Finally we observe that, by Gagliardo-Nirenberg inequality and the uniform bounds provided by Proposition \ref{p:limit}, for any $j\in\{1,2\}$ one gets
$$
\int^t_0\left\|\nabla(\Id-\Delta)r^j(\tau)\right\|^4_{L^4}d\tau\,\leq\,\|r^j\|^2_{L^\infty_t(H^3)}\,\|r^j\|^2_{L^2_t(H^4)}\,\leq\,
C_\mu\,e^{2\,\mu\,t}\left(\|r_0^j\|^2_{H^3}\,+\,\int^t_0\|f^j\|^2_{H^{-2}}\,d\tau\right)^2\,,
$$
and this estimate completes the proof of the theorem.
\end{proof}

In turn, the previous theorem immediately implies uniqueness of weak solutions, so that the proof of Theorem \ref{t:weak} is finally completed.
\begin{coroll} \label{c:uniq-weak}
Weak solutions to system \eqref{eq:main_reform}, in the sense of Definition \ref{d:weak}, are uniquely determined by the initial datum $r_0\,\in\,H^{3}(\R^2)$
and the external force $f\,\in\,L^2_{\rm loc}\bigl(\R_+;H^{-2}(\R^2)\bigr)$.
\end{coroll}

\section{Propagation of higher regularities} \label{s:strong}

In the previous part, we have settle down the theory of weak solutions for equation \eqref{eq:main} under minimal smoothness assumptions on the initial data and forcing terms.
In this section, we aim at investigating more regular solutions. Of course, from Theorem \ref{t:weak} one gathers for free existence and uniqueness of weak solutions
at the $H^3$ level of regularity; our goal is then proving propagation of the initial smoothness.

We start by discussing, in Subsecton \ref{ss:h-4}, solutions corresponding to $H^4$ initial data and $L^2_T(H^{-1})$ external forces.
They deserve special attention, because their theory will directly follow from the energy estimates of second kind, which we have presented in Subsection \ref{ss:energy}.

Subsections \ref{ss:higher} and \ref{ss:intermediate} are devoted to propagation of $H^s$ regularities, respectively for $s>1$ and $0<s<1$.
The method of the proof is analogous to both cases: after a paralinearization of our equation, we will perform energy estimates for the solution in each dyadic block of a Littlewood-Paley
decomposition, estimating carefully the remainder terms which arise from the paralinearization.
As already explained in the Introduction, in the former case, this strategy will allow us to recover fundamental algebraic cancellations in the non-linear term,
which we have already exploited in a crucial way for the basic energy estimates of Subsection \ref{ss:energy}. In the latter case, instead, these cancellations do not involve anymore the highest
order term, thus they are only of partial help: propagation of intermediate smoothness requires a more delicate analysis.

\subsection{On the $H^4$-theory for weak solutions} \label{ss:h-4}
In this subsection we discuss briefly the case when $r_0\in H^4$ and $f\in L^2_T(H^{-1})$.
One can prove the following statement, analogous to Theorem \ref{t:weak}.
\begin{thm} \label{t:weak_2}
For all $r_0\in H^4(\R^2)$ and $f\in L^2_{\rm loc}\bigl(\R_+;H^{-1}(\R^2)\bigr)$,
there exists a unique global in time weak solution $r$ to equation \eqref{eq:main}, such that
$$
r\,\in\,\mc C\bigl(\R_+;H^4(\R^2)\bigr)\,\cap\,L^\infty_{\rm loc}\bigl(\R_+;H^4(\R^2)\bigr)\,\cap\,L^2_{\rm loc}\bigl(\R_+;H^5(\R^2)\bigr)\,.
$$
Moreover, for any $T>0$ fixed, $r$ satisfies the following energy equality, for all $t\in[0,T]$:
\begin{eqnarray*}
& & \hspace{-1cm}
\wtilde{\mc E}[r(t)]\,+\,\mu\int^t_0\left(\|\Delta r(\tau)\|^2_{L^2}+2\|\nabla\Delta r(\tau)\|^2_{L^2}+2\|\Delta^2r(\tau)\|^2_{L^2}+\|\nabla\Delta^2r(\tau)\|^2_{L^2}\right)\,d\tau\,= \\
& & \qquad\qquad\qquad\qquad\qquad\qquad\qquad
=\,\wtilde{\mc E}[r_0]\,+\,\int^t_0\lan f(\tau),(\Id-\Delta+\Delta^2)r(\tau)\ran_{H^{-1}\times H^1}\,d\tau\,,
\end{eqnarray*}
where, for all functions $\vphi\in H^4$, we have defined $\wtilde{\mc E}[\vphi]$ to be the quantity
$$
\wtilde{\mc E}[\vphi]\,:=\,\bigl(\|\vphi\|^2_{L^2}\,+\,2\|\nabla \vphi\|^2_{L^2}\,+\,3\|\Delta \vphi\|^2_{L^2}\,+\,2\|\nabla\Delta \vphi\|^2_{L^2}\,+\,\|\Delta^2 \vphi\|^2_{L^2}\bigr)/2\,.
$$
\end{thm}

The proof of the previous statement follows exactly the main lines of the arguments given in Subsections \ref{ss:weak-ex} and \ref{ss:weak-u} above. So, here we will limit ourselves to
pointing out the main differences and crucial points.

We start by remarking that, thanks to the additional regularity of $r$, one can give a weak formulation of equation \eqref{eq:main} which is slightly different from the one proposed in
Definition \ref{d:weak}: namely, for any $\phi\in\mc{C}_0^\infty\bigl([0,T[\,\times\R^2\bigr)$ one has
\begin{eqnarray*}
& & \hspace{-0.7cm}
-\int^T_0\!\!\!\int_{\R^2}\bigl(\Id-\Delta+\Delta^2\bigr)r\;\d_t\phi\,dx\,dt\,-\,
\int^T_0\!\!\!\int_{\R^2}\Delta^2r\;\nabla^\perp(\Id-\Delta)r\cdot\nabla\phi\,dx\,dt\,- \\
& & \hspace{-0.3cm}
-\,\mu\int^T_0\!\!\!\int_{\R^2}\nabla\Delta(\Id-\Delta)r\cdot\nabla\phi\,dx\,dt\,=\,\int^T_0\!\!\!\langle f(t),\phi(t)\rangle_{H^{-1}\times H^1}\,dt\,+\,
\int_{\R^2}(\Id-\Delta+\Delta^2)r_0\,\phi\,dx\,.
\end{eqnarray*}

\medbreak
Let us come back to the proof of Theorem \ref{t:weak_2}.
First of all, the construction of smooth approximate solutions $\bigl(r^n\bigr)_n$ is absolutely analogous to the one given in Paragraph \ref{sss:approx}. At this point, however,
one uses Proposition \ref{p:a-priori_II} to get more precise uniform bounds on the family $\bigl(r^n\bigr)_n$. In particular, we derive that the limit point
$r$ of this sequence verifies
$$
r\,\in\,L_{\rm loc}^\infty\bigl(\RR_+;H^4(\R^2)\bigr)\,\cap\, L^2_{\rm loc}\bigl(\RR_+; H^5(\R^2)\bigr)\,,
$$
and the weak-$*$ convergence holds in this space.
On the other hand, from Lemma \ref{l:bilin-high} and Sobolev embeddings, one easily gathers that $\bigl(\Lambda(r^n,r^n)\bigr)_n$ is uniformly bounded in e.g. $L^4_T(H^{-1})$.
Using this property and going along the guidelines of the proof to Lemma \ref{l:cpt}, it is easy to prove compactness of $\bigl(r^n\bigr)_n$ in e.g. $L^2_T(H^{4}_{\rm loc})$,
and then to pass to the limit in the weak formulation of \eqref{eq:main}.
In addition, taking advantage once again of the regularity of the non-linear term, one can prove that $(\Id-\Delta+\Delta^2)r$ belongs to $\mc C^{0,1/4}\bigl(\R_+;H^{-1}(\R^2)\bigr)$.

Energy equality and uniqueness are also in this case consequences of the analysis of the parabolic equation \eqref{eq:parab}. Performing energy estimates at the $H^4$ level
(namely, testing against the function $(\Id-\Delta+\Delta^2)w$), one can easily get the analogue of Theorem \ref{t:parabolic}. From this fact, combined with the property
$\Lambda(r,r)\,\in\,L^2_T(H^{-1})$ (keep in mind Lemma \ref{l:bilin-high} again), we deduce that our weak solution $r$ belongs actually to $\mc C\bigl(\R_+;H^4\bigr)$ and
it verifies the energy equality stated in Theorem \ref{t:weak_2}. Finally, the proof of the stability estimates, and then of uniqueness, is absolutely analogous to the previous one.

As a last comment, we notice that the energy equality of Theorem \ref{t:weak_2} easily implies the estimate claimed in Theorem \ref{t:higher-reg} for the case $s=1$
(keep in mind also Proposition \ref{p:a-priori_II}).

\medbreak
We conclude this part by remarking that $H^4$ regularity seems to us the minimal one required to give sense to the inviscid problem, namely equation \eqref{eq:main} with $\mu=0$.
The study of the inviscid case goes beyond the scopes of the present paper; however, for the sake of completeness we give the following statement.
\begin{thm} \label{t:inviscid}
Let $r_0\,\in\,H^4(\R^2)$ and $f\,\in\,L^2_{\rm loc}\bigl(\R_+;L^2(\R^2)\bigr)$, and fix $\mu=0$ in equation \eqref{eq:main}.

Then there exists a global in time weak solution $r\,\in\,L^\infty\bigl(\R_+;H^{4}(\R^2)\bigr)$ to that equation, which moreover verifies
the energy estimate
$$
\|r(t)\|_{H^4}\,\leq\,C\,\left(\|r_0\|_{H^4}\,+\,\int^t_0\left\|f(\tau)\right\|_{L^2}\,d\tau\right)\,,
$$
for a universal constant $C>0$ and for all $t>0$.
\end{thm}

The previous statement can be easily derived from the analysis carried out here. Its proof uses the following ingredients:
\begin{enumerate}[(i)]
 \item construction of a sequence of smooth approximate solutions $\bigl(r^n\bigr)_n$, as done in Paragraph \ref{sss:approx};
 \item proof of uniform bounds for those approximate solutions, priveded by Proposition \ref{p:a-priori_II}, where one takes $\mu=0$
 (notice that, since we can take the $L^2$ norm of $f$, we have no factor $\mu^{-1}$ in front of it);
\item gain of compactness properties for $\bigl(r^n\bigr)_n$, by an inspection of the equation (arguing as in Lemma \ref{l:cpt}, one can deduce e.g. that $\bigl(r^n\bigr)_n$
is compact in the space $L^\infty_T(H^{3-\veps})$ for all $\veps>0$ arbitrarily small);
\item passing to the limit in the weak formulation of the equation.
\end{enumerate}

We will not pursue the study of the inviscid case further here.

\subsection{Higher order energy estimates} \label{ss:higher}

This section is devoted to the proof of Theorem \ref{t:higher-reg} when $s>1$. As explained at the beginning of this section, the main point is to establish higher order \textsl{a priori} estimates.
Indeed, once \textsl{a priori} bounds are obtained, both existence and uniqueness of solutions at this level of regularity will be straightforward consequences of the analysis carried out above.

Therefore, we just focus on energy estimates: we aim at proving the following statement.
\begin{prop} \label{p:higer-reg}
Let us assume the initial datum $r_0$ to be in $H^{4+\s}(\R^2)$ and the source term $f$ in $L^2_{\rm loc}\bigl(\RR_+;H^{\s-1}(\R^2)\bigr)$, for some $\s>0$.
Let $r$ be the solution to system \eqref{eq:main} given by Theorem \ref{t:weak_2}.
Then $r$ belongs to
$$
C(\RR_+, H^{4+\s}(\RR^2))\cap L^\infty_{\rm loc}\bigl(\R_+;H^{4+\s}(\R^2)\bigr)\,\cap\,L^2_{\rm loc}\bigl(\R_+;H^{5+\s}(\R^2)\bigr)\,.
$$ 
Mreover, there exist two positive constants $C_1$ and $C_2$, depending just on the regularity index $\s$, such that the following estimate holds true for all $t\geq0$:
\begin{align}
	&\EE_\s[r(t)]\,+\, \mu\,\int_0^t \big( \| \Delta^2 r (\tau) \|_{H^\s}^2\, +\, \| \nabla \Delta^2r (\tau) \|_{H^\s}^2\big)\,d\tau\, 
	\leq  \label{reg_prop_main_ineq} \\
	&\qquad\leq\,C_1\,
	\Big( 
		\EE_\s[r_0]\, +\, \frac{1}{\mu}\, \int_0^t \| f (\tau) \|_{H^{\s-1}}^2\,d\tau
	\Big)\;
	\exp
	\left\{
		\frac{C_2}{\mu}\,  e^{\mu\,t}\,\Big(\| r_0\|_{H^3}^2\, +\, \frac{1}{\mu}\, \int_0^t \| f (\tau) \|_{H^{-2}}^2\,d\tau\Big)
	\right\}\,, \nonumber
\end{align}
where, for all $\s\geq0$ and functions $\varphi\in H^{4+\s}(\R^2) $, we have defined $\EE_\s[\varphi]$ by
\begin{equation*}
	\EE_\s[\varphi]\,:=\, 
		\| \Delta\varphi \|_{H^{\s}}^2\, +\, 
		\| \nabla \Delta \varphi  \|_{H^{\s}}^2\,  +\,
		\|  \Delta^2  \varphi \|_{H^{\s}}^2 \,.
\end{equation*}
\end{prop}

The proof of the above result relies on a coupling between the second-order energy law, given by Proposition \ref{p:a-priori_II}, and a fine analysis of the non-linear terms.
This latter is based on a paralinearization of the operator $\Lambda$, defined in \eqref{eq:bilin}, and a special decomposition (already used in
\cite{DeA_2017} and first introduced in \cite{Marius-Arghir}) in order to treat the remainder which arise.

\begin{rem}\label{rmk:h-reg}
Let us point out that estimate \eqref{reg_prop_main_ineq} gives a control on the high frequencies of the solution $r$. On the other hand, the control on the low frequencies
comes from Theorem \ref{t:weak_2} (see also Proposition \ref{p:a-priori_II}).
Hence, combining Propositions \ref{p:higer-reg} and \ref{p:a-priori_II}, one immediately gets:
\begin{align*}
	&\| 							r(t) 	\|_{L^2}^2\, +\,
	\| 				\Delta 	r(t) 	\|_{H^{2+\s}}^2\,+\,
	\mu\, \int_0^t 
		\| \nabla \Delta^2 r(\tau) \|_{H^\s}^2\,
	d\tau \,\leq \\
	&\qquad\leq\,C_1\, e^{\mu t}\,
	\Big(
		\| r_0 \|_{H^{\s+4}}^2\,
		+\, \frac{1}{\mu}\,\int_0^t \| f(\tau) \|^2_{H^{\s-1}}\,d\tau	
	\Big)\;
	\exp
	\Big\{
		C_2\,  e^{\mu\,t}\,\Big(\| r_0\|_{H^3}^2\,+\, \frac{1}{\mu}\, \int_0^t \| f (\tau) \|_{H^{-2}}^2\,d\tau
	\Big)\Big\}\,,
\end{align*}
for two new suitable positive constants $C_1$ and $C_2$. This inequality easily leads to the one stated in Theorem \ref{t:higher-reg} for the case $s>1$
($s\,=\,1+\s$ here).

Whenever the source-term $f$ is null, one can achieve the following better estimate, for any $t\geq0$:
\begin{equation*}
	\| r(t) \|_{H^{\s+4}}^2\,+\, \mu\,\int_0^t \| \nabla \Delta^2 r(\tau) \|_{H^\s}^2\,\dd \tau\, \leq\,  
	C_1\,\| r_0 \|_{H^{\s+4}}^2\;e^{C_2\| r_0 \|_{H^3}^2}\,.
\end{equation*}
In this case, the solution acquires global integrability in time: namely, $r\,\in\,L^\infty\bigl(\R_+;H^{\s+4}(\R^2)\bigr)$, with $\nabla\Delta^2r\,\in\,L^2\bigl(\R_+;H^{\s}(\R^2)\bigr)$.
\end{rem}	

We are now in the  position to proving Proposition \ref{p:higer-reg}. In what follows, we will repeadetly use the isomorphism of Hilbert spaces $H^{\s}(\RR^2)\,\cong\,B_{2,2}^{\s}(\RR^2)$
(see also relation \eqref{eq:LP-Sob} in the Appendix), where $B_{2,2}^{\s}(\RR^2)$ stands for a nonhomogeneous Besov space (see Definition \ref{d:B} below).
For the sake of completeness, we postpone to Appendix \ref{app:LP} some specifics of these functional spaces, highlighting the main properties we are interested in;
we refer to that Section also for the notations we are going to use in the course of the proof.

\begin{proof}[Proof of Proposition \ref{p:higer-reg}]
Since $r_0$ belongs to $H^{4+\s}(\R^2)$, which is embedded in $H^{4}(\R^2)$, and 
$ f$ belongs to $ L^2_{\rm loc}\bigl(\RR_+;H^{\s-1}(\R^2)\bigr)\,\hookrightarrow\,L^2_{\rm loc}\bigl(\RR_+;H^{-1}(\R^2)\bigr)$, thanks to Proposition \ref{p:a-priori_II} and Theorem
\ref{t:weak_2} we can assume, as a starting point, that
\begin{equation*}
r\,\in\,\mc C\bigl(\R_+;H^4(\R^2)\bigr)\,\cap\,L^\infty_{\rm loc}\bigl(\R_+;H^4(\R^2)\bigr)\,\cap\,L^2_{\rm loc}\bigl(\R_+;H^5(\R^2)\bigr)\,.
\end{equation*}

In virtue of the previous identification $H^{\s}(\RR^2)\,\cong\,B_{2,2}^{\s}(\RR^2)$, it is apparent that our goal is to establish
suitable estimates on the frequency localization $r_j(t):=\Delta_jr(t)$ of the solution $r$, for any $j\geq-1$.
Consequently, we separately take into consideration low and high frequencies of $r$.

More precisely, let us fix a $N\in\N$ large enough, whose value will be determined in the course of the proof (the choice $N\geq7$ will be enough for us), and let us set
\begin{equation} \label{def:low-high}
r^l(t)\,:=\,S_{N}r(t) = \sum_{-1\leq j\leq  N-1}r_j(t)\qquad\quad\mbox{ and }\qquad\quad r^h(t)\, :=\,\left(\Id - S_{N}\right)r(t)\,=\, \sum_{j\geq N}r_j(t)\,.
\end{equation}
Notice that $r(t) = r^l(t) + r^h(t)$.

Thanks to this relation and the inequality $(a+b)^2\,\leq\,2\left(a^2+b^2\right)$, it is easy to see that it is enough to estimate the quantity on the left-hand side of \eqref{reg_prop_main_ineq},
when computed for $r^l$ and $r^h$ separately.

We begin with considering the low frequency term $r^l$. By spectral localization and the fact that $S_{N}$ is a bounded operator from $L^p$ into itself, for any $p\in [1,\infty]$,
we infer that there exists a positive constant $C$, depending also on the fixed $\s$ and $N$, such that
\begin{align}
	\EE_\s\left[r^l(t)\right]\,+\,\mu\, \int_0^t& \left( \left\| \Delta^2 r^l (\tau) \right\|_{H^\s}^2 \,+\, \left\| \nabla \Delta^2 r^l (\tau) \right\|_{H^\s}^2\right)\,d\tau\,
\leq \label{ineq:low_freq}\\
	&\leq\, 
	C\,\left(	
		\EE_0\left[r^l(t)\right]\,+\,  \mu\, \int_0^t \left( \left\| \Delta^2 r^l (\tau) \right\|_{L^2}^2\, +\, \left\| \nabla \Delta^2 r^l (\tau) \right\|_{L^2}^2\right)\,d \tau
	\right)\nonumber\\
	&\leq\, C\,\left(Y(t)\,+ \, \mu\, \int_0^t \left( \left\| \Delta^2 r (\tau)\right\|_{L^2}^2\, +\, \left\| \nabla \Delta^2 r (\tau) \right\|_{L^2}^2\right)\,d \tau\right)\nonumber\,,
\end{align}
where $Y(t)$ is the quantity defined in \eqref{def:Y}.
Thus the $H^4$-energy estimate \eqref{est:en-3}, provided by Proposition \ref{p:a-priori_II}, immediately implies
\begin{equation} \label{ineq_low_frqs}
\EE_\s\left[r^l(t)\right]\,+\,\mu\, \int_0^t\left( \left\| \Delta^2 r^l (\tau) \right\|_{H^\s}^2 \,+\, \left\| \nabla \Delta^2 r^l (\tau) \right\|_{H^\s}^2\right)\,d\tau\,\leq\,
	C\,\Big( Y_0\,+\,\frac{1}{\mu}\,\int_0^t \| f(\tau)\|^2_{H^{-1}}\,d\tau \Big)\,,
\end{equation}
where $Y_0$ is the same quantity of $Y(t)$, when computed on the initial datum $r_0$. Notice that this estimate is even better than what one needs, since its right-hand
side can be easily bounded by the right-hand side of \eqref{reg_prop_main_ineq}: indeed, one has $Y(t)\equiv\mc E_0(t)\leq\mc E_s(t)$ for all $s>0$ and all $t\geq0$.

\smallskip
Now we focus on $\EE_\s\left[r^h(t)\right]$, the term of higher-frequencies origin.  We claim that it is enough to prove the following inequality:
\begin{align}
&\EE_\s\left[r^h(t)\right]\,+\,\mu\int^t_0 
\Big(\left\|\Delta^2r^h\right\|_{H^\s}^2 \,+\,\left\|\nabla\Delta^2r^h\right\|_{H^\s}^2\Big)\,d\tau\,\leq \label{reg_prop_grnwl}  \\
&\qquad\qquad\leq\mc E_\s\left[r_0\right]\,+\,\frac{C}{\mu}\left(\int^t_0\|f\|_{H^{\s-1}}^2\,d\tau\,+\,
\int^t_0\left( \left\|\Delta r \right\|^2_{L^2}\,+\,\left\| \Delta^2 r\right\|^2_{L^2}\right)\,\mc E_\s[r]\,d\tau\right)\,, \nonumber
\end{align}
for a suitable positive constant $C$, which may depend on the fixed parameters $\s>0$ and $N\in\N$.

Indeed, suppose that \eqref{reg_prop_grnwl} holds true: then, we can sum it to \eqref{ineq_low_frqs}. By the fact that $(a+b)^2\,\leq\,2\left(a^2+b^2\right)$ and the
embedding $H^{s-1}\hookrightarrow H^{-1}$ for all $s>0$, we immediately deduce the analogue of \eqref{reg_prop_grnwl} but for the whole $r$: namely,
\begin{align}
&\EE_\s\left[r(t)\right]\,+\,\mu\int^t_0 
\Big(\left\|\Delta^2r\right\|_{H^\s}^2 \,+\,\left\|\nabla\Delta^2r\right\|_{H^\s}^2\Big)\,d\tau\,\leq \label{est:low-high} \\
&\qquad\qquad\leq\mc E_\s\left[r_0\right]\,+\,\frac{C}{\mu}\left(\int^t_0\|f\|_{H^{\s-1}}^2\,d\tau\,+\,
\int^t_0\left( \left\|\Delta r \right\|^2_{L^2}\,+\,\left\| \Delta^2 r\right\|^2_{L^2}\right)\,\mc E_\s[r]\,d\tau\right)\,. \nonumber
\end{align}
Next, we observe that, thanks to the estimate provided by Proposition \ref{p:a-priori_I}, we have
\begin{align*}
&\int_0^t\Big(
		\|\Delta 					r\|_{L^2}^2\,+\, 
		\left\|					 	\Delta^2 	r\right\|_{L^2}^2
	\Big)\,d\tau\,\leq\,
	\int_0^t \| \Delta r(\tau) \|_{H^2}^2\,d\tau\,\leq\,
	\frac{C}{\mu}\,e^{\mu\,t}\, \Big( \| r_0 \|^2_{H^3}\, +\, \frac{1}{\mu} \int_0^t \| f \|^2_{H^{-2}}\,d\tau\Big)\,.
\end{align*}
Thus, an application of Gronwall Lemma to \eqref{est:low-high}, combined with the previous inequality, immediately implies estimate \eqref{reg_prop_main_ineq}.

In the light of the previous discussion, in order to complete the proof of Proposition \ref{p:higer-reg}, it is sufficient to establish inequality \eqref{reg_prop_grnwl}.
The rest of the proof is devoted to this goal.
 
To begin with, we make a paralinearization of equation \eqref{eq:main}; namely, denoting $u\,:=\,\nabla^\perp (\Id-\Delta)r$, we expand the bilinear term
$\Lambda (r, r)$ into
\begin{equation*}
\Lambda (r,r)\, =\, u\cdot \nabla \Delta^2 r \,=\,T_{u\cdot} \nabla \Delta^2r\, +\,T'_{\nabla\Delta^2r\cdot} u\,,
\end{equation*}
where the bilinear operator $T$ is the non-homogeneous paraproduct operator, defined in \eqref{eq:T-R} in the Appendix, and we have set
\begin{equation*}
			T'_{v} u\,:=\,T_vu\,+\,R(u,v)\,=\,\sum_{q\geq -1}\Delta_{q}u\,S_{q+2}v\,.
\end{equation*}
In particular, by \eqref{eq:bony} we deduce that the equality $u\,v\,=\,T_{u} v \,+\, T'_{v}u$ holds true for all tempered distributions $u$ and $v$ for which the product
is well-defined.

Then, we apply the operator $\Delta_j$, for any index $j\geq N$, to the resulting equation: we gather that the function
$r_j = \Delta_j r$ is a smooth solution of the PDE
\begin{equation}\label{eq:r_j}
	\d_t \left(\Id-\Delta+\Delta^2\right)r_j\,+\,T_{u\cdot}\nabla\Delta^2 r_j\,+\,\mu\,\Delta^2(\Id-\Delta)r_j\,=\,\Delta_jf\,-\,\Delta_jT'_{\nabla\Delta^2r\cdot} u\,+\,
\left[T_{u\cdot}\,,\,\Delta_j\right]\nabla\Delta^2r\,.
\end{equation}

Next, we perform energy estimates of second type on the previous equation: multiplying both sides of \eqref{eq:r_j} by $\Delta^2 r_j$ and integrating over $\R^2$, we achieve
\begin{align}
&	\frac{1}{2}\,\frac{d}{dt}
	\Big[
		\|	\Delta					r_j		\|_{L^2}^2\, +\,
		\|	\nabla	\Delta		 r_j	\|_{L^2}^2\, +\,
		\|			\Delta^2	   r_j		\|_{L^2}^2
	\Big]\,+\,
	\mu \,
	\Big[
		\|			\Delta^2	r_j	\|_{L^2}^2 \,+\,
		\|	\nabla	\Delta^2	r_j	\|_{L^2}^2 
	\Big]\, 
	= \label{reg_prop_eq} \\ 
&=\,\int
	\Delta_j f\,\Delta^2 r_j\,
	-\,
	\underbrace{
	\int
		\big(T_{u \cdot} \nabla \Delta^2 r_j \big)\, 							\Delta^2 r_j}_{\I^{(j)}_1}\, +\,
	\underbrace{
	\int					
						\big[ T_{u \cdot} , \Delta_j\big]\nabla \Delta^2 r\,	\Delta^2 r_j}_{\I^{(j)}_2}\,-\,
	\underbrace{
		\int
		\Delta_jT'_{\nabla \Delta^2 r\cdot}u\, 							\Delta^2 r_j}_{\I^{(j)}_3} 			\nonumber
						.
\end{align}

The nonlinear terms of the above equality are bounded making use of Lemma \ref{lemma:hig_reg} below, whose proof is postponed at the end of this section.
\begin{lemma}\label{lemma:hig_reg}
	The following inequalities are satisfied, for some suitable constants also depending on the fixed value of $\s>0$:
	\begin{equation*}
	\begin{alignedat}{8}
	&(i)	&&\I^{(j)}_1(t)	&&&&\,\leq C\,\|\nabla u(t)\|_{L^2}\,\left\|	 \Delta^2 r_j(t)\right\|_{L^2}\,	\left\|\nabla \Delta^2 r_j(t)\right\|_{L^2}\,;\\
	&(ii)	&&\I^{(j)}_2(t)	&&&&\,\leq\,C\,\|  \nabla u(t)\|_{L^2}\,\left\|\Delta^2 r(t)\right\|_{H^\s}\,\left\|\nabla\Delta^2 r_j(t)\right\|_{L^2}\,a_{j}(t)\,2^{-j\s}\,;\\
	&(iii)\quad&&\I^{(j)}_3(t)	&&&&\,\leq\,C\,\left\|\Delta^2r\right\|_{L^2}\,\left\|\nabla u\right\|_{H^\s}\,\left\|\nabla\Delta^2 r_j\right\|_{L^2}\,\sum_{q\geq j-5}2^{-q\s}\,b_q\,.
	\end{alignedat}
	\end{equation*}
	where the sequences $(a_j(t))_{j\geq-1}$ and $(b_q(t))_{q\geq-1}$ belong to $\ell^2$ for any time $t\geq0$, and they verify the inequality
$$
\sup_{t\geq0}\left\|\bigl(a_j(t)\bigr)_{j\geq-1}\right\|_{\ell^2}\,+\,\sup_{t\geq0}\left\|\bigl(b_q(t)\bigr)_{q\geq-1}\right\|_{\ell^2}\,\leq\,C\,.
$$
\end{lemma}

Combining estimate \eqref{reg_prop_eq} and Lemma \ref{lemma:hig_reg} together, we gather
\begin{align*}
&\frac{1}{2}\,\frac{d}{dt}\,
\Big(\|	\Delta					r_j	\|_{L^2}^2\,+\,
\|	\nabla	\Delta		 r_j	\|_{L^2}^2\, +\,
\left\|			\Delta^2	   r_j	\right\|_{L^2}^2
\Big)\,+\,\mu\, 
\Big(
\left\|			\Delta^2	r_j	\right\|_{L^2}^2 \,+\,
\left\|	\nabla	\Delta^2	r_j	\right\|_{L^2}^2 
\Big)\,\leq \\ 
&\leq\,C\,2^{-j}\,\|\Delta_j f\|_{L^2}\,\left\|\nabla\Delta^2 r_j \right\|_{L^2}\,+\,C\biggl(\|\nabla u\|_{L^2}\,\left\|\Delta^2 r_j\right\|_{L^2}\,\left\|\nabla \Delta^2 r_j\right\|_{L^2}\,+ \\
&\qquad+\,\|\nabla u\|_{L^2}\,\left\|\Delta^2 r\right\|_{H^\s}\,a_{j}\,2^{-j\s}\,\left\|\nabla \Delta^2 r_j\right\|_{L^2}\,+\,
\left\|\Delta^2r\right\|_{L^2}\,\left\|\nabla u\right\|_{H^\s}\,\left\|\nabla\Delta^2 r_j\right\|_{L^2}\,\sum_{q\geq j-5}2^{-q\s}\,b_q\biggr)\,.
\end{align*}
Hence, we multiply both sides by $2^{2j\s}$, and we take the integral in time; after using Young inequality for each term on the right-hand side of the previous relation,
we perform a sum over $j\geq N$, and we finally deduce that
\begin{align*}
&\EE_\s\left[r^h(t)\right]\,+\,2\,\mu\int^t_0 
\Big(
\left\|			\Delta^2	r^h	\right\|_{L^2}^2 \,+\,
\left\|	\nabla	\Delta^2	r^h	\right\|_{L^2}^2 
\Big)\,d\tau\,\leq \\
&\qquad\leq\,\mc E_\s\left[r^h_0\right]\,+\,\frac{C}{\mu}\sum_{j\geq N}2^{2j(\s-1)}\|\Delta_j f\|_{L^2_t(L^2)}^2\,+\, \mu\sum_{j\geq N}2^{2j\s}\,\left\|\nabla\Delta^2 r_j\right\|_{L^2_t(L^2)}^2\,+ \\
&\qquad\qquad+\,\frac{C}{\mu}\int^t_0\|\nabla u\|^2_{L^2}\sum_{j\geq N}
	2^{2j\s}\,\left\|\Delta^2 r_j\right\|_{L^2}^2\,d\tau\,+\, 
\frac{C}{\mu}\int^t_0\|\nabla u\|^2_{L^2}\,\left\|\Delta^2 r\right\|^2_{H^\s}\sum_{j\geq N}a_j^2\,d\tau\,+ \\
	&\qquad\qquad\qquad\qquad\qquad\qquad\qquad\qquad
+\,\frac{C}{\mu}\int^t_0\left\|\Delta^2r\right\|^2_{L^2}\left\|\nabla u\right\|^2_{H^\s}\,\sum_{j\geq N}\left(\sum_{q\geq j-5}2^{-(q-j)\s}\,b_q\right)^{\!\!2}\,d\tau\,.
\end{align*}
Observe that, by definition of $r^h$, we can absorbe the third term in the right-hand side into the left-hand side. Moreover, recalling the properties of the sequences
$\bigl(a_j(t)\bigr)_{j\geq-1}$ and $\bigl(b_q(t)\bigr)_{q\geq-1}$ stated in Lemma \ref{lemma:hig_reg} above, we can estimate
\begin{align*}
\sum_{j\geq N}a_j(\tau)^2\,&\leq\,\left\|\bigl(a_j(\tau)\bigr)_{j\geq-1}\right\|^2_{\ell^2}\,\leq\,C \\
\sum_{j\geq N}\left(\sum_{q\geq j-5}2^{-(q-j)\s}\,b_q(\tau)\right)^{\!\!2}\,&\leq\,\left\|\bigl(\theta_k\bigr)_{k\geq-1}\,*\,\bigl(b_q(\tau)\bigr)_{q\geq-1}\right\|^2_{\ell^2} \\
&\leq\,\left\|\bigl(\theta_k\bigr)_{k\geq-1}\right\|^2_{\ell^1}\,\left\|\bigl(b_q(\tau)\bigr)_{q\geq-1}\right\|^2_{\ell^2}\,\leq\,C\,,
\end{align*}
where, for all $k\geq-1$, we have defined $\theta_k\,=\,2^{-k\s}$, and where the constants $C$ are uniform with respect to time.
Therefore, we obtain the inequality
\begin{align}
&\EE_\s\left[r^h(t)\right]\,+\,\mu\int^t_0 
\Big(
\left\|			\Delta^2	r^h	\right\|_{L^2}^2 \,+\,
\left\|	\nabla	\Delta^2	r^h	\right\|_{L^2}^2 
\Big)\,d\tau\,\leq \label{est:E_s-I} \\
&\leq\,\mc E_\s\left[r_0\right]\,+\,\frac{C}{\mu}\int^t_0\|f\|_{H^{\s-1}}^2\,d\tau\,+\,\frac{C}{\mu}\int^t_0\|\nabla u\|^2_{L^2}\,\left\|\Delta^2 r\right\|^2_{H^\s}\,d\tau\,+\,
\frac{C}{\mu}\,\int^t_0\left\|\Delta^2r\right\|^2_{L^2}\left\|\nabla u\right\|^2_{H^\s}\,d\tau\,, \nonumber
\end{align}
where we used also the fact that $(\Id - S_{N})$ is a bounded operator from $L^p$ into itself, so that $\EE_\s\left[r^h_0\right]$ is actually bounded by $\EE_\s\left[r_0\right]$,
up to a multiplicative constant.

At this point, since $u\,=\,\nabla^\perp(\Id-\Delta)r$, we use the continuity of the Calder\'on-Zygmund operator $\nabla^2(-\Delta)^{-1}$ over $L^p$ for all $1<p<+\infty$,
to deduce
\begin{align*}
\|\nabla u\|^2_{L^2}\,&\leq\,C\,\left( \left\|\Delta r \right\|^2_{L^2}\, +\, \left\| \Delta^2 r\right\|^2_{L^2}\right) \\
\|\nabla u\|^2_{H^\s}\,&\leq\,C\,\left( \left\|\Delta r \right\|^2_{H^\s}\, +\, \left\| \Delta^2 r\right\|^2_{H^\s}\right)\,\leq\,\mc E_\s[r]\,.
\end{align*}
Therefore, from inequality \eqref{est:E_s-I} we immediately obtain, for all $t\geq0$, the bound
\begin{align*}
&\EE_\s\left[r^h(t)\right]\,+\,\mu\int^t_0 
\Big(\left\|\Delta^2r^h\right\|_{L^2}^2 \,+\,\left\|\nabla\Delta^2r^h\right\|_{L^2}^2\Big)\,d\tau\,\leq \\ 
&\qquad\qquad\leq\mc E_\s\left[r_0\right]\,+\,\frac{C}{\mu}\left(\int^t_0\|f\|_{H^{\s-1}}^2\,d\tau\,+\,
\int^t_0\left( \left\|\Delta r \right\|^2_{L^2}\,+\,\left\| \Delta^2 r\right\|^2_{L^2}\right)\,\mc E_\s[r]\,d\tau\right)\,,
\end{align*}
which corresponds exactly to inequality \eqref{reg_prop_grnwl}.
Therefore, as already explained above, this relation concludes the proof of Proposition \ref{p:higer-reg}, provided we show the bounds of Lemma \ref{lemma:hig_reg}.
\end{proof}

In order to complete the proof to Proposition \ref{p:higer-reg}, it remains us to get the estimates of Lemma \ref{lemma:hig_reg}.
\begin{proof}[Proof of Lemma \ref{lemma:hig_reg}]
We begin with proving inequality $(i)$.  First, using the almost-orthogonality property of the dyadic blocks, namely the fact that $\Delta_q\Delta_j\equiv 0$ for any couple $(q,j)$
such that $|q-j|\geq2$ (see e.g. Proposition 2.10 of \cite{B-C-D}), we can reformulate the term $\I^{(j)}_1$ as follows:
\begin{align*}
	\I_1^{(j)}\,=\, 
	\int_{\R^2}
		\big(T_{u \cdot} \nabla \Delta^2 r_j \big)\, 							\Delta^2 r_j \,d x 
	\,&=\,
	\sum_{q\geq -1} \int_{\R^2}
						\big( S_{q-1} u \cdot
						\Delta_q  \nabla \Delta^2 r_j\big)\,\Delta^2 r_j\,dx\\
	&=\sum_{q\geq -1,\, |q-j|\leq 1} \int_{\R^2}
						\big( S_{q-1} u \cdot
						\Delta_q  \nabla \Delta^2 r_j\big)\,\Delta^2 r_j\,dx\,.
\end{align*}
Next, we isolate the localized function $S_{j-1}u$ from the above term: namely, we split $\I_1^{(j)}$ into
\begin{equation} \label{eq:I_1}
	\I_1^{(j)}\,=\,
	\sum_{\substack{q\geq -1,\\|q-j|\leq 1}}\hspace{-0.1cm} \int
						\big( (S_{q-1} - S_{j-1}) u \cdot
						\Delta_q  \nabla \Delta^2 r_j\big)\,\Delta^2 r_j\, +\,\int
						\big( S_{j-1} u \cdot
						\nabla \Delta^2 r_j\big)\,\Delta^2 r_j\,, 					
\end{equation} 
where we have used also the fact that $\sum_{|q-j|\leq 1} \Delta_q r_j = r_j$, i.e. $\sum_{|q-j|\leq 1} \Delta_q\,\equiv\,1$ on the support of $\Delta_j$.
Such a decomposition has two fundamental consequencies in our estimates.
First of all, since $S_{j-1} u \,=\, \nabla^\perp(\Id-\Delta)S_{j-1}r$ is divergence-free, by an
integration by parts we gather that the last non-linear term in equality \eqref{eq:I_1} is null:
\begin{equation} \label{eq:paraprod-0}
	\int_{\R^2}
	S_{j-1} u \cdot
	\nabla \Delta^2 r_j\,\Delta^2 r_j\,dx\, =\, 0\,.
\end{equation}
We remark that this cancellation, which can be viewed as the analogue of Lemma \ref{l:bilin-high} in a frequency localization context, plays a main rule: indeed it allows one to get a
better estimate in the case $s>1$ with respect to the case $0<s<1$, see the statement of Theorem \ref{t:higher-reg}.
Therefore, $\I_1^{(j)}$ reduces to
\begin{align*}
	\I_1^{(j)}\,
	&=\,
	\sum_{\substack{|q-j|\leq 1}} \int
						\big( (S_{q-1} - S_{j-1}) u \cdot
						\Delta_q  \nabla \Delta^2 r_j\big)\,\Delta^2 r_j
	\\
	&\leq\,
	\sum_{\substack{|q-j|\leq 1}}
	\left\|	(S_{q-1} - S_{j-1}) u 			\right\|_{L^\infty}\,
	\left\|	\Delta_q  \nabla \Delta^2 r_j	\right\|_{L^2}\,
	\left\|		\Delta^2 r_j					\right\|_{L^2}\,.
\end{align*}
Now we observe that (this is the second main consequence of decomposition \eqref{eq:I_1} above), if one fixes $N\in\N$ large enough (e.g. $N\geq5$ is fine at this level),
the support of the Fourier transform of the term $(S_{q-1} - S_{j-1})u$ is contained in an annulus, whose radius is proportional to $2^q$ (since $j$ and $q$ are comparable).
Hence, Bernstein inequality and the property $|q-j|\leq 1$ yield the following chain of inequalities:
$$
	\left\|	(S_{q-1} - S_{j-1}) u 			\right\|_{L^\infty}\,
	\leq\,
	C\,2^q\,
	\left\|	(S_{q-1} - S_{j-1}) u 			\right\|_{L^2}\,
	\leq\,
	C\,
	\left\|	(S_{q-1} - S_{j-1})\nabla u 		\right\|_{L^2}
	\,\leq\,
	C\,
	\|	\nabla u 		\|_{L^2}\,.
$$
Summarizing what we have done until now, we finally get inequality $(i)$:
\begin{equation*}
	\I_1^{(j)}\,\leq\,C\,
	\sum_{\substack{|q-j|\leq 1}}
	\left\|\nabla u 			\right\|_{L^2}\,
	\left\|	\Delta_q  \nabla \Delta^2 r_j	\right\|_{L^2}\,
	\left\|		\Delta^2 r_j					\right\|_{L^2}
	\,\leq\, 
	C\,\left\|	\nabla u						\right\|_{L^2}\,
	 \left\| 	 \Delta^2 r_j			\right\|_{L^2}\,
	 \left\|\nabla \Delta^2 r_j					\right\|_{L^2}\,.
\end{equation*}

Next, we take into consideration inequality $(ii)$, involving the nonlinear term $\I^{(j)}_2$. Let us immediately remark that, by spectral localization, one has the equality
\begin{equation*}
\I_2^{(j)}\,=\,\int_{\R^2} \big[ T_{u \cdot} , \Delta_j\big]\nabla \Delta^2 r\,	\Delta^2 r_j\,dx\,=\,
\int_{\R^2}\sum_{|q-j|\leq 4}\big[ S_{q-1}u\,,\,\Delta_j\big]\nabla \Delta^2 r_q\,\Delta^2 r_j\,dx\,.
\end{equation*}
Thus, H\"older  inequality implies
\begin{align*}
	\I^{(j)}_2\,
	\leq\,
	\sum_{|q-j|\leq 4}
	\left\|	\big[ S_{q-1}u\,,\,\Delta_j\big]\nabla \Delta^2 r_q	\right\|_{L^2}\,
	\left\|	\Delta^2 r_j										\right\|_{L^2}\,.	
\end{align*}
Making use of Bernstein inequality and the commutator estimate provided by Lemma \ref{l:commut}, since $j$ and $q$ are comparable, one gets
\begin{equation*}
	\left\|	\big[ S_{q-1}u\,,\,\Delta_j\big]\nabla \Delta^2 r_q	\right\|_{L^2}\,
	\leq\,
	C\,2^{-q}\,\|  S_{q-1}\nabla u\|_{L^\infty}\,\left\|\nabla \Delta^2 r_q	\right\|_{L^2}
	\,\leq \,
	C\,\|  S_{q-1}\nabla u\|_{L^2}\,\|\nabla \Delta^2 r_q	\|_{L^2}\,,
\end{equation*}
and hence one deduces (by Bernstein inequalities again)
\begin{align*}
\I_2^{(j)}\,\leq\,C
\sum_{|q-j|\leq 4}
\|  \nabla u								\|_{L^2}\,
\left\|			\Delta^2 r_q				  	\right\|_{L^2}\,
\,\left\|	\nabla	\Delta^2 r_j				  	\right\|_{L^2}\,\leq\,
C_\s\,\|  \nabla u\|_{L^2}\,\left\|\Delta^2 r				  		\right\|_{H^\s}\,
\left\|	\nabla	\Delta^2 r_j				  	\right\|_{L^2}\,a_{j}(t)\,2^{-j\s}\,,
\end{align*}
where the constant $C_s$ depends also on $s$ and the sequence $\bigl(a_j(t)\bigr)_{j\geq-1}\,\in\,\ell^2$ is defined by the relation
$$
a_j(t)\,:=\,\sum_{|q-j|\leq 4}2^{q\s}\,\left\|	\Delta^2 r_q(t)	\right\|_{L^2}/\left\|\Delta^2 r(t)\right\|_{H^\s}\qquad\mbox{ for }\;j\geq N\,,\qquad\qquad a_j(t)\,=\,0
\quad\mbox{otherwise.}
$$
It is worth to remark that
\begin{equation*}
	\bigl\| (a_j(t))_j \bigr\|^2_{\ell^2}\, \leq\, \frac{C}{\left\|\Delta^2 r(t)\right\|^2_{H^\s}}\,\sum_{j\geq N}\sum_{|q-j|\leq 4}2^{2q\s}\,\left\|\Delta^2 r_q(t)\right\|_{L^2}\,\leq\,C\,,
\end{equation*}
for any time $t\in \R_+$. Thus inequality $(ii)$ is proven.

It remains to control $\I^{(j)}_3$, the non-linear term given by the reminder $\Delta_jT'_{ \nabla \Delta^2 r\cdot}u$. Since $\div u\,=\,0$, by a repeated use of Bernstein inequalities,
one has
\begin{align*}
	\I^{(j)}_3\,&=\,\int\Delta_j T'_{ \nabla \Delta^2 r\cdot} u\,\Delta^2 r_j\,=\, 
	\sum_{q\geq j-5}
	\int\Delta_j\div\bigl( \Delta_q u\, S_{q+2}\Delta^2 r\bigr)\,\Delta^2r_j \\
&\leq\,	 	
	C\sum_{q\geq j-5} 2^j\,
	\left\|  \Delta_j \bigl( \Delta_q u\, S_{q+2}\Delta^2 r\bigr)	\right\|_{L^2}\,
	\left\|			\Delta^2 r_j			  								\right\|_{L^2}\,\leq\, 	
	C\sum_{q\geq j-5} \left\|\Delta_qu\right\|_{L^\infty}\,\left\|S_{q+2}\Delta^2r\right\|_{L^2}\,
	\left\|			\nabla\Delta^2 r_j			  								\right\|_{L^2} \\
&\leq\,C\,\left\|\Delta^2r\right\|_{L^2}\,\left\|\nabla\Delta^2 r_j\right\|_{L^2}\sum_{q\geq j-5} \left\|\Delta_q\nabla u\right\|_{L^2}\,.
\end{align*}
At this point, we remark that
$$
\sum_{q\geq j-5}\left\|\Delta_q\nabla u\right\|_{L^2}\,\leq\,\left\|\nabla u\right\|_{H^\s}\,\sum_{q\geq j-5}2^{-q\s}\,b_q\,,
$$
where the sequence $\bigl(b_q(t)\bigr)_{q\geq-1}$ is defined, for all $t\geq0$, by the formula
$$
b_q(t)\,:=\,\left\|\nabla u(t)\right\|^{-1}_{H^\s}\;2^{q\s}\,\left\|\Delta_q\nabla u(t)\right\|_{H^\s}\,.
$$
Notice that $\bigl(b_q(t)\bigr)_{q\geq-1}$ belongs to $\ell^2$ and it has unitary norm. In the end, we gather
$$
\I^{(j)}_3\,\leq\,C\,\left\|\Delta^2r\right\|_{L^2}\,\left\|\nabla u\right\|_{H^\s}\,\left\|\nabla\Delta^2 r_j\right\|_{L^2}\,\sum_{q\geq j-5}2^{-q\s}\,b_q\,.
$$
This estimate concludes the proof of inequality $(iii)$, and so of the whole Lemma.
\end{proof}

\subsection{Propagation of intermediate regularities} \label{ss:intermediate}

The central issue to be discussed in what follows is the propagation of regularities when the initial datum $r_0$ is assumed to belong to $H^{3+s}$, with $s\in\,]0,1[\,$.
This result completes the analysis started in the previous section, devoted to the propagation of higher regularities $ H^{4+s}$, with $s>0$. We aim at proving the following result.
\begin{prop}\label{p:low-reg} Let us assume the initial datum $r_0$ to be in $H^{3+s}(\RR^2)$ and the source term $f$ to be in $L^2_{\rm loc}\bigl(\RR_+; H^{s-2}(\RR^2)\bigr)$, for some
$s\in\,]0,1[\,$. Let $r$ be the solution to \eqref{eq:main}, given by Theorem \ref{t:weak}. Then 
\begin{equation*}
r\;\in\;	\mc C\bigl(\RR_+; H^{3+s}(\RR^2)\bigr)\,\cap\, L^\infty_{\rm loc}\bigl(\RR_+;H^{3+s}(\RR^2)\bigr)\,\cap\, L^2_{\rm loc}\bigl(\RR_+; H^{4+s}(\RR^2)\bigr)\,.
\end{equation*}
Moreover, there exist two positive constant $C_1$ and $C_2$, depending just on the regularity index $s$, such that the following estimate holds true for all $t\geq 0$:
\begin{align}
\tilde \EE_s[r(t)]\,
 &+\,\mu
 \int_0^t
 \left(\| \Delta r(\tau) \|_{H^s}^2+
 \| \nabla \Delta r(\tau) \|_{H^s}^2+
 \| \Delta^2 r (\tau )		  \|_{H^s}^2\right)\,d\tau\,
 \leq \label{ineq:low_prop} \\
 \leq 
 C 
 \Big( 
	&	\tilde \EE_s[r_0]\, +\, \frac{1}{\mu}\, \int_0^t \| f (\tau) \|_{H^{s-2}}^2\,d\tau
	\Big)\;
	\exp
	\left\{
		\frac{C_2}{\mu}\,  e^{2\,\mu\,t}\,\Big(\| r_0\|_{H^3}^2\, +\, \frac{1}{\mu}\, \int_0^t \| f (\tau) \|_{H^{-2}}^2\,d\tau\Big)^2
	\right\},\nonumber
 \end{align}
where, for all $s\geq0$ and all functions $\varphi\in H^{s+3}$, we have defined $\tilde \EE_s[\varphi]$ by
\begin{equation*}
	\tilde \EE_s[\varphi]\,:=\,
   	 \|							\varphi	\|_{H^s}^2\,+\,
	 \| \nabla 				\varphi	\|_{H^s}^2\,+\,
	 \| 			\Delta 	\varphi	\|_{H^s}^2\,+\,
 	 \| \nabla	\Delta 	\varphi	\|_{H^s}^2\,.
\end{equation*}
\end{prop}

\begin{rem}
The proof of Proposition \ref{p:low-reg} should be seen as a coupling technique between the proof to Proposition \ref{p:higer-reg} and  the one to Theorem \ref{t:weak},
where we introduced a first-type balance equation. This balance was achieved as an energy estimate of the continuum variable $(\Id - \Delta )r $, leading to the non-trivial
cancellation of the most challenging non-linear term:
	\begin{equation}\label{non-triv-canc}
	 	\langle\Lambda (r,\,r ),\,(\Id-\Delta)r\rangle_{L^2} =0
	\end{equation}
One should expect this cancellation to be somehow preserved also in higher regularities than $L^2(\RR^2)$. 
Moreover, once introduced the paradifferential localization used in Proposition \ref{p:higer-reg}, it is natural to believe that a similar approach can be performed also in the case of
low regularities, treated in Proposition \ref{p:low-reg}. 

Eventually this claim results to be false, since the structure of the system does not allow any sort of preservation of the identity \eqref{non-triv-canc}, whenever $L^2$ is replaced by
$H^s$ with $s>0$. More specifically, Bony's decomposition, that was successful for Proposition \ref{p:low-reg}, here fails when used to cancel
the corresponding term in \eqref{eq:paraprod-0}: namely, one has
	$\langle \Lambda (S_{q-1} r,\, \Delta_q r),\, (\Id - \Delta)\Delta_q r\rangle_{L^2}\, \neq\, 0$.
Moreover, we notice that, by symmetry, one can write
$$
\Lambda(\rho,\z)\,=\,\nabla^\perp(\Id-\Delta)\rho\cdot\nabla\Delta^2\z\,=\,-\,\nabla^\perp\Delta^2\z\cdot\nabla(\Id-\Delta)\rho\,,
$$
from which we deduce the cancellation of the term $\langle \Lambda (\Delta_{q} r,\, S_{q-1} r),\, (\Id - \Delta)\Delta_q r\rangle_{L^2}$. But this relation reveals to be only of partial help,
since it does not allow to get rid of the most dangerous term (i.e. the previous one $\langle \Lambda (S_{q-1} r,\, \Delta_q r),\, (\Id - \Delta)\Delta_q r\rangle_{L^2}$,
where the high frequencies of the highest order term appear).
\end{rem}

The previous remark, although being quite technical, represents one of the main challenges when propagating $H^{3+s}(\RR^2)$ regularity, with $s\in\,]0,1[\,$: we cannot rely on our previous
energy balances and we need to independently estimate each non-linear term. This fact is deeply analyzed in Lemma \ref{lemma:low_reg} below, whose proof is postponed at the end of this section.

\begin{proof} With the notations introduced in \eqref{def:low-high}, we decompose once again $r\,=\,r^l\, +\, r^h$.
In the same spirit of the proof of Proposition \ref{p:higer-reg}, it is enough to separately estimate the quantity on the left-hand side of \eqref{ineq:low_prop} when computed in $r^l$ and in $r^h$.

Similarly as for proving inequality \eqref{ineq:low_freq} in the previous subsection, one can control the low-frequencies $r^l$ through
\begin{align}
\tilde \EE_s[r^l(t)]\, &+\, \mu\int_0^t \left(\| \Delta r^l(\tau)\|_{H^s}^2 + 2\|\nabla \Delta r^l(\tau)\|_{H^s}^2 +\| \Delta^2 r^l(\tau)\|_{H^s}^2\right)\,d\tau\,\leq \label{ineq-boh} \\
&\leq\,C\left(\tilde \EE_0[r^l(t)]\,+\,\mu\int_0^t \left(\| \Delta r^l(\tau)\|_{L^2}^2 + 2\|\nabla \Delta r^l(\tau)\|_{L^2}^2 +\| \Delta^2 r^l(\tau)\|_{L^2}^2\right)\,d\tau\right) \nonumber \\
&\leq\,C\left( \tilde \EE_0[r_0]\, +\, \frac{C}{\mu}\int_0^t \| f(\tau) \|^2_{H^{-2}}\,d\tau\right)\,, \nonumber
\end{align}
and this last quantity is clearly controlled by the right-hand side of \eqref{ineq:low_reg1}.
Thus, we can focus on the control of $\tilde \EE_s[r^h(t)]$, the term of higher-frequencies origin.

We start by multiplying equation \eqref{eq:r_j} by $(\Id-\Delta)r_j$ and integrating in $\R^2$: we gather 
\begin{align}	
&	\frac{1}{2}\,\frac{d}{dt}
	\Big(
		\|					r_j		\|_{L^2}^2+
		2\|	\nabla	 r_j	\|_{L^2}^2+
		2\|	\Delta	   r_j		\|_{L^2}^2+
		\|	\nabla \Delta	   r_j		\|_{L^2}^2
	\Big)\,+  \label{ineq:low_reg1} \\
&	+\,\mu \Big(
		\|			\Delta	r_j	\|_{L^2}^2 +
		2\|	\nabla	\Delta	r_j	\|_{L^2}^2 +
		\|\Delta^2	r_j	\|_{L^2}^2 
	\Big)\,=\,\lan\Delta_jf\,,\,(\Id-\Delta)r_j\ran_{H^{-2}\times H^2}\,-  \nonumber \\
&	-
	\underbrace{
	\int_{\RR^2}
		\big(T_{u \cdot} \nabla \Delta^2 r_j \big)\, 							(\Id-\Delta)r_j}_{\J^{(j)}_1}\, +\,
	\underbrace{
	\int_{\RR^2}					
						\big[ T_{u \cdot} , \Delta_j\big]\nabla \Delta^2 r\,	(\Id-\Delta)r_j}_{\J^{(j)}_2}\,-\,
	\underbrace{
		\int_{\RR^2}
		\Delta_jT'_{\nabla \Delta^2 r\cdot}u\, 							(\Id-\Delta)r_j}_{\J^{(j)}_3} 			\nonumber
						.
\end{align}
We provide now an important lemma, that allow us to control each non-linear term on the right-hand side. Its proof is postponed at the end of this section.
\begin{lemma}\label{lemma:low_reg}
	The following inequalities are satisfied for some suitable constants, which depend also on the value of $s>0$:
	\begin{equation*}
	\begin{alignedat}{8}
	&(i)	&&\J^{(j)}_1(t)	&&&&\,\leq\, C\,\big\|  u \big\|_{L^2}^{1/2}\;
	 \big\| \nabla u \big\|_{L^2}^{1/2}\;
	 \| \nabla \Delta r_j \|_{L^2}^{1/2}\;
	 \big\| (\Delta r_j\,,\,  \Delta^2 r_j)\big\|_{L^2}^{3/2} \\
	&(ii)	&&\J^{(j)}_2(t)	&&&&\,\leq\,C\,\|  \nabla u(t)\|_{L^2}\;\left\|\nabla \Delta r(t)\right\|_{H^s}\;\left\|(\Delta r_j\,,\, \Delta^2 r_j) 	\right\|_{L^2}\,a_{j}(t)\,2^{-js} \\
	&(iii)\quad&&\J^{(j)}_3(t)	&&&&\,\leq\,C\,\left\|\Delta^2r\right\|_{L^2}\;\left\| u\right\|_{H^s}\;\left\|(\Delta r_j\,,\, \Delta^2 r_j) \right\|_{L^2}\,\sum_{q\geq j-5}2^{-qs}\,b_q\,,
	\end{alignedat}
	\end{equation*}
	where we clarify the notation $\| (f,\,g) \|_X = \| f \|_X + \| g \|_X$, for any elements $f$, $g$ in a normed space $X$. 
	The sequences $(a_j(t))_{j\geq -1}$ and $(b_j(t))_{j\geq -1}$ belong to $\ell^2$ for any time $t>0$, and they fulfill the 
	inequality
	\begin{equation*}
		\sup_{t\geq 0} \left\|\bigl(a_j(t)\bigr)_{j\geq -1} \right\|_{\ell^2} +
		\sup_{t\geq 0} \left\|\bigl(b_j(t)\bigr)_{j\geq -1} \right\|_{\ell^2}
		\leq C,
	\end{equation*}
	for a suitable positive constant $C$.
\end{lemma} 
Combining estimate \eqref{ineq:low_reg1} together with Lemma \ref{lemma:low_reg}, we obtain
\begin{align*}
&\frac{1}{2}\frac{d}{d t}
	\Big(
		\|					r_j		\|_{L^2}^2 +
		2\|	\nabla	 r_j	\|_{L^2}^2+
		2\|	\Delta	   r_j		\|_{L^2}^2+
		\|	\nabla \Delta	   r_j		\|_{L^2}^2
	\Big)+
	\mu 
	\Big(
		\|			\Delta	r_j	\|_{L^2}^2 +
		2\|	\nabla	\Delta	r_j	\|_{L^2}^2 +
		\|\Delta^2	r_j	\|_{L^2}^2 
	\Big)\leq \\
& \leq 
 C\biggl(2^{-2j} \| \Delta_j f \|_{L^2}\| (\Delta r_j,\Delta^2r_j) \|_{L^2}+ 
\big\|  u \big\|_{L^2}^{1/2}
	 \big\| \nabla u \big\|_{L^2}^{1/2}
	 \| \nabla \Delta r_j \|_{L^2}^{1/2}
	 \big\| (\Delta r_j, \Delta^2 r_j)\big\|_{L^2}^{3/2} + \\
	& +\|  \nabla u\|_{L^2}\,\left\|\nabla \Delta r\right\|_{H^s}\,\left\|(\Delta r_j,\Delta^2 r_j) 		  	\right\|_{L^2}\,a_{j}2^{-js}
	 +\left\|\Delta^2r\right\|_{L^2}\,\left\|\nabla u\right\|_{H^s}\,\left\|(\nabla r_j,\nabla \Delta r_j) 		  \right\|_{L^2}\,\sum_{q\geq j-5}2^{-qs}\,b_q\biggr).
\end{align*}
Hence, we multiply both sides by $2^{2js}$, and we take the integral in time; after using Young inequality for each term on the right-hand side of the previous relation,
we perform a sum over $j\geq N$, and we finally deduce that
\begin{align*}
&\tilde \EE_s\left[r^h(t)\right]\,+\mu\int^t_0 
\Big(
\left\|			\Delta	r^h	\right\|_{L^2}^2 \,+\,
\left\|	\nabla		\Delta	r^h	\right\|_{L^2}^2\,+\,
\left\|					\Delta^2	r^h	\right\|_{L^2}^2 
\Big)\,d\tau\,\leq \\
&\quad\leq\,\tilde{ \mc E}_s\left[r^h_0\right]\,+\,\frac{C}{\mu}\sum_{j\geq N}2^{2j(s-2)}\|\Delta_j f\|_{L^2_t(L^2)}^2\,+\,
\frac{\mu}{2}\sum_{j\geq N}2^{2js}\,\left\| (\Delta r_j, \Delta^2 r_j)\,\right\|_{L^2_t(L^2)}^2\,+ \\
&\quad+\,
\frac{C}{\mu}\int^t_0\|u \|_{L^2}^2\|\nabla u\|^2_{L^2}\left(\sum_{j\geq N}
	2^{2js}\,\left\|\nabla \Delta r_j\right\|_{L^2}^2\right)\,d\tau\,+\, 
 \\
	&\quad +\,\frac{C}{\mu}\int^t_0\|\nabla u\|^2_{L^2}\,\left\|\nabla \Delta r \right\|^2_{H^s}\left(\sum_{j\geq N}a_j^2\right)\,d\tau\,+\,
	\frac{C}{\mu}\int^t_0\left\|\Delta^2r\right\|^2_{L^2}\left\| u\right\|^2_{H^s}\,\sum_{j\geq N}\left(\sum_{q\geq j-5}2^{-(q-j)s}\,b_q\right)^{\!\!2}\,d\tau\,.
\end{align*}
Observe that, by definition of $r^h$, we can absorbe the third term in the right-hand side into the left-hand side. Moreover, recalling the properties of the sequences
$\bigl(a_j(t)\bigr)_{j\geq-1}$ and $\bigl(b_q(t)\bigr)_{q\geq-1}$ stated in Lemma \ref{lemma:low_reg} above, we can estimate
\begin{align*}
\sum_{j\geq N}a_j(\tau)^2\leq C,\quad
\sum_{j\geq N}\Big(\sum_{q\geq j-5}2^{-(q-j)s}\,b_q(\tau)\Big)^{\!\!2}\, \leq\,C\,,
\end{align*}
where the constant $C$ is uniform with respect to time. Therefore, we obtain the inequality
\begin{align}
&\tilde \EE_s\left[r^h(t)\right]\,+\,\mu\int^t_0 
\Big(
\left\|			\Delta	r^h	\right\|_{L^2}^2 \,+\,
2\left\|\nabla			\Delta	r^h	\right\|_{L^2}^2 \,+\,
\left\|	\Delta^2	r^h	\right\|_{L^2}^2 
\Big)\,d\tau\,\leq \tilde \EE_s\left[r_0\right]+\label{est:E_s-I2} \\
&+\,\frac{C}{\mu}\int^t_0\|f\|_{H^{s-2}}^2\,d\tau\,+\,\frac{C}{\mu}\int^t_0
(\| u\|^2_{L^2}+1)\|\nabla u\|^2_{L^2}\,\left\|\nabla \Delta r\right\|^2_{H^s}\,d\tau\,+\,
\frac{C}{\mu}\,\int^t_0\left\|\Delta^2r\right\|^2_{L^2}\left\| u\right\|^2_{H^s}\,d\tau\,, \nonumber
\end{align}
where we used also the fact that $(\Id - S_{N})$ is a bounded operator from $L^p$ to itself, so that $\tilde\EE_s\left[r^h_0\right]$ is actually bounded by $\tilde\EE_s\left[r_0\right]$,
up to a multiplicative constant.

At this point, since $u\,=\,\nabla^\perp(\Id-\Delta)r$, we use the continuity of the Calder\'on-Zygmund operator $\nabla^2(-\Delta)^{-1}$ over $L^p$ for all $1<p<+\infty$,
to deduce
\begin{align*}
\| u\|^2_{L^2}\,&\leq\,C\,\left( \left\|\nabla r \right\|^2_{L^2}\, +\, \left\|\nabla  \Delta r\right\|^2_{L^2}\right) \\
\|\nabla u\|^2_{L^2}\,&\leq\,C\,\left( \left\|\Delta r \right\|^2_{L^2}\, +\, \left\| \Delta^2 r\right\|^2_{L^2}\right) \\
\| u\|^2_{H^s}\,&\leq\,C\,\left( \left\|\nabla r \right\|^2_{H^s}\, +\, \left\| \nabla \Delta r\right\|^2_{H^s}\right)\,\leq\,\tilde{ \EE}_s[r]\,.
\end{align*}
Therefore, from inequality \eqref{est:E_s-I2} we immediately obtain, for all $t\geq0$, the bound
\begin{align*}
&\tilde \EE_s\left[r^h(t)\right]\,+\,\mu\int^t_0 
\Big(\left\|\Delta r^h\right\|_{H^s}^2 \,+\,2\left\|\nabla\Delta r^h\right\|_{H^s}^2 \,+\,\left\|\Delta^2r^h\right\|_{H^s}^2\Big)\,d\tau\,\leq \\ 
&\leq\tilde \EE_s\left[r_0\right]\,+\,\frac{C}{\mu}\left(\int^t_0\|f\|_{H^{s-2}}^2\,d\tau\,+\,
\int^t_0\left( 1+\left\|\nabla r \right\|^2_{L^2}\, +\, \left\|\nabla  \Delta r\right\|^2_{L^2}\right)\left( \left\|\Delta r \right\|^2_{L^2}\,+\,\left\| \Delta^2 r\right\|^2_{L^2}\right)\,\mc E_s[r]\,d\tau\right).
\end{align*}
Combining this relation  together with \eqref{ineq-boh}, we finally gather an estimate for $\tilde \EE_s[r(t)]$, namely
\begin{align*}
&\tilde \EE_s\left[r(t)\right]\,+\,\mu\int^t_0 
\Big(\left\|\Delta r\right\|_{H^s}^2 \,+\,2\left\|\nabla\Delta r\right\|_{H^s}^2 \,+\,\left\|\Delta^2r\right\|_{H^s}^2\Big)\,d\tau\,\leq \tilde \EE_s\left[r_0\right]+ \\
&\qquad\qquad\qquad+\,\frac{C}{\mu}\left(\int^t_0\|f\|_{H^{s-2}}^2\,d\tau\,+\,
\int^t_0\left( 1+\left\|\nabla r \right\|^2_{L^2}\, +\, \left\|\nabla  \Delta r\right\|^2_{L^2}\right)\,\left\|\Delta r \right\|^2_{H^2}\,\mc E_s[r]\,d\tau\right). 
\end{align*}
Then, Gronwall inequality together with Proposition \ref{p:a-priori_I} conclude the proof of Proposition \ref{p:low-reg}.
\end{proof}

It remains us to show the estimates of Lemma \ref{lemma:low_reg}.

\begin{proof}[Proof of Lemma \ref{lemma:low_reg}] We begin with estimating the term $\J_1^{(j)}(t)$. Classical H\"older's inequalities yield
\begin{equation*}
\begin{aligned}
	\J^{(j)}_1(t)\,
	&=\, \int_{\RR^2}
		\big(T_{u \cdot} \nabla \Delta^2 r_j \big)\, 							(\Id-\Delta)r_j\,d x\,\leq\, 
		\Big\|
			\sum_{|q-j|\leq 5}S_{q-1} u\, \Delta_q \nabla \Delta^2 r_j 
		\Big\|_{L^2}\,
		\Big\|
			(\Id - \Delta )r_j 
		\Big\|_{L^2}	\\
	&\leq\, 
		\sum_{|q-j|\leq 5}
		\big\| S_{q-1} u \big\|_{L^4}
		\big\| 	\Delta_q \nabla \Delta^2 r_j  \big\|_{L^4}
		\big\| ( r_j\,,\, \Delta r_j)\big\|_{L^2}.	
\end{aligned}
\end{equation*}
On the other hand, by Gagliardo-Nirenberg inequality we infer
$	\| f \|_{L^4}\,\leq\, C\, \| f \|_{L^2}^{1/2}\; \| \nabla f \|_{L^2}^{1/2}$,
for a positive constant $C$ which does not depend on $f$. Thus, we deduce
\begin{equation*}
\begin{aligned}
	\J^{(j)}_1(t)\,
	&\leq \,
	C\,\sum_{|q-j|\leq 5}	
	\big\| S_{q-1} u \big\|_{L^2}^{1/2}\,
	\big\| S_{q-1}\nabla u \big\|_{L^2}^{1/2}\,
	\big\| 	\Delta_q \nabla \Delta^2 r_j  \big\|_{L^2}^{1/2}\,
	\big\| 	\Delta_q  \Delta^3 r_j  \big\|_{L^2}^{1/2}\,
	\big\| ( r_j\,,\, \Delta r_j)\big\|_{L^2}\\
	&\leq \,
	C	\,\big\|  u \big\|_{L^2}^{1/2}\,
	\big\| \nabla u \big\|_{L^2}^{1/2}\,
	\big\| 	\nabla \Delta^2 r_j  \big\|_{L^2}^{1/2}\,
	\big\| 	 \Delta^3 r_j  \big\|_{L^2}^{1/2}\,
	\big\| ( r_j\,,\, \Delta r_j)\big\|_{L^2}\,.
\end{aligned}
\end{equation*}
Now, to handle the large number of derivatives in $\Delta^3 r_j$, we make use of the Bernstein inequalities (see Lemma \ref{l:bern} in the Appendix) to find
\begin{equation*}
	\big\| 	 \Delta^3 r_j  \big\|_{L^2}\, \leq \,C\, 2^{2j}\, \| \Delta^2 r_j \|_{L^2}\qquad
	\text{ and }\qquad
	2^j\, \big\| ( r_j\,,\, \Delta r_j)\big\|_{L^2} \,\leq \,
	C\,\big\| (\nabla r_j\,,\, \nabla \Delta r_j)\big\|_{L^2}\,,
\end{equation*}
for a positive constant which does not depend on $j$. Using the same trick also to handle the term $\nabla\Delta^2r_j$, we finally gather
\begin{equation*}
\begin{aligned}
	\J^{(j)}_1(t)\,
	&\leq \,
	C\,\big\|  u \big\|_{L^2}^{1/2}\,
	 \big\| \nabla u \big\|_{L^2}^{1/2}\,
	\big\| 	\nabla \Delta r_j  \big\|_{L^2}^{1/2}\,
	 \|  \Delta^2 r_j \|_{L^2}^{1/2}\,
	 \big\| (\Delta r_j\,,\,  \Delta^2 r_j)\big\|_{L^2}	 \\
	  &\leq\, 
	C\,\big\|  u \big\|_{L^2}^{1/2}\,
	 \big\| \nabla u \big\|_{L^2}^{1/2}\,
	 \| \nabla \Delta r_j \|_{L^2}^{1/2}\,
	 \big\| (\Delta r_j\,,\,  \Delta^2 r_j)\big\|_{L^2}^{3/2}\,.
\end{aligned}
\end{equation*}
Hence inequality $(i)$ is proven.

Next, we take into consideration inequality $(ii)$, involving the nonlinear term $\J^{(j)}_2$. We proceed similarly as for proving $(ii)$ in Lemma \ref{lemma:hig_reg}.
By spectral localization, one has the equality
\begin{equation*}
\I_2^{(j)}\,=\,\int_{\R^2} \big[ T_{u \cdot} , \Delta_j\big]\nabla \Delta^2 r\,	(\Id - \Delta) r_j\,dx\,=\,
\int_{\R^2}\sum_{|q-j|\leq 4}\big[ S_{q-1}u\,,\,\Delta_j\big]\nabla \Delta^2 r_q\,(\Id - \Delta) r_j\,dx\,.
\end{equation*}
Thus, H\"older's  inequality implies
\begin{align*}
	\J^{(j)}_2\,
	\leq\,
	\sum_{|q-j|\leq 4}
	\left\|	\big[ S_{q-1}u\,,\,\Delta_j\big]\nabla \Delta^2 r_q	\right\|_{L^2}\,
	\left\|	(\Id - \Delta) r_j										\right\|_{L^2}\,.	
\end{align*}
Making use of Bernstein inequality and the commutator estimate provided by Lemma \ref{l:commut}, since $j$ and $q$ are comparable, one gets
\begin{equation*}
	\left\|	\big[ S_{q-1}u\,,\,\Delta_j\big]\nabla \Delta^2 r_q	\right\|_{L^2}\,
	\leq\,
	C\,2^{-q}\,\|  S_{q-1}\nabla u\|_{L^\infty}\,\left\|\nabla \Delta^2 r_q	\right\|_{L^2}
	\,\leq \,
	C\,\|  S_{q-1}\nabla u\|_{L^2}\,\|\nabla \Delta^2 r_q	\|_{L^2}\,,
\end{equation*}
and hence one deduces (by Bernstein inequalities again)
\begin{align*}
\J_2^{(j)}\,&\leq\,C
\sum_{|q-j|\leq 4}
\|  \nabla u								\|_{L^2}\,
\left\|		\nabla	\Delta^2 r_q				  	\right\|_{L^2}\,
\,\left\|	(\Id - \Delta) r_j 		  	\right\|_{L^2}\\
&\leq\, C\sum_{|q-j|\leq 4}
\|  \nabla u								\|_{L^2}\,
\left\|			\Delta^2 r_q				  	\right\|_{L^2}2^q\,
\,\left\|	(r_j,\,\Delta r_j) 		  	\right\|_{L^2}\\
&\leq\,
C_s\,\|  \nabla u\|_{L^2}\,\left\|\Delta^2 r				  		\right\|_{H^s}\,
\left\|(\nabla r_j,\,\nabla \Delta r_j) 		  	\right\|_{L^2}\,a_{j}(t)\,2^{-js}\,,
\end{align*}
where the constant $C_s$ depends also on $s$, and the sequence $\bigl(a_j(t)\bigr)_{j\geq-1}\,\in\,\ell^2$ is defined by the relation
$$
a_j(t)\,:=\,\sum_{|q-j|\leq 4}2^{qs}\,\left\|	\Delta^2 r_q(t)	\right\|_{L^2}/\left\|\Delta^2 r(t)\right\|_{H^s}\qquad\mbox{ for }\;j\geq N\,,\qquad\qquad a_j(t)\,=\,0
\quad\mbox{otherwise.}
$$
As already exposed in Lemma \ref{lemma:hig_reg}, it is worth to remark that
\begin{equation*}
	\bigl\| (a_j(t))_j \bigr\|^2_{\ell^2}\, \leq\, \frac{C}{\left\|\Delta^2 r(t)\right\|^2_{H^s}}\,\sum_{j\geq N}\sum_{|q-j|\leq 4}2^{2qs}\,\left\|\Delta^2 r_q(t)\right\|_{L^2}\,\leq\,C\,,
\end{equation*}
for any time $t\in \R_+$. Thus inequality $(ii)$ is proven.

It remains to control $\J^{(j)}_3$, the non-linear term given by the reminder $\Delta_jT'_{ \nabla \Delta^2 r\cdot}u$. Since $\div u\,=\,0$, by a repeated use of Bernstein inequalities,
one has
\begin{align*}
	\J^{(j)}_3\,&=\,\int\Delta_j T'_{ \nabla \Delta^2 r\cdot} u\,(\Id-\Delta) r_j\,=\, 
	\sum_{q\geq j-5}
	\int\Delta_j\div\bigl( \Delta_q u\, S_{q+2}\Delta^2 r\bigr)\,(\Id-\Delta) r_j \\
&\leq\,	 	
	C\sum_{q\geq j-5} 2^j\,
	\left\|  \Delta_j \bigl( \Delta_q u\, S_{q+2}\Delta^2 r\bigr)	\right\|_{L^1}\,
	\left\|			(\Id-\Delta)  r_j			  								\right\|_{L^\infty}\,\\
	&\leq\, 	
	C\sum_{q\geq j-5} \left\|\Delta_qu\right\|_{L^2}\,\left\|S_{q+2}\Delta^2r\right\|_{L^2}\,
	\left\|			(\nabla r_j,\, \nabla \Delta r_j)			  								\right\|_{L^\infty} \\
&\leq\,C\,\left\|\Delta^2r\right\|_{L^2}\,\left\|(\Delta r_j,\, \Delta^2 r_j)\right\|_{L^2}\sum_{q\geq j-5} \left\|\Delta_q u\right\|_{L^2}\,.
\end{align*}
At this point, we remark that
$$
\sum_{q\geq j-5}\left\|\Delta_q u\right\|_{L^2}\,\leq\,\left\| u\right\|_{H^s}\,\sum_{q\geq j-5}2^{-qs}\,b_q\,,
$$
where the sequence $\bigl(b_q(t)\bigr)_{q\geq-1}$ is defined, for all $t\geq0$, by the formula
$$
b_q(t)\,:=\,\left\| u(t)\right\|^{-1}_{H^s}\;2^{qs}\,\left\|\Delta_q u(t)\right\|_{L^2}\,.
$$
Notice that $\bigl(b_q(t)\bigr)_{q\geq-1}$ belongs to $\ell^2$ and it has unitary norm. In the end, we gather

$$
\J^{(j)}_3\,\leq\,C\,\left\|\Delta^2r\right\|_{L^2}\,\left\| u\right\|_{H^s}\,\left\|(\Delta r_j,\, \Delta^2 r_j)\right\|_{L^2}\,\sum_{q\geq j-5}2^{-qs}\,b_q\,.
$$
This estimate concludes the proof of inequality $(iii)$, and so of the whole Lemma.
\end{proof}

\section{Long-time dynamics} \label{s:long-time}

In this section, we investigate the long-time behaviour of (smooth enough) solutions to equation \eqref{eq:main}. Not surprisingly, we prove that they converge to the solution
of the linear parabolic equation \eqref{eq:parab} with $g=0$, which we recall here for convenience:
\begin{equation}\label{eq:lim-par}
	\begin{cases}
		\partial_t\left(\Id-\Delta+\Delta^2\right)w\,+\,\mu\,\Delta^2(\Id-\Delta)w\,=\,f 	\\[1ex]
		w_{|t=0}\,=\,w_0\,,
	\end{cases}
\end{equation}
where $w_0$ and $f$ are respectively the same initial datum and external force as for \eqref{eq:main}.

\begin{rem} \label{r:parabolic}
Equation \eqref{eq:lim-par} being linear, it is easy to establish energy estimates (see also the discussion at the beginning of Paragraph \ref{sss:parabolic}) of first and
second type, and propagation of higher regularities (just test the equation on $(-\Delta)^{s}(\Id-\Delta)w$, for $s>0$, and repeat the computations as for first-type energy estimates).

For all $r_0\,\in\,H^{3+s}$ and $f\,\in\,L^2_{T}(H^{s-2})$ for all $T>0$, a straightforward adaptation of Theorem \ref{t:parabolic} (where we take $g=0$) immediately gives us
the existence and uniqueness of solutions
$$
w\;\in\;\mc{C}\bigl(\R_+;H^{3+s}(\R^2)\bigr)\,\cap\,L^\infty_{\rm loc}\bigl(\R_+;H^{3+s}(\R^2)\bigr)\,\cap\,L^2_{\rm loc}\bigl(\R_+;H^{4+s}(\R^2)\bigr)\,,
$$
which verify in addition the energy estimates
$$
\|w(t)\|^2_{H^{3+s}}\,+\,\mu\int^t_0\|\Delta w(\tau)\|^2_{H^{2+s}}\,d\tau\,\leq\,C\,e^{\mu\,t}\left(\|w_0\|^2_{H^{3+s}}\,+\,\frac{1}{\mu}\int^t_0\|f(\tau)\|^2_{H^{s-2}}\,d\tau\right)\,.
$$
\end{rem}

\begin{rem} \label{r:parab-unif}
Notice that, whenever $f\,\in\,L^2\bigl(\R_+;H^{s-2}(\R^2)\bigr)$, or $f\,\in\,L^1\bigl(\R_+;H^{s-1}(\R^2)\bigr)$, one deduces the uniform (in time) control
$$
\|w(t)\|^2_{H^{3+s}}\,+\,\int^t_0\left\|\Delta r(\tau)\right\|^2_{H^{4+s}}\,d\tau\,\leq\,C\qquad\qquad\qquad\forall\;t\geq0\,.
$$
Indeed, in the former instance this estimate is clear from the inequality in Remark \ref{r:parabolic} above, where one can avoid the presence of the exponential term arguing as in Lemma \ref{l:r-prelim}
below (see also the subsequent Remark \ref{r:r-prelim} for further comments).

In the latter case, instead, let us sketch the proof for $s=0$: the general case follows by similar considerations. Summing up relations \eqref{est:en-0} and \eqref{est:en_1}, estimating
$$
\lan f\,,\,(\Id-\Delta)r\ran_{H^{-1}\times H^1}\,\leq\,C\,\|f\|_{H^{-1}}\,\sqrt{X}
$$
and forgetting about the viscosity terms, one gets
$$
\frac{d}{dt}X\,\leq\,C\,\|f\|_{H^{-1}}\,\sqrt{X}\qquad\qquad\Longrightarrow\qquad\qquad \sqrt{X}(t)\,\leq\,\sqrt{X_0}\,+\,C\int^t_0\|f(\tau)\|_{H^{-1}}\,d\tau\,\leq\,C\,.
$$
Coming back to the resulting expression with the viscosity terms, using the previous bound and integrating in time, one finally discovers the claimed inequality.
\end{rem}

This having been pointed out, let us state the main result of this section. For reasons which will appear clear in the course of the proof, we need to consider smooth enough initial data
and forces.
\begin{thm} \label{th:long-time}
Let $s>0$. Take an initial datum $r_0\,\in\,H^{4+s}(\R^2)$ and an external force $f\,\in\,L^2\bigl(\R_+;H^{s-1}(\R^2)\bigr)$. Suppose that
$$
f\;\in\;L^\infty\bigl(\R_+;H^{s}(\R^2)\bigr)\,\cap\,L^\infty\bigl(\R_+;L^1(\R^2)\bigr)\;,\qquad
\|f(t)\|_{H^{s}\cap L^1}\,\leq\,K\,(1+t)^{-1-\eta}\,,
$$
for some constant $K>0$ and some exponent $1/2<\eta<1$, for almost every $t>0$.
Let $r$ be the solution to equation \eqref{eq:main}, given by Theorem \ref{t:higher-reg}, and let $w$ be the solution to \eqref{eq:lim-par}, given by Remark \ref{r:parabolic} above.

Then, there exists a constant $C>0$ (depending only on $\mu$, $K$ and $\|r_0\|_{H^{4+s}}$) such that, for all $t\geq0$, one has
$$
\left\|r(t)\,-\,w(t)\right\|_{H^3}\,\leq\,C\;(1+t)^{1/2-\eta}\,.
$$
\end{thm}

The rest of this section is devoted to the proof of the previous theorem. In a first time (see Subsection \ref{ss:disp_lin}), we study the linear problem, establishing some decay properties
for $w$. These properties turn out to be not strong enough for proving our result; hence, we need to investigate the non-linear equation \eqref{eq:main} and to show (in Subsection \ref{ss:disp_non-lin})
dispersive estimates also for $r$. Finally, in Subsection \ref{ss:disp_end} we put these estimates together to gather the result.

\subsection{The linear problem} \label{ss:disp_lin}

Let us consider the linear problem \eqref{eq:lim-par}: we want to analyse dispersive properties for its solution $w$.

We recall that we have defined above $\mc A(D)\,:=\,\Id-\Delta+\Delta^2$, see \eqref{def:A-D}; for notational convenience, let us also set $\mc H(D)\,:=\,\Delta^2\bigl(\Id-\Delta\bigr)\mc A^{-1}(D)$,
namely $\mc A(D)$ and $\mc H(D)$ are the pseudo-differential operators associated to the Fourier multipliers
\begin{equation} \label{def:Lambda-a}
a(\xi)\,:=\,1\,+\,|\xi|^2\,+\,|\xi|^4\qquad\qquad\mbox{ and }\qquad\qquad h(\xi)\,:=\,\frac{1\,+\,|\xi|^2}{1\,+\,|\xi|^2\,+\,|\xi|^4}\;|\xi|^4\,.
\end{equation}

The first step is to apply the Fourier transform $\mc F$ to equation \eqref{eq:lim-par} and rewrite it as
$$
\d_t\what w\,+\,\mu\,h(\xi)\,\what w\,=\,\what f_0\,,
$$
where we have defined $f_0\,:=\,\mc A^{-1}(D)f$. From this relation we immediately deduce that
\begin{equation} \label{eq:Fw}
\what w(t,\xi)\,=\,e^{-\mu\, h(\xi)\,t}\,\what w_0\,+\,\int^t_0e^{-\mu\, h(\xi) \,(t-\tau)}\,\what f_0(\tau,\xi)\,d\tau\,.
\end{equation}

We first study the evolution through the free propagator, namely the case $f=0$. From this analysis, we will then deduce bounds for the term related to the external force
in a standard way.

\medbreak
Hence, let us suppose $f=0$ (and then $f_0=0$) for a while.
We begin with presenting an intermediate estimate, which we include it for the sake of completeness.
\begin{lemma} \label{l:w_0-prelim}
Let $w$ solve equation \eqref{eq:lim-par} with initial datum $w_0\,\in\,H^{2}(\R^2)$ and $f\equiv0$.

Then there exists a positive constant $C>0$, just depending on $\mu>0$, such that, for all $s>0$, one has the following estimate:
$$
\|w(t)\|_{L^\infty}\,\leq\,C\;\bigl(\|w_0\|_{L^1}\,+\,\|w_0\|_{H^{1+s}}\bigr)\;(1+t)^{-1/4}\,.
$$
\end{lemma}

\begin{proof}
Recalling (see e.g. Theorem 1.19 of \cite{B-C-D}) that $\mc F^{-1}\,=\,(2\pi)^{-d}\,\check{\mc{F}}$, where $d$ is the space dimension and $\check{\mc{F}}$ is defined by the formula
$\check{\mc{F}}h(\xi)\,=\,\mc Fh(-\xi)$, by use of \eqref{eq:Fw} we can estimate
\begin{equation} \label{est:w-inf}
\|w(t)\|_{L^\infty}\,\leq\,C\,\left\|\what w(t)\right\|_{L^1}\,,\qquad\qquad\mbox{ with }\qquad \left\|\what w(t)\right\|_{L^1}\,=\,
\int_{\R^2}e^{-\mu\,h(\xi)\,t}\,\left|\what w_0(\xi)\right|\,d\xi\,.
\end{equation}
We observe that this quantity remains finite for (say) $t\in[0,1]$. Indeed, it is enough to multiply and divide the term inside the integral by the factor $(1+|\xi|^2)^{(1+s)/2}$ and use
Cauchy-Schwarz inequality, to bound it by $C\,\|w_0\|_{H^{1+s}}$.
Hence, we can restrict our attention to large times: namely in what follows we assume that $t\geq1$.

We now split the previous integral as the sum of the integral over the ball $B(0,1)$ of center $0$ and radius $1$ and the integral over the complement $\R^2\setminus B(0,1)$.
After passing in polar coordinates, straightforward computations lead us to
$$
\int_{\R^2\setminus B(0,1)}e^{-\mu\,h(\xi)\,t}\,\left|\what w_0(\xi)\right|\,d\xi\,\leq\,C\,\left\|\what w_0\right\|_{L^\infty}\,\int_{1}^{+\infty}e^{-\mu\,h(r)\,t}\,r\,dr\,\leq\,
C\,e^{-\mu\,t}\,\|w_0\|_{L^1}\,,
$$
where, with a little abuse of notation, we have set $h(r)\,=\,r^4\,(1+r^2)/(1+r^2+r^4)$.

As for the intregral over $B(0,1)$, instead, we use the fact that $(1+r^2)/(1+r^2+r^4)\,\geq\,1/2$ for $0\leq r\leq1$. Hence, passing again in polar coordinates we can estimate
$$
\int_{B(0,1)}e^{-\mu\,h(\xi)\,t}\,\left|\what w_0(\xi)\right|\,d\xi\,\leq\,C\,\left\|\what w_0\right\|_{L^\infty}\,\int_0^1e^{-\mu\,t\,r^4/2}\,dr\,,
$$
and a simple change of variables implies the final estimate
$$
\int_{B(0,1)}e^{-\mu\,h(\xi)\,t}\,\left|\what w_0(\xi)\right|\,d\xi\,\leq\,\frac{C}{\left(\mu\,t\right)^{1/4}}\,\|w_0\|_{L^1}\,.
$$

The previous inequality completes the proof of the lemma.
\end{proof}

Let us remark that the decay estimate of Lemma \ref{l:w_0-prelim} is much worse than the one expected for the classical heat equation, namely $\sim\,t^{-d/2}$. As it appears
clear from our computations, this is due to the fact that our parabolic operator vanishes at a higher order for $|\xi|\,\sim\,0$.

Nontheless, exactly due to the high smoothing effect of the parabolic operator, we can improve the decay rate if we consider first and higher order derivatives of $w$, and this with no
loss of derivatives. This is in sharp contrast with what happens in non-diffusive cases: we refer to e.g. paper \cite{Elg}, concerning the inviscid incompressible porus medium equation, where
the decay can be improved if one accepts to lose a high enough number of derivatives.
\begin{lemma} \label{l:w_decay}
Let $f\equiv0$ and $w_0\,\in\,H^{4+s}(\R^2)$, for some $s>0$. Let $w$ be a solution to system \eqref{eq:lim-par}.

Then the following estimates hold true:
\begin{align*}
\|\nabla w(t)\|_{L^\infty}\,&\leq\,C\,\Bigl(\left\|w_0\right\|_{L^2}\,+\,\min\bigl\{\|w_0\|_{L^1}\,,\,\|w_0\|_{L^2}\bigr\}\,+\,\|\nabla w_0\|_{H^{1+s}}\Bigr)\;(1+t)^{-1/2} \\
\|\nabla\Delta w(t)\|_{L^\infty}\,&\leq\,C\,\Bigl(\min\bigl\{\|w_0\|_{L^1}\,,\,\|w_0\|_{L^2}\bigr\}\,+\,\|\nabla w_0\|_{L^1}\,+\,\|\nabla\Delta w_0\|_{H^{1+s}}\Bigr)\;(1+t)^{-1}\,,
\end{align*}
for a suitable positive constant $C$, also depending on $\mu>0$.
\end{lemma}

\begin{proof}
As above, we can limit ourselves to consider the case $t\geq1$, since for $t\in[0,1]$ the quantities $\|\nabla w(t)\|_{L^\infty}$ and $\|\nabla\Delta w(t)\|_{L^\infty}$
remain uniformly bounded by (up to a multiplicative constant) $\|\nabla w_0\|_{H^{1+s}}$ and $\|\nabla\Delta w_0\|_{H^{1+s}}$ respectively.

Take then $t\ge1$. We start by considering $\nabla w$. Arguing as in the proof of Lemma \ref{l:w_0-prelim}, we can estimate
\begin{align*}
\|\nabla w(t)\|_{L^\infty}\,&\leq\,C\int_{\R^2}e^{-\mu\,h(\xi)\,t}\,|\xi|\,\left|\what w_0(\xi)\right|\,d\xi \\
&\leq\,C\left(\int_{B(0,1)}e^{-\mu\,h(\xi)\,t}\,|\xi|\,\left|\what w_0(\xi)\right|\,d\xi\,+\,\int_{\R^2\setminus B(0,1)}e^{-\mu\,h(\xi)\,t}\,|\xi|\,\left|\what w_0(\xi)\right|\,d\xi\right) \\
&\leq\,C\left(\int_{B(0,1)}e^{-2\,\mu\,h(\xi)\,t}\,|\xi|^2\,dx\right)^{1/2}\,\left\|w_0\right\|_{L^2}\,+\,C\,e^{-\mu\, t}\,\min\bigl\{\|w_0\|_{L^1}\,,\,\|w_0\|_{L^2}\bigr\}\,, 
\end{align*}
where, in order to get the last line, we have also used Cauchy-Schwarz inequality.
Passing now in polar coordinates, it is easy to bound
$$
\int_{B(0,1)}e^{-2\,\mu\,h(\xi)\,t}\,|\xi|^2\,d\xi\,\leq\,2\,\pi\int_0^{1}e^{-\mu\,t\,r^4}\,r^3\,dr\,\leq\,C\,(\mu\,t)^{-1}\,,
$$
which finally implies, for some suitable constant $C>0$, also depending on $\mu$,
$$
\|\nabla w(t)\|_{L^\infty}\,\leq\,C\,\left(\left\|w_0\right\|_{L^2}\,+\,\min\bigl\{\|w_0\|_{L^1}\,,\,\|w_0\|_{L^2}\bigr\}\right)\,(\mu\,t)^{-1/2}\,.
$$

Let us consider now the term $\nabla\Delta w$. Applying the same technique as above, we easily find
\begin{align*}
\|\nabla\Delta w(t)\|_{L^\infty}\,&\leq\,C\int_{\R^2}e^{-\mu\,h(\xi)\,t}\,|\xi|^3\,\left|\what w_0(\xi)\right|\,d\xi \\
&\leq\,C\left(\int_{B(0,1)}e^{-\mu\,h(\xi)\,t}\,|\xi|^2\,\left|\mc F(\nabla w_0)(\xi)\right|\,d\xi\,+\,\int_{\R^2\setminus B(0,1)}e^{-\mu\,h(\xi)\,t}\,|\xi|^3\,\left|\what w_0(\xi)\right|\,d\xi\right) \\
&\leq\,C\left(\int_{B(0,1)}e^{-\mu\,h(\xi)\,t}\,|\xi|^2\,dx\right)\,\left\|\nabla w_0\right\|_{L^1}\,+\,C\,e^{-\mu\,t}\,\min\bigl\{\|w_0\|_{L^1}\,,\,\|w_0\|_{L^2}\bigr\}\,,
\end{align*}
At this point, dealing with the integral in the last line is absolutely analogous to what we have done before: this allows us to arrive at the claimed inequality.
\end{proof}

\begin{rem} \label{r:w_decay}
Let us come back to the estimate for $\|\nabla\Delta w(t)\|_{L^\infty}$. Notice that, passing from the second line to the last one, we could use Cauchy-Schwarz inequality
for the former term, to obtain
$$
\|\nabla\Delta w(t)\|_{L^\infty}\,\leq\,C\left(\int_{B(0,1)}e^{-2\,\mu\,h(\xi)\,t}\,|\xi|^2\,dx\right)^{1/2}\,\left\|\nabla^2w_0\right\|_{L^2}\,+\,
C\,e^{-\mu\,t}\,\min\bigl\{\|w_0\|_{L^1}\,,\,\|w_0\|_{L^2}\bigr\}\,,
$$
and this would give rise to the bound
$$
\|\nabla\Delta w(t)\|_{L^\infty}\,\leq\,C\,\Bigl(\min\bigl\{\|w_0\|_{L^1}\,,\,\|w_0\|_{L^2}\bigr\}\,+\,\|\Delta w_0\|_{L^2}\,+\,\|\nabla\Delta w_0\|_{H^{1+s}}\Bigr)\;(1+t)^{-1/2}\,,
$$
\end{rem}

We now consider the case when the external force $f$ is not $0$.
\begin{lemma} \label{l:w-f_decay}
Let $s>0$ and $f\,\in\,L^2\bigl(\R_+;H^{s}\bigr)$. Suppose moreover that there exist $\eta\in\,]0,1[\,$ and a constant $K>0$ such that, for almost every $t\in\R_+$, one has
$$
\|f(t)\|_{H^{s}}\,\leq\,K\,(1+t)^{-1-\eta}\,.
$$
Set $f_0\,=\,\mc A^{-1}(D)f$ (as defined above) and let $g$ be defined by the formula
$$
\what g(t,\xi)\,:=\,\int^t_0e^{-\mu\,h(\xi)\,(t-\tau)}\,\what f_0(\tau,\xi)\,d\tau\,.
$$

Then there exists a constant $\wtilde C>0$ such that the following estimate holds true, for all time $t>0$:
$$
\left\|\nabla g(t)\right\|_{L^\infty}\,+\,\left\|\nabla\Delta g(t)\right\|_{L^\infty}\,\leq\,\wtilde C\,K\,(1+t)^{-1/2}\,.
$$
\end{lemma}

\begin{proof}
These estimates are deduced in a standard way from the ones of Lemma \ref{l:w_decay}. Indeed, arguing as in the proof of that result, we easily get
\begin{align*}
\left\|\nabla g(t)\right\|_{L^\infty}\,&\leq\,\int_0^t\int_{\R^2}e^{-\mu\,h(\xi)\,(t-\tau)}\,|\xi|\,\left|\what f_0(\tau,\xi)\right|\,d\xi\,d\tau \\
&\leq\,C\int^t_0\frac{1}{(1+t-\tau)^{1/2}}\;\bigl(\left\|f_0(\tau)\right\|_{L^2}\,+\,\left\|\nabla f_0(\tau)\right\|_{H^{1+s}}\bigr)\;d\tau\,.
\end{align*}
At this point, we use the fact that, since $\mc A$ is an operator of order $4$, one can control
$$
\left\|f_0(\tau)\right\|_{L^2}\,+\,\left\|\nabla f_0(\tau)\right\|_{H^{1+s}} 
\,\leq\,C\,\left\|f(\tau)\right\|_{H^{-2+s}}\,\leq\,C\,\left\|f(\tau)\right\|_{H^{s}}\,.
$$
Hence, using the dispersion hypothesis over the last quantity together with Lemma 2.4 of \cite{Elg}, we gather
$$
\left\|\nabla g(t)\right\|_{L^\infty}\,\leq\,C\,K\int^t_0\frac{1}{(1+t-\tau)^{1/2}}\;\frac{1}{(1+\tau)^{1+\eta}}\;d\tau\,\leq\,C'\,K\,(1+t)^{-1/2}\,.
$$

The proof of the estimate for $\left\|\nabla\Delta g(t)\right\|_{L^\infty}$ is absolutely analogous. The only difference is the use of Remark \ref{r:w_decay}
in order to make the $L^2$ norm of $\Delta f_0(\tau)$ appear; noticing that
$$
\left\|\Delta f_0(\tau)\right\|_{L^2}\,+\,\left\|\nabla\Delta f_0(\tau)\right\|_{H^{1+s}}\,\leq\,\left\|f(\tau)\right\|_{H^{s}}
$$
completes the proof of the estimate, and hence of the whole lemma.
\end{proof}

Putting Lemmas \ref{l:w_decay} and \ref{l:w-f_decay} together, by Duhamel's formula \eqref{eq:Fw} we finally get the next statement, which reveals decay properties for the solutions
to the linear problem \eqref{eq:lim-par}.
\begin{prop} \label{p:w_decay}
Let $s>0$. Take an initial datum $w_0\,\in\,H^{4+s}(\R^2)$ and an external force $f\,\in\,L^\infty\bigl(\R_+;H^{s}(\R^2)\bigr)$ such that
$$
\|f(t)\|_{H^{s}}\,\leq\,K\,(1+t)^{-1-\eta}\,,
$$
for some constant $K>0$ and some exponent $\eta\in\,]0,1[\,$, for almost every $t>0$. Let $w$ be the solution to system \eqref{eq:lim-par}, with
initial datum $w_0$ and forcing term $f$.

Then there exist positive constants $C_1$, depending on the viscosity $\mu$, and $K_0$, depending on the $H^{4+s}(\R^2)$-norm of $w_0$ and on $K$ appearing in the estimate for $f$,
such that, for all times $t>0$, one has the inequality
$$
\left\|\nabla w(t)\right\|_{L^\infty}\,+\,\left\|\nabla\Delta w(t)\right\|_{L^\infty}\,\leq\,C_1\,K_0\,(1+t)^{-1/2}\,.
$$
\end{prop}

\subsection{Coming back to the non-linear problem} \label{ss:disp_non-lin}

The decay rate provided by Proposition \ref{p:w_decay} for the solution to the linear problem turns out to be too weak to prove the convergence of $r$ to $w$.
So, we need to establish dispersion properties also for the solution $r$ to the non-linear equation: this is the main goal of the present section.

First of all, we need some preliminary lemmas.
\begin{lemma} \label{l:r-prelim}
Let $r_0\,\in\,H^{3}(\R^2)$ and $f\,\in\,L^2\bigl(\R_+;H^{-2}(\R^2)\bigr)\cap L^1\bigl(\R_+;H^{-2}(\R^2)\bigr)$.
Denote by $r$ the unique solution to \eqref{eq:main} related to the initial datum $r_0$ and bulk force $f$, as given by Theorem \ref{t:weak}.

Then, for all $t\geq0$, $r$ satisfies the estimate
$$
\|r(t)\|^2_{H^3}\,+\,\mu\int^t_0\|\Delta r(\tau)\|^2_{H^2}\,d\tau\,\leq\,C\,.
$$
Moreover, if $r_0\,\in\,H^{4}(\R^2)$ and $f\,\in\,L^2\bigl(\R_+;H^{-1}(\R^2)\bigr)$, then one also gets
$$
\|\Delta r(t)\|^2_{L^2}\,+\,\|\nabla\Delta r(t)\|^2_{L^2}\,+\,\|\Delta^2r(t)\|^2_{L^2}\,+\,\mu\int^t_0\|\Delta^2r(\tau)\|^2_{H^1}\,d\tau\,\leq\,C\,.
$$
\end{lemma}

\begin{proof}
Arguing as in the proof of Proposition \ref{p:a-priori_I} and using the same notations adotped therein, it is easy to obtain the following modified version of estimate \eqref{est:X}:
\begin{align}
&X(t)\,+\,\mu\int^t_0\left(\|\Delta r(\tau)\|^2_{L^2}+\|\nabla\Delta r(\tau)\|^2_{L^2}+\|\Delta^2r(\tau)\|^2_{L^2}\right)d\tau\,\leq \label{est:decay-X} \\
&\qquad\qquad\qquad\leq\,X_0\,+\,\frac{C}{\mu}\int^t_0\|f(\tau)\|^2_{H^{-2}}\,d\tau\,+\,C\int^t_0\bigl(\|r(\tau)\|_{L^2}\,+\,\|\nabla r(\tau)\|_{L^2}\bigr)\;\|f(\tau)\|_{H^{-2}}\,d\tau\,. \nonumber
\end{align}
At this point, by Young inequality we can bound
$$
\int^t_0\bigl(\|r(\tau)\|_{L^2}\,+\,\|\nabla r(\tau)\|_{L^2}\bigr)\;\|f(\tau)\|_{H^{-2}}\,d\tau\,\leq\,2\int^t_0\bigl(X(\tau)+1\bigr)\,\|f(\tau)\|_{H^{-2}}\,d\tau\,;
$$
then, forgetting the viscosity term and applying Gronwall's inequality give us
$$
\|r(t)\|_{H^3}\,\leq\,C\left(\|r_0\|_{H^3}\,+\,\|f\|^2_{L^2_t(H^{-2})}\,+\,\|f\|_{L^1_t(H^{-2})}\right)\,\exp\left(C\int^t_0\|f(\tau)\|_{H^{-2}}\,d\tau\right)\,\leq\,C\,.
$$
Now, insterting this estimate into \eqref{est:decay-X} yields, for some positive constants $C$ depending also on $\mu$,
$$ 
X(t)+\mu\int^t_0\left(\|\Delta r(\tau)\|^2_{L^2}+\|\nabla\Delta r(\tau)\|^2_{L^2}+\|\Delta^2r(\tau)\|^2_{L^2}\right)d\tau\,\leq\,
C\left(X_0+\|f\|^2_{L^2_t(H^{-2})}+\|f\|_{L^1_t(H^{-2})}\right)\,.
$$ 
The former inequality is thus proved.

As for the latter one, it is just a matter of following the computations of Proposition \ref{p:a-priori_II}, to get estimate \eqref{est:en-3}.
Noticing that, under our hypotheses, the right-hand side is uniformly bounded in time completes the proof of the second inequality, and thus of the lemma.
\end{proof}

\begin{rem} \label{r:r-prelim}
Easy modifications in the control of the forcing term (keep in mind also Remark \ref{r:parab-unif} above) show that:
\begin{itemize}
 \item the former inequality still holds true even when $f$ is not globally in $L^2(\R_+;H^{-2})$, provided we assume that $f\,\in\,L^1\bigl(\R_+;H^{-1}(\R^2)\bigr)$;
\item the latter inequality still holds true even when $f$ is not globally in $L^2(\R_+;H^{-1})$, provided we assume that $f\,\in\,L^1\bigl(\R_+;L^2(\R^2)\bigr)$.
\end{itemize}
\end{rem}

The second preliminary result which is needed in our study is the following one.
\begin{lemma} \label{l:Fr_point}
Let $r_0\,\in\,H^{3}(\R^2)$ and $f\,\in\,L^2_{\rm loc}\bigl(\R_+;H^{-2}(\R^2)\bigr)\cap L^1\bigl(\R_+;H^{-1}(\R^2)\bigr)$.
Denote by $r$ the unique solution to \eqref{eq:main} related to the initial datum $r_0$ and bulk force $f$, as given by Theorem \ref{t:weak}.

Then $r$ verifies, for almost every $t>0$ and $\xi\in\R^2$, the pointwise estimate
$$
\left|\what r(t,\xi)\right|\,\leq\,\left|\what r_0(\xi)\right|\,+\,|\xi|\,a^{-1}(\xi)\,\sqrt{t}\,\|r_0\|^2_{H^3}\,+\,\int^t_0a^{-1}(\xi)\,\left|\what f(\tau,\xi)\right|\,d\tau\,.
$$
\end{lemma}

\begin{proof}
Passing in Fourier variables, we recast the equation for $r$ in the following form:
$$
\d_t\what r\,+\,\mu\,h(\xi)\,\what r\,=\,\what f_0\,-\,a^{-1}(\xi)\,\mc F\bigl(u\cdot\nabla\Delta^2r\bigr)\,,
$$
where we have set $u\,=\,\nabla^\perp(\Id-\Delta)r$ and $f_0\,=\,\mc A^{-1}(D)f$ as before. Multiplying by $\what r$ and forgetting about the viscosity term, we get, after integration in time,
$$
\left|\what r(t,\xi)\right|\,\leq\,\left|\what r_0(\xi)\right|\,+\,\left|\what f_0\right|\,+\,a^{-1}(\xi)\,\left|\mc F\bigl(u\cdot\nabla\Delta^2r\bigr)\right|\,.
$$

Now, since $\div u\,=\,0$, we deduce
$$
\left|\mc F\bigl(u\cdot\nabla\Delta^2r\bigr)\right|\,\leq\,|\xi|\,\left\|u\;\Delta^2r\right\|_{L^1}\,\leq\,|\xi|\,\left\|u\right\|_{L^2}\;\left\|\Delta^2r\right\|_{L^2}\,\leq\,
C\,|\xi|\,\|r\|_{H^3}\,\|\Delta r\|_{H^2}\,.
$$
Plugging this estimate into the previous inequality and using Lemma \ref{l:r-prelim} allow us to arrive at the desired control.
\end{proof}

We are now ready to state and prove dispersive estimates for the solution $r$ to the non-linear problem. We will argue in a similar way as in \cite{C-Wu}
(see also the original approach of \cite{Scho_1985}, \cite{Scho_1986} and \cite{A-B-S}), paying attention to remainders coming from the different operators involved in our computations.
Notice also that our result is consistent with Theorem 3.4 of \cite{C-Wu} (our $\eta$ here plays the same role of their $\alpha$, which has to be taken equal to $1$ to match
the order of the parabolic operator).
\begin{prop} \label{p:r-decay}
Let $r_0\,\in\,H^{3}(\R^2)$ and $f\,\in\,L^2_{\rm loc}\bigl(\R_+;H^{-2}(\R^2)\bigr)\cap L^1\bigl(\R_+;H^{-1}(\R^2)\bigr)$.
Denote by $r$ the unique solution to \eqref{eq:main} related to the initial datum $r_0$ and bulk force $f$, as given by Theorem \ref{t:weak}.
Suppose moreover that there exist positive constants $K_1$ and $K_2$ and exponents $\eta\in\,]0,1[\,$ and $\g\geq0$ such that, for almost every $t>0$ and $\xi\in\R^2$, one has
$$
\|f(t)\|_{H^{-1}}\,\leq\,K_1\,(1+t)^{-1-\eta}\qquad\qquad\mbox{ and }\qquad\qquad
\left|\what f(t,\xi)\right|\,\leq\,K_2\,\left(1\,+\,|\xi|^\g\right)\,.
$$

Then there exists a constant $C$ (just depending on $\mu$, $K_1$, $K_2$ and $\|r_0\|_{H^3}$) such that, 
for all $t>0$, one gets the decay property
$$
\left\|r(t)\right\|_{H^3}\,\leq\,\frac{C}{(1+t)^{\eta/2}}\,.
$$
\end{prop}

\begin{proof}
Multiplying equation \eqref{eq:main} in $L^2$ as to get first-type energy estimates (see Proposition \ref{p:a-priori_I}), by Fourier-Plancherel theorem we get
\begin{align*}
& \frac{1}{2}\,\frac{d}{dt}\int_{\R^2}|\what r(t,\xi)|^2\,\left(1\,+\,2|\xi|^2\,+\,2|\xi|^4\,+\,|\xi|^6\right)\,d\xi\,+ \\
&\qquad\qquad+\,\mu\int_{\R^2}|\xi|^4\,|\what r(t,\xi)|^2\,\left(1\,+\,2\,|\xi|^2\,+\,|\xi|^4\right)\,d\xi\,=\,
\int_{\R^2}\what f(t,\xi)\;(1+|\xi|^2)\what r(t,\xi)\,d\xi\,.
\end{align*}
For convenience of notation, we set $\what R(t,\xi)\,:=\,|\xi|\left(2\,+\,2\,|\xi|^2\,+\,|\xi|^4\right)^{1/2}\,\what r(t,\xi)$. Then, from the previous relation
we easily deduce
\begin{align}
& \hspace{-1cm}
\frac{1}{2}\,\frac{d}{dt}\int|\what r(t,\xi)|^2\,d\xi\,+\,\frac{1}{2}\,\frac{d}{dt}\int|\what R(t,\xi)|^2\,d\xi\,+\,\frac{\mu}{2}\int|\xi|^2\,\left|\what R(t,\xi)\right|^2\,d\xi\,\leq 
\label{est:disp-r_1} \\
&\qquad\qquad\qquad\qquad\qquad
\leq\,C\,\left\|f(t)\right\|_{H^{-1}}\,\left\|r\right\|_{H^3}\,\leq\,C\,\left\|f(t)\right\|_{H^{-1}}\,\left\|r_0\right\|_{H^3}\,. \nonumber
\end{align}

Next, for $t\geq0$, we define the ball $B_t(0)\subset\R^2$ as the ball centered in $0$ and of radius $\nu(t)$, for some function $\nu(t)$ to be determined later. Then we can write
\begin{align*}
\int|\xi|^2\,\left|\what R(t,\xi)\right|^2\,d\xi\,&\geq\,\int_{\R^2\setminus B_t(0)}|\xi|^2\,\left|\what R(t,\xi)\right|^2\,d\xi \\
&\geq\,\nu^2(t)\int_{\R^2}\left|\what R(t,\xi)\right|^2\,d\xi\,-\,\nu^2(t)\int_{B_t(0)}\left|\what R(t,\xi)\right|^2\,d\xi\,;
\end{align*}
plugging this inequality into \eqref{est:disp-r_1} and multiplying everything by $2$ yield
\begin{align}
&\hspace{-1cm} \frac{d}{dt}\int|\what r(t,\xi)|^2\,d\xi\,+\,\frac{d}{dt}\int|\what R(t,\xi)|^2\,d\xi\,+\,\mu\,\nu^2(t)\int_{\R^2}\left|\what R(t,\xi)\right|^2\,d\xi\,\leq 
\label{est:disp-r_2} \\
&\qquad\qquad\qquad\qquad \leq\,C\,\left\|f(t)\right\|_{H^{-1}}\,\left\|r_0\right\|_{H^3}\,+\,\mu\,\nu^2(t)\int_{B_t(0)}\left|\what R(t,\xi)\right|^2\,d\xi\,. \nonumber
\end{align}
At this point, we notice that, after setting $\mc V(t)\,:=\,\mu\int^t_0\nu^2(\tau)\,d\tau$, we can write
$$
\frac{d}{dt}\int|\what R(t,\xi)|^2\,d\xi\,+\,\mu\,\nu^2(t)\int\left|\what R(t,\xi)\right|^2\,d\xi\,=\,e^{-\mc V(t)}\,\frac{d}{dt}\left(e^{\mc V(t)}\,\int\left|\what R(t,\xi)\right|^2\,d\xi\right)\,.
$$
Moreover we have
\begin{align*}
\frac{d}{dt}\int|\what r(t,\xi)|^2\,d\xi\,&=\,e^{-\mc V(t)}\frac{d}{dt}\left(e^{\mc V(t)}\int|\what r(t,\xi)|^2\,d\xi\right)\,-\,\mu\,\nu^2(t)\int|\what r(t,\xi)|^2\,d\xi \\
&\geq\,e^{-\mc V(t)}\frac{d}{dt}\left(e^{\mc V(t)}\int|\what r(t,\xi)|^2\,d\xi\right)\,-\,\mu\,\nu^2(t)\,\|r_0\|^2_{H^3}\,,
\end{align*}
where we have made use of Lemma \ref{l:r-prelim} in passing to the last line.
Putting all these relations into \eqref{est:disp-r_2} we gather
\begin{align}
&e^{-\mc V(t)}\frac{d}{dt}\left(e^{\mc V(t)}\,\left(\left\|r(t)\right\|^2_{L^2}\,+\,\left\|R(t)\right\|^2_{L^2}\right)\right)\,\leq\
\label{est:disp-r_3}  \\
&\qquad\qquad\qquad\qquad \leq\,C\,\left\|f(t)\right\|_{H^{-1}}\,\|r_0\|_{H^3}\,+\,\mu\,\nu^2(t)\,\|r_0\|^2_{H^3}\,+\,\mu\,\nu^2(t)\int_{B_t(0)}\left|\what R(t,\xi)\right|^2\,d\xi\,. \nonumber
\end{align}
for a new constant $C>0$ also depending on $\mu$.

Now, we pass to estimate the last term in the previous inequality: Lemma \ref{l:Fr_point} implies
\begin{align*}
\int_{B_t(0)}\left|\what R(t,\xi)\right|^2\,&=\,\int_{B_t(0)}|\xi|^2\left(2\,+\,2\,|\xi|^2\,+\,|\xi|^4\right)\,\left|\what r(t,\xi)\right|^2\,d\xi \\
&\leq\,2\,\|r_0\|^2_{H^3}\,+\,\int_{B_t(0)}|\xi|^4\,\left(2+2|\xi|^2+|\xi|^4\right)\,a^{-2}(\xi)\,t\,\|r_0\|^4_{H^3}\,d\xi\,+ \\
&\qquad\qquad\qquad\qquad+\int_{B_t(0)}|\xi|^2\,\left(2+2|\xi|^2+|\xi|^4\right)\,a^{-2}(\xi)\left(\int^t_0\left|\what f(\tau,\xi)\right|\,d\tau\right)^2d\xi \\
&\leq\,C\,\left(\|r_0\|^2_{H^3}\,+\,\nu^2(t)\,t\,\|r_0\|^4_{H^3}\,+\,\nu^{4}(t)\,\left(1+\nu^{\g}\right)^2\,t^2\right)\,,
\end{align*}
where we have used also the pointwise hypothesis on $\left|\what f(\tau,\xi)\right|$.

Hence, we can insert this bound into \eqref{est:disp-r_3} and integrate the resulting expression in time: after noticing that $\|R(t)\|^2_{L^2}\,\sim\,\left\|\nabla r(t)\right\|^2_{H^2}$,
we find, for a new constant $C$, depending also on $\mu$ and on $\|r_0\|_{H^3}$,
\begin{align}
e^{\mc V(t)}\,\left\|r(t)\right\|^2_{H^3}\,&\leq\,\left\|r_0\right\|^2_{H^3}\,+\,C\,\int^t_0e^{\mc V(\tau)}\,\left(\nu^2(\tau)+\nu^4(\tau)\,\tau\right)d\tau\,+
\label{est:disp-r_4} \\
&\qquad\qquad
+\,C\int^t_0e^{\mc V(\tau)}\bigl(\nu(\tau)\bigr)^{6}\,\left(1+\nu^{\g}\right)^2\,\tau^2\,d\tau\,+\,C\,\int^t_0e^{\mc V(\tau)}\,\|f(\tau)\|_{H^{-1}}\,d\tau\,. \nonumber
\end{align}
for a new constant $C>0$ also depending on $\mu$.

To conclude the proof, we choose the function $\nu^2(t)\,:=\,\alpha/(1+t)$, for some constant $\alpha>0$: then we get $\mc V(t)\,=\,(1+t)^\alpha$. After remarking that
the worst terms are
$$
\int^t_0e^{\mc V(\tau)}\,\nu^2(\tau)\,d\tau\,\leq\,C\,(1+t)^{\alpha-1} 
\qquad\quad\mbox{ and }\qquad\quad \int^t_0e^{\mc V(\tau)}\,\|f(\tau)\|_{H^{-1}}\,d\tau\,\leq\,C\,(1+t)^{\alpha-\eta}\,,
$$
we finally discover that $\left\|r(t)\right\|^2_{H^3}\,\leq\,C\,(1+t)^{-\eta}$, for a suitable positive constant $C$.

The proof of the proposition is then accomplished.
\end{proof}

We have to remark that, in light of Proposition \ref{p:w_decay} (see also Subsection \ref{ss:disp_end} below) the previous decay is not enough for our scopes.
So, we need also the following statement.
\begin{prop} \label{p:r-decay_high}
Under the hypotheses of Proposition \ref{p:r-decay}, suppose moreover that $r_0\,\in\,H^4(\R^2)$ and that $f\,\in\,L^1\bigl(\R_+;L^2(\R^2)\bigr)$ is such that,
for some $\eta_1\in\,]0,1[\,$, 
$$
\|f(t)\|_{L^2}\,\leq\,K\,(1+t)^{-1-\eta_1}\,.
$$

Then there exists a constant $C$ (just depending on $\mu$, $K_1$, $K_2$ of Proposition \ref{p:r-decay} and on $K$ and $\|r_0\|_{H^4}$) such that, for all $t>0$, one has
$$
\left\|r(t)\right\|_{H^4}\,\leq\,\frac{C}{(1+t)^{\eta_1/2}}\,.
$$
\end{prop}

\begin{proof}
The proof of the previous statement follows the main lines of the proof to Proposition \ref{p:r-decay}. So, let us just sketch it.

We start by multiplying equation \eqref{eq:main} by $\Delta^2r$, as in the proof of Proposition \ref{p:a-priori_II}; after passing in Fourier variables, 
thanks to the latter inequality in Lemma \ref{l:r-prelim} we find
\begin{align*}
& \hspace{-1cm}
\frac{1}{2}\,\frac{d}{dt}\int|\xi|^4\,|\what r(t,\xi)|^2\,d\xi\,+\,\frac{1}{2}\,\frac{d}{dt}\int|\what R_1(t,\xi)|^2\,d\xi\,+\,\mu\int|\xi|^2\,\left|\what R_1(t,\xi)\right|^2\,d\xi\,\leq \\
&\qquad\qquad\qquad\qquad\qquad
\leq\,C\,\left\|f(t)\right\|_{L^2}\,\left\|\Delta^2r\right\|_{L^2}\,\leq\,C\,\left\|f(t)\right\|_{L^{2}}\,\left\|r_0\right\|_{H^4}\,,
\end{align*}
where we have defined $\what R_1(t,\xi)\,:=\,\bigl(|\xi|^6+|\xi|^8\bigr)^{1/2}\,\what r(t,\xi)$.

As above, we define $B_t(0)$ to be the ball of center $0$ and radius $\nu(t)$, for some positive function $\nu$ to be defined later on. Therefore we can estimate
\begin{align*}
\int|\xi|^2\,\left|\what R_1(t,\xi)\right|^2\,d\xi\,&\geq\,\nu^2(t)\int_{\R^2}\left|\what R_1(t,\xi)\right|^2\,d\xi\,-\,\nu^2(t)\int_{B_t(0)}\left|\what R_1(t,\xi)\right|^2\,d\xi\,.
\end{align*}
Moreover, if we set $\mc V(t)\,:=\,\mu\int^t_0\nu^2(\tau)\,d\tau$, then we have
$$
\frac{d}{dt}\int|\what R_1(t,\xi)|^2\,d\xi\,+\,\mu\,\nu^2(t)\int\left|\what R_1(t,\xi)\right|^2\,d\xi\,=\,
e^{-\mc V(t)}\,\frac{d}{dt}\left(e^{\mc V(t)}\,\int\left|\what R_1(t,\xi)\right|^2\,d\xi\right)\,,
$$
while Lemma \ref{l:r-prelim} implies that
\begin{align*}
\frac{d}{dt}\int|\xi|^4\,|\what r(t,\xi)|^2\,d\xi\,&\geq\,e^{-\mc V(t)}\frac{d}{dt}\left(e^{\mc V(t)}\int|\xi|^4\,|\what r(t,\xi)|^2\,d\xi\right)\,-\,\mu\,\nu^2(t)\,\|\Delta r(t)\|^2_{L^2}\,.
\end{align*}
Putting these relations together, we deduce the estimate
\begin{align}
& \hspace{-0.5cm}
e^{-\mc V(t)}\,\frac{d}{dt}\left(e^{\mc V(t)}\left(\|\Delta r(t)\|^2_{L^2}\,+\,\left\|R_1(t)\right\|^2_{L^2}\right)\right)\,\leq \label{est:disp-r_5} \\
&\qquad\qquad
\leq\,C\,\left\|f(t)\right\|_{L^{2}}\,\left\|r_0\right\|_{H^4}\,+\,\mu\,\nu^2(t)\,\|\Delta r(t)\|^2_{L^2}\,+\,\mu\,\nu^2(t)\int_{B_t(0)}\left|\what R_1(t,\xi)\right|^2\,d\xi\,. \nonumber
\end{align}

Once again, we resort to Lemma \ref{l:Fr_point} to control the last term in the right-hand side of the previous inequality. After some easy computations, we gather
\begin{align*}
\int_{B_t(0)}\left|\what R_1(t,\xi)\right|^2\,d\xi\,&\leq\,\nu^2(t)\,\|r_0\|^2_{H^3}\,+\,t\,\nu^4(t)\,\|r_0\|_{H^3}^4\,+\,\nu^4(t)\,\bigl(1+\nu^{\g}(t)\bigr)^2\,t^2\,.
\end{align*}
Using this estimate in \eqref{est:disp-r_5}, together with the hypothesis over $f$ and Proposition \ref{p:r-decay} for the term $\|\Delta r(t)\|^2_{L^2}$,
we finally find, after integration in time, for some constant $C$ also depending on the norm of the initial datum,
\begin{align*}
\hspace{-0.5cm}
e^{\mc V(t)}\,\left\|\Delta r(t)\right\|^2_{H^2}\,&\leq\,\|r_0\|_{H^4}^2\,+\,\int^t_0e^{\mc V(\tau)}\,(1+\tau)^{-1-\eta_1}\,d\tau\,+\,\int^t_0e^{\mc V(\tau)}\,\nu^2(\tau)\,(1+\tau)^{-\eta}\,+ \\
&\qquad +\,\int^t_0e^{\mc V(\tau)}\,\nu^4(\tau)\,\left(1\,+\,\tau\,\nu^2(\tau)\right)\,d\tau\,+\,\int^t_0e^{\mc V(\tau)}\,\nu^6(\tau)\,\tau^2\,\bigl(1+\nu^{\g}(\tau)\bigr)^2\,d\tau\,.
\end{align*}
As before, we define $\nu^2(t)\,:=\,\alpha_1/(1+t)$, with $\alpha_1>0$ to be fixed. Then, the worst term is the first integral in the right-hand side:
$$
\int^t_0e^{\mc V(\tau)}\,(1+\tau)^{-1-\eta_1}\,d\tau\,\leq\,C\,(1+t)^{\alpha_1-\eta_1}\,,
$$
while the other terms all present a better bound (namely, a bound in terms of a lower power of $1+t$). The desired inequality easily follows.
\end{proof}

\subsection{Completing the proof} \label{ss:disp_end}

We are now ready to complete the proof of Theorem \ref{th:long-time}. To this end, let us introduce the quantities
$$
z(t,x)\,:=\,r(t,x)\,-\,w(t,x)\qquad\qquad\mbox{ and }\qquad\qquad z_0\,:=\,r_0\,-\,w_0\,\equiv\,0\,.
$$
Then, $z$ satisfies the equation
\begin{equation} \label{eq:delta}
\d_t\bigl(\Id-\Delta+\Delta^2\bigr)z\,+\,\mu\,\Delta^2\bigl(\Id-\Delta\bigr)z\,=\,-\,u\cdot\nabla\Delta^2r\,,
\end{equation}
where $u\,=\,\nabla^\perp(\Id-\Delta)r$, related to the initial datum $z_0\,=\,0$.

First of all, we need the equivalent version of Lemma \ref{l:Fr_point}.
\begin{lemma} \label{l:Fz_point}
Under the hypotheses of Theorem \ref{th:long-time}, $z$ fulfills the pointwise estimate
$$
\left|\what z(t,\xi)\right|\,\leq\,|\xi|\,a^{-1}(\xi)\,\sqrt{t}\,\|r_0\|^2_{H^3}
$$
\end{lemma}

\begin{proof}
As done above, we recast equation \eqref{eq:delta} on the Fourier side: we have
$$
\d_t\what z\,+\,\mu\,h(\xi)\,\what z\,=\,-\,a^{-1}(\xi)\,\mc F\bigl(u\cdot\nabla\Delta^2r\bigr)\,.
$$
Since $\div u\,=\,0$, we deduce
$$
\left|\mc F\bigl(u\cdot\nabla\Delta^2r\bigr)\right|\,\leq\,|\xi|\,\left\|u\;\Delta^2r\right\|_{L^1}\,\leq\,|\xi|\,\left\|u\right\|_{L^2}\;\left\|\Delta^2r\right\|_{L^2}\,\leq\,
C\,\|r\|_{H^3}\,\|\Delta r\|_{H^2}\,.
$$
The desired estimate is thus a striaghtforward consequence of Lemma \ref{l:r-prelim}.
\end{proof}

\begin{rem} \label{r:disp-data}
We exlicitly point out here that, under our hypotheses, we can assume that there exists a constant $C>0$ large enough, so that
$$
\|w(t)\|^2_{H^3}\,+\,\mu\int^t_0\|\Delta w(\tau)\|_{H^2}^2\,d\tau\,+\,\|r(t)\|^2_{H^3}\,+\,\mu\int^t_0\|\Delta r(\tau)\|_{H^2}^2\,d\tau\,\leq\,C
$$
for all times $t\geq0$. This is a consequence of Lemma \ref{l:r-prelim} and Remarks \ref{r:parab-unif} and \ref{r:r-prelim}.
\end{rem}

We are now in the position of proving Theorem \ref{th:long-time}. We argue in a similar way as in Proposition \ref{p:r-decay}: so, let us perform energy estimates of first type
for equation \eqref{eq:delta}. We observe that
$$
\int_{\R^2}\nabla^\perp(\Id-\Delta)r\cdot\nabla(\Id-\Delta)z\,\Delta^2r\,dx\,=\,\int_{\R^2}\nabla^\perp(\Id-\Delta)r\cdot\nabla(\Id-\Delta)w\,\Delta^2r\,dx\,;
$$
hence, after passing in Fourier variables, the same computations as before yield
\begin{align}
& \hspace{-0.5cm}
\frac{d}{dt}\int|\what z(t,\xi)|^2\,d\xi\,+\,\frac{d}{dt}\int|\what Z(t,\xi)|^2\,d\xi\,+\,\mu\int|\xi|^2\,\left|\what Z(t,\xi)\right|^2\,d\xi\,\leq 
\label{est:disp-z_1} \\
&\qquad\qquad\qquad\qquad\qquad
\leq\,C\,\left\|\nabla(\Id-\Delta)r(t)\right\|_{L^2}\,\left\|\Delta^2r(t)\right\|_{L^2}\,\left\|\nabla(\Id-\Delta)w(t)\right\|_{L^2}\,, \nonumber
\end{align}
where we have defined $\what Z(t,\xi)\,:=\,|\xi|\left(2\,+\,2\,|\xi|^2\,+\,|\xi|^4\right)^{1/2}\,\what z(t,\xi)$ as above.

Next, for $t\geq0$, we introduce the ball $B_t(0)\subset\R^2$ as the ball centered in $0$ and of radius $\nu(t)$, for some function $\nu(t)$ to be determined later. Then we can write
\begin{align*}
\int|\xi|^2\,\left|\what Z(t,\xi)\right|^2\,d\xi\,&\geq\,\nu^2(t)\int_{\R^2}\left|\what Z(t,\xi)\right|^2\,d\xi\,-\,\nu^2(t)\int_{B_t(0)}\left|\what Z(t,\xi)\right|^2\,d\xi\,,
\end{align*}
and plugging this inequality into \eqref{est:disp-z_1} we find, after setting $\mc V(t)\,:=\,\mu\int^t_0\nu^2(\tau)\,d\tau$,
\begin{align}
&\frac{d}{dt}\int|\what z(t,\xi)|^2\,d\xi\,+\,
e^{-\mc V(t)}\,\frac{d}{dt}\left(e^{\mc V(t)}\,\int\left|\what Z(t,\xi)\right|^2\,d\xi\right)
\label{est:disp-z_2} \\
&\qquad \leq\,C\,\left\|\nabla(\Id-\Delta)r(t)\right\|_{L^2}\,\left\|\Delta^2r(t)\right\|_{L^2}\,\left\|\nabla(\Id-\Delta)w(t)\right\|_{L^2}\,+\,
\mu\,\nu^2(t)\int_{B_t(0)}\left|\what Z(t,\xi)\right|^2\,d\xi\,. \nonumber
\end{align}

Repeating the computations performed for $\what r$ in the proof of Proposition \ref{p:r-decay}, and using Remark \ref{r:disp-data} above, we can bound
\begin{align*}
\frac{d}{dt}\int|\what z(t,\xi)|^2\,d\xi\,&=\,e^{-\mc V(t)}\frac{d}{dt}\left(e^{\mc V(t)}\int|\what z(t,\xi)|^2\,d\xi\right)\,-\,\mu\,\nu^2(t)\int|\what z(t,\xi)|^2\,d\xi \\
&\geq\,e^{-\mc V(t)}\frac{d}{dt}\left(e^{\mc V(t)}\int|\what z(t,\xi)|^2\,d\xi\right)\,-\,C\,\mu\,\nu^2(t)\,.
\end{align*}
On the other hand, the last term in the right-hand side of \eqref{est:disp-z_2} can be controlled in view of Lemma \ref{l:Fz_point}: denoting by $C$ a positive constant
possibly depending also on the norm of the initial datum $r_0$, we gather
\begin{align*}
\int_{B_t(0)}\left|\what Z(t,\xi)\right|^2\,&=\,\int_{B_t(0)}|\xi|^2\left(2\,+\,2\,|\xi|^2\,+\,|\xi|^4\right)\,\left|\what r(t,\xi)\right|^2\,d\xi \\
&\leq\,C\int_{B_t(0)}|\xi|^4\,\left(2+2|\xi|^2+|\xi|^4\right)\,a^{-2}(\xi)\,t\,\|r_0\|^4_{H^3}\,d\xi\;\leq\;\,C\,\nu^2(t)\,t\,.
\end{align*}
Therefore, plugging these inequalities into \eqref{est:disp-z_2} leads us to
\begin{align}
&e^{-\mc V(t)}\frac{d}{dt}\left(e^{\mc V(t)}\,\left(\left\|z(t)\right\|^2_{L^2}\,+\,\left\|Z(t)\right\|^2_{L^2}\right)\right)\,\leq\
\label{est:disp-z_3}  \\
&\qquad\qquad
\leq\,C\,\Bigl(\left\|\nabla(\Id-\Delta)r(t)\right\|_{L^2}\,\left\|\Delta^2r(t)\right\|_{L^2}\,\left\|\nabla(\Id-\Delta)w(t)\right\|_{L^2}\,+\,\nu^4(t)\,t\,+\,\nu^2(t)\Bigr)\,. \nonumber
\end{align}
for a new constant $C>0$ also depending on $\mu$.

At this point, we use estimates provided by Propositions \ref{p:w_decay}, \ref{p:r-decay} and \ref{p:r-decay_high} in order to bound the first term in the right-hand side
of the previous inequality. Since $\|Z(t)\|^2_{L^2}\,\sim\,\left\|\nabla z(t)\right\|^2_{H^2}$, an integration of \eqref{est:disp-z_3} in time yields
\begin{align}
e^{\mc V(t)}\,\left\|z(t)\right\|^2_{H^3}\,&\leq\,C\,\int^t_0e^{\mc V(\tau)}\,(1+\tau)^{-(\eta+1/2)}\,d\tau\,+\,
C\int^t_0e^{\mc V(\tau)}\,\bigl(\nu^4(\tau)\,\tau\,+\,\nu^2(\tau)\bigr)\,d\tau\,.
\label{est:disp-z_4}
\end{align}

To conclude the proof, we choose as before the function $\nu^2(t)\,:=\,\beta/(1+t)$, for some constant $\beta>0$, so that $\mc V(t)\,=\,(1+t)^\beta$. First of all, we notice
that
$$
\int^t_0e^{\mc V(\tau)}\,\bigl(\nu^4(\tau)\,\tau\,+\,\nu^2(\tau)\bigr)\,d\tau\,\leq\,C\,\left((1+t)^{\beta-2}\,+\,(1+t)^{\beta-1}\right)\,\leq\,C\,(1+t)^{\beta-1}\,.
$$
Then, we treat the other integral: we have
$$
\int^t_0e^{\mc V(\tau)}\,(1+\tau)^{-(\eta+1/2)}\,d\tau\,\leq\,C\,(1+t)^{\beta+1/2-\eta}\,,
$$
and this inequality completes the proof of Theorem \ref{th:long-time}.

\appendix

\section{Fourier analysis toolbox} \label{app:LP}

We recall here the main ideas of Littlewood-Paley theory, which we exploited in the previous analysis.
We refer e.g. to Chapter 2 of \cite{B-C-D} for details.

For simplicity of exposition, let us deal with the $\R^d$ case; however, the whole construction can be adapted also to the $d$-dimensional torus $\T^d$.

\medbreak
First of all, let us introduce the so called ``Littlewood-Paley decomposition'', based on a non-homogeneous dyadic partition of unity with
respect to the Fourier variable. 

We, fix a smooth radial function $\chi$ supported in the ball $B(0,2)$, equal to $1$ in a neighborhood of $B(0,1)$
and such that $r\mapsto\chi(r\,e)$ is nonincreasing over $\R_+$ for all unitary vectors $e\in\R^d$. Set
$\varphi\left(\xi\right)=\chi\left(\xi\right)-\chi\left(2\xi\right)$ and
$\vphi_j(\xi):=\vphi(2^{-j}\xi)$ for all $j\geq0$.

The dyadic blocks $(\Delta_j)_{j\in\Z}$ are defined by\footnote{Throughout we agree  that  $f(D)$ stands for 
the pseudo-differential operator $u\mapsto\mc{F}^{-1}(f\,\mc{F}u)$.} 
$$
\Delta_j\,:=\,0\quad\mbox{ if }\; j\leq-2,\qquad\Delta_{-1}\,:=\,\chi(D)\qquad\mbox{ and }\qquad
\Delta_j\,:=\,\varphi(2^{-j}D)\quad \mbox{ if }\;  j\geq0\,.
$$
We  also introduce the following low frequency cut-off operator:
\begin{equation} \label{eq:S_j}
S_ju\,:=\,\chi(2^{-j}D)\,=\,\sum_{k\leq j-1}\Delta_{k}\qquad\mbox{ for }\qquad j\geq0\,.
\end{equation}
By Remark 2.11 of \cite{B-C-D}, the operators $S_j$ and $\Delta_j$ map $L^p$ into itself, for all $j\geq-1$ and all $p\in[1,+\infty]$, with norms independent of $j$ and $p$.

The following classical property holds true: for any $u\in\mc{S}'$, then one has the equality~$u=\sum_{j}\Delta_ju$ in the sense of $\mc{S}'$.
Let us also mention the so-called \emph{Bernstein's inequalities}, which explain the way derivatives act on spectrally localized functions.
  \begin{lemma} \label{l:bern}
Let  $0<r<R$.   A constant $C$ exists so that, for any nonnegative integer $k$, any couple $(p,q)$ 
in $[1,+\infty]^2$, with  $p\leq q$,  and any function $u\in L^p$,  we  have, for all $\lambda>0$,
$$
\displaylines{
{\rm supp}\, \widehat u \subset   B(0,\lambda R)\quad
\Longrightarrow\quad
\|\nabla^k u\|_{L^q}\, \leq\,
 C^{k+1}\,\lambda^{k+d\left(\frac{1}{p}-\frac{1}{q}\right)}\,\|u\|_{L^p}\;;\cr
{\rm supp}\, \widehat u \subset \{\xi\in\R^d\,|\, r\lambda\leq|\xi|\leq R\lambda\}
\quad\Longrightarrow\quad C^{-k-1}\,\lambda^k\|u\|_{L^p}\,
\leq\,
\|\nabla^k u\|_{L^p}\,
\leq\,
C^{k+1} \, \lambda^k\|u\|_{L^p}\,.
}$$
\end{lemma}   

By use of Littlewood-Paley decomposition, we can define the class of Besov spaces.
\begin{defin} \label{d:B}
  Let $s\in\R$ and $1\leq p,r\leq+\infty$. The \emph{non-homogeneous Besov space}
$B^{s}_{p,r}$ is defined as the subset of tempered distributions $u$ for which
$$
\|u\|_{B^{s}_{p,r}}\,:=\,
\left\|\left(2^{js}\,\|\Delta_ju\|_{L^p}\right)_{j\in\N}\right\|_{\ell^r}\,<\,+\infty\,.
$$
\end{defin}

Besov spaces are interpolation spaces between the Sobolev ones. In fact, for any $k\in\N$ and $p\in[1,+\infty]$
we have the following chain of continuous embeddings:
$$
 B^k_{p,1}\hookrightarrow W^{k,p}\hookrightarrow B^k_{p,\infty}\,,
$$
where  $W^{k,p}$ denotes the classical Sobolev space of $L^p$ functions with all the derivatives up to the order $k$ in $L^p$.
Moreover, for all $s\in\R$ we have the isomorphism of Banach spaces $B^s_{2,2}\cong H^s$, with
\begin{equation} \label{eq:LP-Sob}
\|f\|_{H^s}\,\sim\,\left(\sum_{j\geq-1}2^{2 j s}\,\|\Delta_jf\|^2_{L^2}\right)^{1/2}\,.
\end{equation}
More in general, the previous isomorphism is an isomorphism of Hilbert spaces. As a matter of fact, if we define the $B^s_{2,2}$ scalar product by the formula,
\begin{equation*}
\lan\langle f,\, g\ran\rangle_{B_{2,2}^{\s}} := \sum_{q\geq -1} 2^{2\,q\,\s}\langle \Delta_q f\,,\, \Delta_q g\rangle_{L^2}\,,
\end{equation*}
by Proposition 2.10 of \cite{B-C-D} we get a scalar product which is equivalent to the classical one over $H^s$.

As an immediate consequence of the first Bernstein's inequality, one gets the following embedding result.
\begin{prop}\label{p:embed}
The space $B^{s_1}_{p_1,r_1}$ is continuously embedded in the space $B^{s_2}_{p_2,r_2}$ for all indices satisfying $p_1\,\leq\,p_2$ and
$$
s_2\,<\,s_1-d\left(\frac{1}{p_1}-\frac{1}{p_2}\right)\qquad\mbox{ or }\qquad
s_2\,=\,s_1-d\left(\frac{1}{p_1}-\frac{1}{p_2}\right)\;\;\mbox{ and }\;\;r_1\,\leq\,r_2\,. 
$$
\end{prop}

We recall also Lemma 2.73 of \cite{B-C-D}.
\begin{lemma} \label{l:Id-S}
If $1\leq r<+\infty$, for any $f\in B^s_{p,r}$ one has
$$
\lim_{j\ra+\infty}\left\|f\,-\,S_jf\right\|_{B^s_{p,r}}\,=\,0\,.
$$
\end{lemma}

Let us now introduce the paraproduct operator (after J.-M. Bony, see \cite{Bony}). Constructing the paraproduct operator relies on the observation that, 
formally, any product  of two tempered distributions $u$ and $v,$ may be decomposed into 
\begin{equation}\label{eq:bony}
u\,v\;=\;T_uv\,+\,T_vu\,+\,R(u,v)\,,
\end{equation}
where we have defined
\begin{equation} \label{eq:T-R}
T_uv\,:=\,\sum_jS_{j-1}u\Delta_j v,\qquad\qquad\mbox{ and }\qquad\qquad
R(u,v)\,:=\,\sum_j\sum_{|j'-j|\leq1}\Delta_j u\,\Delta_{j'}v\,.
\end{equation}
The above operator $T$ is called ``paraproduct'' whereas $R$ is called ``remainder''.
We recall that, for all $u$ and $v$ in $\mc S'$, the sequence $\bigl(S_{j-1}u\,\Delta_jv\bigr)_{j\in\N}$ is spectrally supported in dyadic annuli, whose radius is proportional to
$2^j$.

The paraproduct and remainder operators have many nice continuity properties. 
The following ones have been of constant use in this paper (see the proof in e.g. Chapter 2 of \cite{B-C-D}).
\begin{prop}\label{p:op}
For any $(s,p,r)\in\R\times[1,\infty]^2$ and $t>0$, the paraproduct operator 
$T$ maps continuously $L^\infty\times B^s_{p,r}$ in $B^s_{p,r}$ and  $B^{-t}_{\infty,\infty}\times B^s_{p,r}$ in $B^{s-t}_{p,r}$.
Moreover, the following estimates hold:
$$
\|T_uv\|_{B^s_{p,r}}\,\leq\, C\,\|u\|_{L^\infty}\,\|\nabla v\|_{B^{s-1}_{p,r}}\qquad\mbox{ and }\qquad
\|T_uv\|_{B^{s-t}_{p,r}}\,\leq\, C\|u\|_{B^{-t}_{\infty,\infty}}\,\|\nabla v\|_{B^{s-1}_{p,r}}\,.
$$

For any $(s_1,p_1,r_1)$ and $(s_2,p_2,r_2)$ in $\R\times[1,\infty]^2$ such that 
$s_1+s_2>0$, $1/p:=1/p_1+1/p_2\leq1$ and $1/r:=1/r_1+1/r_2\leq1$,
the remainder operator $R$ maps continuously $B^{s_1}_{p_1,r_1}\times B^{s_2}_{p_2,r_2}$ into $B^{s_1+s_2}_{p,r}$.
In the case $s_1+s_2=0$, provided $r=1$, operator $R$ is continuous from $B^{s_1}_{p_1,r_1}\times B^{s_2}_{p_2,r_2}$ with values
in the space $B^{0}_{p,\infty}$.
\end{prop}

We recall also a classical commutator estimate (see e.g. Lemma 2.97 of \cite{B-C-D}).
\begin{lemma} \label{l:commut}
Let $\Phi\in\mc{C}^1(\R^d)$ such that $\bigl(1+|\,\cdot\,|\bigr)\what{\Phi}\,\in\,L^1$. There exists a constant $C$ such that,
for any function $h$ for which $\nabla h\in L^{p}(\R^d)$, for any $f\in L^q(\R^d)$ and for all $\lambda>0$, one has
$$
\left\|\bigl[\Phi(\lambda^{-1}D),h\bigr]f\right\|_{L^r}\,\leq\,C\,\lambda^{-1}\,\left\|\nabla h\right\|_{L^p}\,\|f\|_{L^q}\,,
$$
where $r\in[1,+\infty]$ satisfies the relation $1/r\,=\,1/p\,+\,1/q$.
\end{lemma}
Going along the lines of the proof, it is easy to see that the constant $C$ depends just on the $L^1$ norm
of the function $|x|\,k(x)$, where $k\,=\,\mc{F}_\xi^{-1}\Phi$ is the inverse Fourier transform of $\Phi$.

To conclude, let us quote a compactness result (see Theorem 2.94 of \cite{B-C-D}).
\begin{thm} \label{t:comp}
For any $s'<s$, for all smooth $\phi$ in the Schwartz class $\mc{S}(\R^d)$, the multiplication by $\phi$ is a compact operator from
$B^s_{p,\infty}$ into $B^{s'}_{p,1}$.
\end{thm}


\end{document}